\documentclass[11pt,a4paper]{article}
\usepackage[utf8]{inputenc}
\usepackage[T1]{fontenc}
\usepackage{amsmath}
\usepackage{soul}
\usepackage{amsfonts}
\usepackage{amssymb}
\usepackage{amsthm}
\usepackage{graphicx} 
\usepackage{caption}
\usepackage{subcaption}
\usepackage[left=2.5cm,right=2.5cm,top=2cm,bottom=2.5cm]{geometry}
\usepackage[usenames,dvipsnames]{xcolor}
\usepackage{tikz}
\usepackage[normalem]{ulem}
\usepackage{comment}
\usepackage{array}
\usepackage{multirow}
\usepackage{algorithm2e}
\usepackage{hyperref}
\usepackage{dsfont}

\newtheorem{remark}{Remark}
\newtheorem{lemma}{Lemma}[section]
\newtheorem{proposition}{Proposition}[section]

\newtheorem{theorem}{Theorem}

\newtheorem{corollary}{Corollary}[section]
\newtheorem{assumption}{Assumption}

\newcommand{\R}{\mathbb{R}}

\providecommand{\keywords}[1]
{
  \small	
  \textbf{\textit{Keywords---}} #1
}

\newcommand{\nset}{\mathbb{N}}
\newcommand{\rset}{\mathbb{R}}

\newcommand{\Sb}{\mathbf{S}}

\newcommand{\va}{\text{V@R}\ensuremath{_\alpha}}

\newcommand{\cva}{\text{CV@R}\ensuremath{_\alpha}}
\newcommand{\X}{X}

\newcommand{\hZ}{\ensuremath{\hat{Z}}}
\newcommand{\ee}{\ensuremath{\mathfrak{e}}}
\newcommand{\hg}{\ensuremath{\hat{g}}}

\newcommand{\un}{\ensuremath{\mathds{1}}}
\newcommand{\PP}{\mathbb{P}} % proba
\newcommand{\PE}{\mathbb{E}} % esperance
\newcommand{\usl}{u^\star_{\lambda}} % hedge optimal
\newcommand{\Sm}{\mathcal{S}_m} % proba

\newcommand{\var}{\theta}

\newcommand{\Dp}{\mathcal{D}_{\varphi}}
\newcommand{\DP}{\mathcal{D}_{\Phi}}
\newcommand{\norm}[2]{\big\| #1 \big\|_{#2}}
\newcommand{\Db}{{\Delta}}

\begin{document}
\title{{CV@R  penalized portfolio optimization} with {biased} stochastic mirror descent}
\author{M. Costa\footnote{Institut de Math\'ematiques de Toulouse. CNRS UMR 5219. Universit\'e Paul Sabatier, 118 route de Narbonne, F-31062 Toulouse cedex 09. Email: manon.costa@math.univ-toulouse.fr}, S. Gadat\footnote{Toulouse School of Economics, CNRS UMR 5314, Université Toulouse 1 Capitole, Esplanade de l’Université, Toulouse, France. Institut Universitaire de France. Email: sebastien.gadat@tse-fr.eu}, L. Huang\footnote{INSA de Toulouse, IMT UMR CNRS 5219, Université de Toulouse, 135 avenue de Rangueil 31077 Toulouse Cedex 4 France. Email: lorick.huang@insa-toulouse.fr.}}
\maketitle

\maketitle

\begin{abstract}
This article studies and solves the problem of optimal portfolio allocation with {CV@R penalty}  when dealing with imperfectly simulated financial assets.
We use a  Stochastic biased Mirror Descent to find optimal resource allocation for a portfolio whose underlying assets cannot be generated exactly and may only be approximated with a numerical scheme that satisfies suitable error bounds, under a risk management constraint.
We establish almost sure asymptotic properties as well as the rate of convergence for the averaged algorithm. We then focus on the optimal tuning of the overall procedure to obtain an optimized numerical cost. Our results are then illustrated numerically on simulated as well as real data sets.
%Insert your abstract here. Include keywords,  mathematical
%subject classification and JEL numbers as needed.

%\subclass{65C20\and60E05\and37N40}
%\JEL{C61 \and C63 and C68}

\end{abstract}
\section{Introduction}
\label{intro}
Given a portfolio of financial assets, what is the best way to allocate resources under a risk management constraint? 
Without any fixed constraint, it is likely that the asset with the highest expected return would be favored, leading to a trivial optimization problem. 
However, ethical issues arise when this latter return possesses high variability, since even though in the long run some profits are expected, large losses can occur in between wins.
This problematic is especially relevant for institutions with some large capital investments such as banks and insurances.

Hence, the necessity for regulation appears. As a matter of fact, this is precisely the point of the Basel III banking regulation that has been implemented in response to the various past financial crises.
Risks can be quantified with various metrics, and policies developed using one or another metric could lead to different investment optimization problems and strategies. 
In this work, we consider a { penalty} on the conditional value at risk denoted by CV@R (also referred to as the expected shortfall in the literature) of the portfolio, which stands for a coherent measure of risk of an investment \cite{Artzner99,Rockafellar2002}. We refer to Equation \eqref{def:cvar} in the next section for a formal definition. 
Heuristically, the CV@R measures how big the average loss is, at a given rate for worst scenarii (say $10\%$ for example). Our problem is therefore to optimize the expected portfolio returns while keeping the risk (measured via CV@R) under a fixed threshold. 

Portfolio optimization under CV@R constraint has already been investigated in previous works, we can notably cite Rockafellar and Uryasev \cite{Rockafellar2000} who develop a Monte-Carlo strategy to approximate the CV@R of the portfolio. These results have been refined in subsequent works of the authors, notably in Krokhmal et al. \cite{Krokhmal:2001}, Rockafellar et Uryasev \cite{Rockafellar2002} and more recently in \cite{rockafellar:royset:2014} , but basically the authors of these former works use the approximation of the V@R and CV@R with a batch empirical sum to feed an optimization algorithm via classical tools of convex analysis.

The approximation with a Monte-Carlo empirical mean to handle the V@R and the CV@R can be traced back to Ben Tal and Teboulle \cite{Teboulle:1986} and Pluf \cite{Pflug}, who link these two quantities by a $\min$ and $\arg\min$ of a convex function.
This article spun a series of articles dealing with the Robbins-Monro approximation of V@R and CV@R. We can cite a series of papers by Bardou, Frikha and Pagès \cite{bardou2016cvar},  \cite{bardou2009computing},  and the book \cite{bardou:etal:book}, as well as \cite{frikha:huang:2015} all dealing with Robbins Monro stochastic approximation. However, as in \cite{frikha:huang:2015}, a question still remains when the underlying asset's distribution is unobtainable. 
A natural solution to that problem would be to approximate the underlying distribution, introducing a bias.
This trend of biased stochastic approximation is very natural but presents some challenges in itself, which can be addressed in multiple ways. 
We can cite the works of \cite{atchade:fort:moulines:2017} or \cite{laruelle:pages:2012} (see also the references therein) for an overview of the techniques used for biased stochastic approximation algorithms.

In mathematical finance, assets are commonly modeled using stochastic differential equations, and in general, these equations do not have any explicit solutions. A natural approach to this issue is simply to replace the SDE with a discretization scheme. However, a such discretization naturally introduces a bias in the stochastic approximation procedure, and the effect of this bias should be addressed to finely tune the optimization algorithm. Discretization for SDEs has been studied since the pioneering work of Talay and Tubaro \cite{talay:tubaro:1990}, where strong regularity is assumed on the coefficients. As results developed, regularity assumptions have been dropped, and sharp bounds have been derived for both weak and strong error (see \textit{e.g.} \cite{kloeden:platen:1999}) on specific stochastic models for specific discretization schemes (see \textit{e.g.} \cite{alfonsi:2005,alfonsi:2010} among others).

In our paper, we assume in addition that one of the asset in our portfolio is risk-less, which allows to model debt obligations, or treasury bonds. In mathematical terms, we consider a stochastic rate, modeled by a Cox-Ingersoll-Ross {(shortened as CIR below)} \cite{CIR} process and we use this rate for the growth ratio of one component of our portfolio (see Subsection \ref{portfolio:dynamics} below the definition).
It can be shown that CIR processes can be calibrated to always stay positive (which is desirable for modeling rates - but sometimes not necessary). Unfortunately, positivity is not preserved with the usual explicit Euler scheme. This fact drives us, when discretizing the CIR, to use a drift-implicit Euler scheme, introduced in \cite{alfonsi:2005}, that is known to preserve positivity (see Section \ref{subsec:discretization} and Proposition \ref{alfonsi:discr:CIR} below for more details). Thankfully, rate of convergence for the approximation of the CIR process has been widely studied, see \textit{e.g.} \cite{Dereich_neuenkirch2012,alfonsi:2013} and these quantitative results may be exploited to control the bias and its propagation in the specific terms induced by some functional related to the CV@R variational formulation.

Finally, we also impose a reinvestment condition, saying that the strategy must reallocate all funds available. 
This means that if we have an initial capital of $V$ and $n$ elements in our portfolio, we are looking for a vector $(v_1, \dots, v_n)$ such that $v_1+ \cdots + v_n= V$.
Equivalently, we consider $u=(u_1,\dots ,u_n)$ in the simplex: $u_i \in [0,1]$ and $u_1 + \dots + u_n=1$ and look for an optimal set of weights $u$.
Note that this is asking each weights to be positive. This prevents short selling, which is the practice of selling a stock one does not own in hopes of buying it back in the future at a lower price, and is yet another behavior that is regulated by   financial institutions. 

In order to determine a strategy $u=(u_1,\dots ,u_n)$ in the simplex, one natural idea would be to derive an algorithm without the constraint $u_1 + \dots + u_n=1$, and project back the result over the simplex.
However, projection over the simplex tends to favor extremal points, which slows down the convergence of algorithms while projection is itself an additional step that may be numerically costly when some high-dimensional portfolio are considered. To overcome this issue, we set up a stochastic mirror descent, as pioneered by Nemirovskii \& Yudin \cite{Nemirovski:1983}. 
Mirror descent (and its variants) stands for a widely used class of optimization algorithms, whose main interest is to avoid a costly projection into the set of constraints by replacing the Euclidean distance by a Bregman divergence in the second order Taylor expansions. It may be adapted to the stochastic framework with the help of carefully tuned step-size sequences, we can cite among others \cite{nesterov:2009} that initiates the use of a noisy stochastic gradient in the mirror descent,  
 \cite{Lan:2012} for specific derivations on convex optimization problems on the simplex 
 and more recently Zhou et al \cite{zhou:etal:2017} that derives some results on the a.s. convergence in the general case with asymptotic pseudo-trajectories, that are common tools of random dynamical systems (see \textit{e.g.} \cite{benaim1996asymptotic}).

In summary, to find an investment strategy while managing risk, we will use a stochastic mirror descent with an approximation of the portfolio, combined with a fine
 approximation for the CV@R. The discretization of the Portfolio is handled by an implicit Euler scheme.
This discretization feeds the approximation of the V@R and the CV@R, which is done through a biased Robbins-Monro algorithm, which in turn feeds a stochastic mirror descent to determine the optimal portfolio allocation. Even if every individual steps have a known (sometimes involved) solution, we show that plugging-in these solutions together yields a strategy that solves the constrained optimization problem. {Our main results are as follows:
\begin{itemize}
\item First, using some sophisticated tools of stochastic approximation (ODE method) and of mirror descent (Fenchel conjugate among others), we establish the convergence of our biased procedure towards the minimum of our penalized criterion.
\item Second, we establish some new sharp probabilistic upper bounds dedicated to the simulation of the CIR process that are necessary to assess the influence of the number of assets $m$ on our stochastic algorithm. These bounds are at the cornerstone of our work and of crucial importance for the fine tuning of the numerical implicit Euler scheme.
\end{itemize}}

 In what follows, we carefully study the influence of the number of considered assets and of the discretization step-size on the numerical cost of the overall procedure. We then use our method to construct a Markowitz portfolio envelope \cite{markowitz:1952} that crosses the risk (built with the CV@R) and the expected return to identify the optimal investment strategy.

This article is organized as follows. In Section \ref{sec:cvar_SMD} we detail the optimization problem and the stochastic mirror descent that we propose. We state the almost sure convergence result as well as the non asymptotic bounds that are proved in Appendix \ref{app:proofs_smd}. Then in Section \ref{sec:portfolio} we describe the structure of a portfolio of interest and the simulation strategy. In particular we highlight the link between the step size of the discretization and the bias that will be induced in the SMD algorithm. Finally, we illustrate our results in Section \ref{sec:num} both on simulated and on real data.

\section{\cva constrained mean value optimization}
\label{sec:cvar_SMD}

\subsection{Financial model and constrained optimization}

We consider the relative return of a portfolio, described by a random vector $Z=(Z_1,\dots, Z_m)$, that is composed of $m$ assets whose capital at the initialization time $t_0=0$ is scaled to $1$. We address the question of the optimal investment strategy in order to maximize the mean relative return at a fixed time horizon $t_0+T$ where $T=1$,  \textit{i.e.} each random variable $Z_i$ corresponds to $Z_i=\frac{A_i(T)}{A_i(0)}-1$ where $A_i(t)$ is the price of asset $i$ at time $t$.
\\
An investment strategy  corresponds to an allocation of the initial capital, modeled by a collection of $m$ positive weights $u=(u_1, \dots, u_m)$ that belongs to the $m-1$ dimensional simplex denoted by $\Db_m$ and defined as:
$$
\Db_m := \Big\{ u \in \rset_+^m \, : \, \sum_{i=1}^m u_i = 1\Big\}.
$$
We are interested in a constrained optimization of the mean of the random variable $\langle Z,u\rangle$ defined by:
\begin{equation*}%\label{def:Zu}
\langle Z,u\rangle = \sum_{i=1}^m u_i Z_i.
\end{equation*}
The constraint we are using involves a risk measure, named the conditional value at risk, which quantifies the mean value of a loss given that a loss has occurred. More precisely, if $\alpha$ refers to a chosen risk level, $\va(u)$ corresponds to the classical statistical $\alpha-$quantile defined as:
$$
\va(u) := \sup \{ q \in \rset \, : \, \mathbb{P}[\langle Z,u\rangle\le q] \leq \alpha\},
$$
while $\cva$, that defines the active constraint we will use, is the mean value of the \textit{loss} when $\langle Z,u\rangle$ is below $\va$, namely:
\begin{equation}\label{def:cvar}
\cva(u) = \mathbb{E}[-\langle Z,u\rangle\, \vert \, \langle Z,u\rangle \le \va(u)].
\end{equation}
Note that here, we implicitly assume that the risk level $\alpha$ is chosen such that the associated quantile $\va(u)$ is negative, which implies that $\cva$ is a positive quantity.

Then, for any fixed positive level $M$, the optimization problem of $u \longmapsto \mathbb{E}[\langle Z,u\rangle]$  with \cva \, constraints we are interested in, is the problem defined by:

\begin{align*}
\mathcal{P}_{M} &:= \arg\displaystyle\max_{u \in \Db_m} \Big\{ \sum_{i=1}^m u_i \mathbb{E}[Z_i] \, : \, \cva(u) \leq M \Big\} \nonumber\\
& = \arg\displaystyle\min_{u \in \Db_m} \Big\{ - \sum_{i=1}^m u_i \mathbb{E}[Z_i] \, : \, \cva(u) \leq M \Big\} 
%\label{def:P1}.
\end{align*}

%\subsection{Primal problem and lagrangian formulation}
For our purpose, we also introduce the unconstrained penalized optimization problem: we consider for any $\lambda>0$ the solution $v_{\lambda}$ of the following (convex) optimization problem:
\begin{align*}%\label{def:P2}
\mathcal{Q}_{\lambda} := \arg\displaystyle\min_{u \in \Db_m} \Big\{ -\sum_{i=1}^m u_i\PE[Z_i] + \lambda \cva(u) \Big\}.
\end{align*}
The (standard) following result is key for our forthcoming analysis.
\begin{proposition}\label{prop:equivalence}
The collection of the two convex problems $(\mathcal{P}_M)_{M >0}$ and $(\mathcal{Q}_\lambda)_{\lambda>0}$ are equivalent. More precisely, for any $M>0$ that defines a feasible constraint, a solution $u_M^\star$ exists such that:
$$\exists \lambda_M^\star >0 \quad u^\star_M = \arg \min_{u \in \Db_m} \Big\{ - \sum_{i=1}^m u_i \mathbb{E}[Z_i]+\lambda_M^\star \cva(u)\Big\}.$$
Moreover, $\lambda_M^\star$ is a decreasing function of $M$.
Oppositely, any solution $v_\lambda$ of $\mathcal{Q}_{\lambda}$ solves $\mathcal{P}_{M}$ with $M=\cva(v_\lambda)$.
\end{proposition}
{
We emphasize that this latter result is the keystone of our forthcoming analysis: instead of solving directly $\mathcal{P}_M$, we will instead define a collection of penalized problems 
$\mathcal{Q}_{\lambda}$ for several values of $\lambda$. Proposition \ref{prop:equivalence} then establishes that this approach leads to optimal portfolio allocation under CV@R constraints. We also highlight that the relationship between $M$ and $\lambda^\star_M$ remains unknown and challenging to obtain.}
The proof of this Lagrangian formulation is deferred to Appendix \ref{sec:appendix_lagrange}.
Using the result of \cite{Rockafellar2000,Krokhmal:2001} and in particular the convex representation of the $\cva$, we observe  that:
$$
\cva(u)
 = \displaystyle\min_{\theta \in \mathbb{R}} \psi_\alpha(u,\theta), 
$$
where $\psi$ is the convex coercive Lipschitz continuous and differentiable function defined by:
\begin{equation*}
%\label{def:psi}
\psi_\alpha(u,\theta)=\theta + \frac{1}{1-\alpha} \mathbb{E}\big[(\langle Z,u\rangle - \theta )^{+}\big], 
\end{equation*}
where $\ x ^+ = \max(0, x)$.
We emphasize that despite the definition with $()^+$, when $Z$ has an absolutely continuous distribution with respect to the Lebesgue measure, $\psi_{\alpha}$ is differentiable (we refer to \cite{bardou2016cvar} for further details on this remark).
\\
Then, we deduce from Proposition \ref{prop:equivalence} and the above representation that the collection of the optimization problems $(\mathcal{P}_M)_{M>0}$ are then described equivalently   by the convex unconstrained problem:
\begin{equation}\label{def:P22}
\mathcal{Q}_{\lambda} = \arg\displaystyle\min_{ (u,\theta) \in \Db_m \times \rset}\{ p_\lambda(u,\theta)\},
\end{equation}
where  the key function $p_\lambda$ is defined by:
\begin{equation*}
%\label{def:p_lambda}
p_\lambda  (u,\theta)= - \sum_{i=1}^m u_i \mathbb{E}[Z_i] + \lambda \psi_\alpha(\theta,u).
\end{equation*}
 We emphasize that $(p_{\lambda})_{\lambda > 0}$ forms a collection of convex functions defined on $\Db_m \times \rset$, coercive w.r.t. $\theta$ and that $\mathcal{Q}_{\lambda}$ enables to avoid the \cva constraint with a suitable reparametrization of the problem.

 \subsection{Biased stochastic Mirror Descent}

To solve the minimization problem \eqref{def:P22} we are led to use stochastic approximation theory, that originates in the pioneering works of \cite{RobbinsMonro1951} and \cite{RobbinsSiegmund1971} since the (convex) function $p_\lambda$  is written as an expectation. 

However, we encounter here specific difficulties. First we need to handle the minimization over the simplex $\Db_m$. Second, the random variables $Z$ involved in $p_\lambda$ cannot generally be simulated exactly, and it will therefore be necessary to control the bias coming from the stochastic simulation. Lastly, note that even though $p_\lambda$ is differentiable, it is the expectation of a non-differentiable function of $(u,\theta)$ which will require some specific  attention in the following algorithms.

\subsubsection{Deterministic case}

Mirror Descent (MD below) originates from the pioneering work of \cite{Nemirovski:1983} and permits to naturally handle constrained optimization problems especially when the mirror/proximal mapping is explicit, which is indeed the case for a convex problem constrained on $\Db_m$ (see \textit{e.g.} \cite{Lan:2012},\cite{Bubeck:2015}). The MD approach has the nice feature to define a smooth evolution that lives inside the constrained set without adding some supplementary projection step and  ``pushes'' the frontiers of the simplex at an infinite distance from any point strictly inside 
$\Db_m$.
Though we will need to encompass the \textit{stochastic} framework, since $p_\lambda$ is defined through an expectation which is not explicit,   for the sake of clarity we describe first the \textit{deterministic} version. 

For this purpose, we first introduce the strongly convex negative entropy function over the simplex $\Db_m$ of probability distributions, defined by:
\begin{equation*}%\label{def:phi}
\forall u \in \Db_m, \qquad \varphi(u) = \sum_{i=1}^m u_i \log u_i.
\end{equation*}
If $\langle \cdot , \cdot \rangle$ refers to the standard Euclidean inner product, the Bregman divergence $\Dp$ is defined by:
\begin{equation*}%\label{def:Dphi}
\forall (u,v) \in \Db_m^2, \qquad \Dp(u,v)=\varphi(u)-\varphi(v)-\langle \nabla \varphi(v),u-v\rangle.
\end{equation*}
This Bregman divergence will induce the natural metric associated to the component $u$ involved in the problem $\mathcal{Q}_{\lambda}$.
In the same way, we also introduce the standard Euclidean square distance over $\mathbb{R}^2$, which may be viewed as a Bregman divergence. It then leads to a Bregman divergence over pairs $(\theta,u)$ and $(\theta',v)$ defined as:
\begin{equation*}%\label{def:DPhi}
\forall (\theta,\theta') \in \rset^2, \quad \forall (u,v) \in \Db_m^2, \qquad \DP((u,\theta),(v,\theta'))=\frac{(\theta-\theta')^2}{2}+\Dp(u,v),
\end{equation*}
which is the Bregman divergence associated to the strongly convex function:
$$
\Phi(u,\theta) = \varphi(u) + \frac{\theta^2}{2}.
$$
We emphasize that the strong convexity of $\Phi$ induces the following lower bound:
\begin{equation}\label{eq:rho_convex}
\forall (\theta,\theta') \in \rset^2, \quad \forall (u,v) \in \Db_m^2, \qquad
\DP((u,\theta),(v,\theta')) \ge \frac{(\theta-\theta')^2}{2}+\frac{1}{2} \|u-v\|^2.
\end{equation}

The (deterministic) MD with a step-size sequence $(\eta_k)_{k \ge 1}$ consists in minimizing from $k$ to $k+1$ the first order Taylor approximation of $p_{\lambda}$ penalized by the Bregman divergence.
To alleviate the notations , we will denote by $\X_k = (U_k,\theta_k)$ the position of the algorithm and generally $x=(u,\theta)$ an element of $\Db_m\times \R$. We observe that the MD iterative step corresponds to:
\begin{equation*}%\label{def:X}
\X_{k+1}= \arg \min_{x \in \Db_m \times \R} \bigg\{ \langle \nabla p_{\lambda}(X_k),x -\X_k\rangle + \frac{1}{\eta_{k+1}} \DP(x,\X_{k})
\bigg\}.
\end{equation*}

A remarkable feature is that this latter minimization can be made explicit :
\begin{equation*}%\label{eq:MD}
X_{k+1}=\begin{pmatrix}
  U^{k+1}\\\theta^{k+1}\end{pmatrix} \quad\text{with}\quad 
  \begin{cases}
 & U^{k+1}=\frac{U^k e^{-\eta_{k+1}\partial_{u} p_\lambda(U^k,\theta^k)}}{\|U^k e^{-\eta_{k+1} \partial_{u} p_\lambda(U^k,\theta^k)}\|_1}\\
&\theta^{k+1}=\theta^{k}-\eta_{k+1} \partial_{\theta} p_\lambda(U^k, \theta^k)\\
  \end{cases}
\end{equation*}
where the first equation has to be understood within a $m$ dimensional vector structure.
%:
%$$
%\forall i \in \{1, \ldots, m\} \qquad 
%U^{k+1}_i = \frac{U^k_i e^{-\eta_{k+1} \partial_{u_i}p_\lambda(U_k,\theta_k)}}{\sum_{j=1}^m U^k_j e^{-\eta_{k+1} \partial_{u_j}p_\lambda(U_k,\theta_k)}}.
%$$

For a fixed $\lambda>0$, we define $x^\star_{\lambda}=(u^\star_\lambda,\theta^\star_\lambda)$ the minimizer of $p_\lambda$.
Following the work of \cite{Beck:2003,Nemirovski:1983}, it can be shown that an averaged version of the sequence $(\X_k)_{k \ge 1}$ defined by  $$
\tilde{\chi}_n = \big(\sum_{k=0}^n \eta_k\big)^{-1}  \sum_{k=0}^n \eta_k X_k, 
$$
satisfies the next error bound:

\begin{theorem}[Proposition 1 in \cite{Lan:2012}] For any $\lambda>0$, we note by $$L = \arg \max_{x \in \Db_m \times \rset}\{ \|\nabla p_{\lambda}(x)\|\},$$ then
$$
p_{\lambda}(\tilde{\chi}_n)-p_\lambda(x^\star_\lambda) \leq  \frac{\{\Delta_\Phi^0\}^2+L^2 \displaystyle\sum_{k=1}^n \eta_k^2}{\displaystyle2 \sum_{k=1}^n \eta_k},
$$
where $\{\Delta_\Phi^0\}^2=\frac{(\theta_0-\va(u^\star_\lambda))^2}{2} + \log m $.
\end{theorem}
It is possible to produce several variations around this previous result. 
{It is also possible to assess an upper bound on $L^2$, using in particular Equations \eqref{eq:partial_u} and \eqref{eq:partial_theta} stated below, which implies
$$
L^2 \leq \frac{m ( 1-\alpha+\lambda )^2 \mathbb{E}[\|Z\|^2] + \lambda^2 \alpha^2}{(1-\alpha)^2}.
$$
Nevertheless, we emphasize that we will not use this upper bound as our study will be more involved.} We should only keep in mind that it is possible to finely tune the step-size sequence to obtain a $\mathcal{O}(n^{-1/2})$ upper bound. We refer to \cite{Lan:2012} for further details.

\subsubsection{Stochastic Mirror Descent (SMD) using biased simulations}
In our problem, $p_{\lambda}$ and $\nabla p_{\lambda}$ involves the computation of several expectations, which are not explicit and for which the computational time of approximation may not be reasonable. If we denote by $\partial_u$ and $\partial_\theta$ the partial derivatives with respect to $u$ and $\theta$, and using that $p_{\lambda}$ is  expectation of a convex function,  we then verify that:
\begin{equation}
\label{eq:partial_u}
\partial_u p_{\lambda}(u,\theta) = 
- \mathbb{E}[Z] + \frac{\lambda \mathbb{E}[Z \un_{\langle Z,u\rangle \ge \theta}]}{1-\alpha} ,
\end{equation}
and 
\begin{equation}
\label{eq:partial_theta}
\partial_\theta p_{\lambda}(u,\theta) = 
\lambda \Big(1-\frac{1}{1-\alpha}\mathbb{E}[\un_{\langle Z,u\rangle  \ge \theta}]\Big).
\end{equation}

\begin{remark}
We formulated the above convergence and rate estimation results for $\nabla p_\lambda$. However, these results would stay true as soon as the key function $p_\lambda$ is convex when we can access to a sub-gradient, which has been shown to be compatible with mirror descent.
\end{remark}

We assume that we observe a sequence of mutually independent random variables $(\hat{Z}^k)_{k \ge 0}$ that are also sampled independently from the previous positions of the algorithm.
The expressions  \eqref{eq:partial_u} and \eqref{eq:partial_theta} lead to a natural (possibly biased) stochastic approximation of
$\nabla p_{\lambda}$ with the help of the sequence $(\hat{Z}^k)_{k \ge 0}$. 
Assuming that the algorithm is at step $k$ at position $(U_k,\theta_k)$, we introduce the stochastic approximation of the sub-gradients: 
\begin{equation}\label{def:bias_sub_diff}
\begin{cases}
\hg_{k+1,1}& = - \hZ^{k+1} + \frac{\lambda}{1-\alpha} \hZ^{k+1} \un_{\langle \hZ^{k+1}, U_k\rangle \ge \theta_k}\vspace{1em}\\ 
\hg_{k+1,2} &= \lambda\big( 1- \frac{1}{1-\alpha}\un_{\langle \hZ^{k+1}, U_k \rangle \ge \theta_k} \big)
\end{cases}.
\end{equation}

We now describe the necessary assumptions on the sequence $(\hat{Z}^k)_{k \ge 0}$ to build a consistent SMD algorithm.\\

%\noindent \textbf{Assumption $(\mathbf{H}_1)$}
\begin{assumption}
\label{hyp:biais}
{\bf Biased simulations}\\
{Let $\mathcal{L}(\hat{Z}^{k+1})$ and $\mathcal{L}(Z)$ denote the distributions of $\hat{Z}^{k+1}$ and $Z$ respectively.} We assume that the sequence $(\hat{Z}^k)_{k \ge 0}$ satisfies both:
\begin{equation}\label{def:bias_simulation1}
\mathcal{W}_1(\mathcal{L}(\hat{Z}^{k+1}),\mathcal{L}(Z)) \leq \delta_{k+1},
\end{equation} 
and
\begin{equation}\label{def:bias_simulation2}
\forall u \in \Db_m, \quad \forall \theta \in \rset, \quad 
\Big\|\mathbb{E}\big[\langle Z, u\rangle \un_{\langle Z, u\rangle \ge \theta} - \langle\hat{Z}^{k+1}, u\rangle \un_{\langle\hat{Z}^{k+1},u\rangle \ge \theta} \, \vert \, \mathcal{F}_k \big] \Big\| \leq \upsilon_{k+1}.
\end{equation}
where $\mathcal{W}_1$ stands for the Wasserstein-1 distance.
 \end{assumption}
 
The sequences $(\delta_{k+1})_{k \ge 0}$ and $(\upsilon_{k+1})_{k \ge 0}$ translate the fact that we may observe some biased realizations of the assets $Z$ at step $k$, the perfect simulation framework being translated by $\delta_{k+1}=\upsilon_{k+1}=0$ for all integer $k$. {The Wasserstein distance used in $\delta_{k+1}$
 is indeed an easiest way to upper bound the bias naturally involved in our approximation procedure. Indeed, Equation \eqref{def:h_k} (see below) will involve the Kolmogorov distance instead, denoted by $d_{Kol}$ below, which quantifies the difference between cumulative distribution function. As the discretization of S.D.E. with (implicit) Euler scheme are well understood in terms of Wasserstein-1 distance, and because of the straightfoward inequality $d_{Kol} \leq 2 \sqrt{\mathcal{W}_1}$, we prefered to introduce Assumption \ref{hyp:biais} in terms of $\mathcal{W}_1$.}

We emphasize  that both $(\delta_{k+1})_{k \ge 0}$ and $(\upsilon_{k+1})_{k \ge 0}$ might heavily depend on $m$ the dimension of the vector $Z$. A specific example will be detailed below in our work.\\

The method we propose is then defined with Algorithm \ref{algo:SMD}.
\begin{algorithm}[h!]
  \KwData{Step-size sequence $(\eta_n)_{n\in\nset}$ and
$U_0 \in \rset$, $\var_0 \in \rset$; $\alpha \in (0,1)$}
  \KwResult{Two sequences: $X_k = (U_k,\var_k)_{k \ge 0}$}
    \For{$k=0, \ldots,$}{ 
Simulate the random $Z^{k+1}$ satisfying \eqref{def:bias_simulation1} and \eqref{def:bias_simulation2}\;  
    Compute a stochastic approximation $\hg_{k+1}$ of $\nabla  p_\lambda(U_k,\var_k)$  with:
  $$\begin{cases}  \hg_{k+1,1}& = - \hZ^{k+1} + \frac{\lambda}{1-\alpha} \hZ^{k+1} \un_{\langle\hZ^{k+1}, U_k\rangle \ge \theta_k}\vspace{1em}\\ 
\hg_{k+1,2} &= \lambda\big( 1- \frac{1}{1-\alpha}\un_{\langle\hZ^{k+1}, U_k\rangle \ge \theta_k} \big)
\end{cases}.$$
   Update the algorithm  $ X_{k+1}   = \arg \min_{x \in x \in \Db_m \times \R} \big\{ \langle \hg_{k+1}, x-X_k \rangle + \frac{1}{\eta_{k+1}} \DP(x,X_k)\big\}$ using:
  $$ 
 X_{k+1}= (U_{k+1},\theta_{k+1}),\qquad   \begin{cases}
  U^{k+1}& = \frac{U^k e^{-\eta_{k+1}\hg_{k+1,1}}}{\|U^k e^{-\eta_{k+1} \hg_{k+1,1}}\|_1}\\
    \theta^{k+1}& =\theta^{k}-\eta_{k+1} \hg_{k+1,2}
\end{cases}.
$$
  } 
    \caption{Biased SMD} \label{algo:SMD}
\end{algorithm}
The next result states the asymptotic almost sure convergence of the sequence $(X_k)_{k \ge 0}$ constructed in Algorithm \ref{algo:SMD} towards the target point $x^\star_{\lambda}=(u^\star_{\lambda},\va(u^\star_\lambda))$. Under a suitable assumption on  the biased sequences $(\delta_{k+1})_{k \ge 0}$ and $(\upsilon_{k+1})_{k \ge 0}$, we derive the next result.

\begin{theorem}[Almost sure convergence of the biased SMD]\label{theo:bias_SMD_independent}
Assume that   $$\sum_{k \ge 0} \eta_{k+1} = + \infty, \quad \text{and}\quad \sum_{k \ge 0} \eta_{k+1}^2 <  + \infty, $$
and that the bias sequences $(\delta_{k+1})_{k \ge 0}$ and $(\upsilon_{k+1})_{k \ge 0}$ introduced in Assumption \ref{hyp:biais} satisfy
$$ \sum_{k\ge0} \eta_{k+1} (\sqrt{\delta_{k+1}}+\upsilon_{k+1}) < + \infty,$$
 then:
 \begin{itemize}
 \item[$i)$] $$
 \sum_{k \ge 0} \eta_{k+1} (p_{\lambda}(X_k)-\min(p_{\lambda})) \leq 
\sum_{k \ge 0} \eta_{k+1} \langle \nabla p_{\lambda}(X_k),X_k-x^\star_{\lambda}\rangle < + \infty \quad a.s.
$$
\item[$ii)$] The Cesaro average $\bar{X}^{\eta}_k$ defined by
\begin{equation}
\label{eq:def_cesaro_average}
\bar{X}^{\eta}_k := \Big(\sum_{i=0}^k \eta_i\Big)^{-1}\Big(\sum_{i=0}^k \eta_i X_i\Big)
\end{equation}
is almost surely convergent and 
$$
p_{\lambda}(\bar{X}^{\eta}_k) \longrightarrow min(p_{\lambda}) \quad a.s.
$$
 \item[$iii)$]
 Assume that: $$
 \sum_{k\ge0} \sqrt{\delta_{k+1}} + \upsilon_{k+1}<+\infty,
 $$
 then the sequence $(X_k)_{k \ge 1}$ almost surely converges and verifies that:
 $$
 \lim_{k \longmapsto + \infty} p_{\lambda}(X_k) = \min(p_{\lambda})  \quad a.s.
 $$
 \end{itemize}
\end{theorem}
{
The results obtained in $i)$ and $ii)$ are rather classical in the stochastic optimization community. The difficulty here is to derive a ``Lyapunov'' function and then to quantify  the effect of the bias induced by $\hat{Z}^{k+1}$ satisfying \eqref{def:bias_simulation1} and 
\eqref{def:bias_simulation2}. To do so, we will follow the arguments developed in \cite{Lan:2012} in order to obtain a recursion expression but we will need to use some ad-hoc supplementary steps to handle the bias and to use the Robbins Siegmund theorem. The last item is much more difficult to obtain, roughly speaking we use the O.D.E. method introduced in  \cite{benaim1996asymptotic}. Even though it is a standard tool in stochastic optimization, applying this method to our framework is challenging because of the mirror descent O.D.E. and it deserves some sophisticated additional developments that are far from being a straightforward application of  \cite{benaim1996asymptotic}. In particular, we will need to use and adapt some fine results of \cite{Mertikopoulos} to our framework. The proof is given in Appendix \ref{app:proofs_smd}.}
\medskip

%\subsubsection{Finite time guarantee of biased SMD\label{sec:finite_time_SMD}}

Asymptotic results such as the one stated in Theorem \ref{theo:bias_SMD_independent} are commonly criticized because algorithms are always ran within a finite number of iterations. In particular, Theorem \ref{theo:bias_SMD_independent} does not provide any insight on the rate of convergence of the method, which may crucially influence the practical usefulness of the algorithm.  
Nevertheless, it shows that Algorithm \ref{algo:SMD} converges and minimizes the function $p_{\lambda}$ up to a Cesaro averaging procedure (see $i)$ and $ii)$ of Theorem \ref{theo:bias_SMD_independent}). It may be shown indeed that the algorithm converges without Cesaro averaging at the price of a more stringent condition on the terms involved in the bias of the algorithm quantified by $\delta_{k+1}$ and $\upsilon_{k+1}$.

To cope with the legitimate asymptotic  criticisms, we also state non asymptotic  theoretical guarantees for the value of the objective function. 
\begin{theorem}[Finite-time guarantees]
\label{theo:bias_SMD_finite}
Let us consider $(X_k)_{k \ge 0}$ defined in Algorithm \ref{algo:SMD} and its Cesaro averaging counterpart introduced in \eqref{eq:def_cesaro_average}, then for any $n>1$, 
$$\PE [p_\lambda(\bar{X}^\eta_n)]-p_\lambda(x^\star_\lambda)\le \big(\sum_{j=0}^{n-1}\eta_{j+1}\big)^{-1} \Big(\mathbf{D}_\Phi^0+ \sum_{k=0}^{n-1} \big(a_{k+1} \mathbf{D}_\Phi^k + b_{k+1}\big) \Big),
$$
where $\mathbf{D}_\Phi^k =\PE[\mathcal{D}_\Phi(x^\star_\lambda,X_k)]$ and
$$a_{k+1}=2 \eta_{k+1} \big(2 \sqrt{\delta_{k+1}} + \delta_{k+1} + \frac{\lambda \upsilon_{k+1}}{1-\alpha}\big), \quad b_{k+1}=  C\Big(\eta_{k+1}^2 + a_{k+1} \Big).$$
The constant $C>0$ appaering in $b_{k+1}$ is precised in \eqref{def:ab}.
\end{theorem}
Finally, we aim at applying these guarantees in the case of a finite horizon of computation $n$ and in this view, $\mathbf{D}_\Phi^k $ can be controlled using a recursion inequality \eqref{eq:borne_Dphi_k}.
We consider the specific case of the SMD with a constant step-size sequence stopped at iteration $n$:
$$\eta_{k+1}=\eta >0, \quad \forall\, 0\le k \le n.$$

We will also assume a constant upper bound of the bias in the simulation of the random variables $\hZ^k$: we denote by $\omega$ its resulting impact in the SMD. This is legitimate since we can reduce this bias with the use of an arbitrarily small discretization step-size, which of course harms the computational cost.
More precisely we consider fixed values of $\delta_{k+1}$ and $\upsilon_{k+1}$ such that:
$$2 \sqrt{\delta_{k+1}} + \delta_{k+1} + \frac{\lambda \upsilon_{k+1}}{1-\alpha}=\omega>0,\qquad \forall 1\le k\le n.$$
\begin{corollary}
\label{cor:finite_horizon}
For a given $n \in \mathbb{N}$, if $(\eta,\omega)$ are chosen such that $$\eta = \frac{\Delta_{\Phi}^0}{2\sqrt{n+1}} \quad \text{and} \quad \omega = \frac{1}{\sqrt{n+1} \Delta_{\Phi}^0}$$ with $\{\Delta_\Phi^0\}^2=\frac{(\theta_0-\va(u^\star_\lambda))^2}{2} + \log m $,
then there exists $C>0$ large enough such that:
 $$
\PE [p_\lambda (\hat{X}^\eta_n) ]- p_\lambda(x^\star_\lambda) 
\le C \frac{|\theta_0- \va(u^\star_\lambda)| + \sqrt{\log m}}{\sqrt{n+1}}. 
$$
\end{corollary}
{This previous corollary is built using the optimal tuning of the parameters $\eta$ and $\omega$ derived from our proof that is detailed in Appendix \ref{app:smd_finite}. Therefore, these values may be seen as purely theoretical as we do not exactly know the value of $\Delta_{\Phi}^0$. Nevertheless, any upper bound of $\Delta_{\Phi}^0$ may be used to derive a strategy and an upper bound of the excess risk (see Equation \eqref{eq:up_general} in Appendix \eqref{app:smd_finite}).
In the meantime, the choice: $$\eta=\frac{1}{2 \sqrt{n+1}} \quad \text{and} \quad \omega=\frac{1}{\sqrt{n+1}},$$ yields:
 $$
\PE [p_\lambda (\hat{X}^\eta_n) ]- p_\lambda(x^\star_\lambda) 
\le C \frac{|\theta_0- \va(u^\star_\lambda)|^2 + \log m}{\sqrt{n+1}},
$$
which is slightly worse  than the upper bound stated in Corollary \eqref{cor:finite_horizon}.
}

\section{Portfolio hedging under $\cva$ constraint} 
\label{sec:portfolio}

We aim to apply our optimization strategy to the specific situation of portfolio hedging.  As a consequence, we precise the structure of $Z$, a portfolio of $m$ financial assets, and detail the strategy of biased simulation that can be used. We will illustrate this setting using numerical simulations in Section \ref{sec:num}. We are exactly in the field of application described in Section \ref{sec:cvar_SMD}.

\subsection{Description of the portfolio dynamics}\label{portfolio:dynamics}
In what follows, we consider the situation where $Z$ contains $m = m'+1$ assets: a return $Y$ obtained as the baseline short-term interest rate Cox-Ingersoll-Ross process or CIR, $(r_t)_{t \ge 0}$ with no drift, and a family $\Sb=(S^1,\ldots S^{m'})$ of $m'=m-1$ geometric Brownian motions that encode  some risky assets in the portfolio.

Recall that the CIR has been introduced in \cite{CIR} as a diffusion process and that it is commonly used for the description of the dynamics of interest rates. The  CIR short rate model induces a trajectory $t \longmapsto r_t$, whose stochastic differential equation depends on a triple $(a,b,\sigma_0)$ and is given by:
\begin{equation}\label{eq:def_CIR}
d r_t = a(b-r_t) dt + \sigma_0 \sqrt{r}_t dB_0(t),
\end{equation}
where $(B_0(t))_{t \ge 0}$ stands for a standard real Brownian motion.
The parameter $b$ stands for the long-time mean of the short rate while $a$ quantifies the strength of the mean-reversion effect. The volatility $\sigma_0$ is  multiplied by $\sqrt{r_t}$, and the condition $2 a b > \sigma_0^2$ guarantees that the interest rate remains positive with probability $1$. We refer to \cite{Glasserman} for further details. This almost sure positivity motivated our decision to use the CIR process instead of the Vasicek one. Moreover, the CIR model also incorporates both the mean reversion and the conditional heteroscedasticity since the volatility of the short rate process is increased when the short rate increases. 

The assets $Z_t  = (Y_t,S^1_t,\ldots,S^{m'}_t)$ are then described by the following system of stochastic differential equations:
\begin{equation}
\forall t \ge 0 \qquad 
\begin{cases}
dY_t & = r_t Y_t dt, \\
dS^i_t & = \mu_i S^i_t dt + \sigma_i S^i_t dB_i(t), \quad \forall i \in \{1,\ldots,m'\},
\end{cases}\label{eq:portfolio}
\end{equation}
where $\mathbf{B}=(B_0,B_1,\ldots,B_{m'})$ refers to a multivariate Brownian motion with correlated components. For the sake of simplicity, these components are assumed to satisfy:
$$
\mathbb{E}[B_{i}(t) B_j(t)]= \rho_{i,j} t,
$$
but more general correlation structures could be handled in our framework with further efforts.
The correlation matrix is the symmetric positive matrix, denoted by $\Sigma = (\rho_{i,j})_{1 \leq (i,j) \leq m}$.

The first process $(Y_t)_{t \le 1}$ corresponds to a neutral risk process whereas each geometric Brownian motion $(S^i_t)_{t \leq 1}$ corresponds to a specific risky asset parametrized by \textit{known} parameters $\mu_i$ and $\sigma_i$. We refer to  Pitman and Yor \cite{pitman:yor:1982} and to Gulisashvili and Stein \cite{gulisashvili:stein:2010} for several details on these classical processes used (among others) for portfolio modeling.   
Below we also assume that all the parameters that describe the CIR evolution and the portfolio dynamics (see Equations \eqref{eq:def_CIR} and \eqref{eq:portfolio}) \textit{ are known} and we are simply interested in the optimal hedging strategy, \textit{i.e.} we are looking for the optimal $\usl$ associated to the optimal mean return penalized by $\lambda \cva$ defined in Equation \eqref{def:P22} at time $T=1$. As indicated in Section \ref{sec:cvar_SMD}, we will use a Stochastic Mirror Descent strategy and we then need to sample some realizations of $Z$ at time $T=1$. Before detailing our sampling strategy, we summarize our assumptions on the parameters.\\

%\noindent \textbf{Assumption $(\mathbf{H}_2)$}
\begin{assumption}
\label{hyp:param}
{\bf Assumptions on the portfolio parameters}\\
We assume that:
\begin{enumerate}
\item[i)]the CIR parameters satisfy $ab>\sigma_0^2$ and $a>2\sqrt{2}\sigma_0$.
\item[ii)]the correlation matrix of the Brownian motions $\Sigma$ in invertible.
\end{enumerate} 
\end{assumption}
These assumptions on the coefficients defining the CIR ensure a control of the $L^2$ moment of the weak error rate as well as exponential integrability of the integral of the CIR for all time $t$ (see Appendix~\ref{app:CIR}).

\subsection{Simulation of the portfolio}
\label{subsec:discretization}
We are interested in an efficient simulation method that satisfies Assumption \ref{hyp:biais}: both a good approximation of the law of $Z_1$ \eqref{def:bias_simulation1} and a bias upper bound of the form \eqref{def:bias_simulation2}.
 
For this purpose, we emphasize that each G.B.M. used in Equation \eqref{eq:portfolio} may be simulated exactly since an exact representation of $S^i_t$ is available:
\begin{equation*}%\label{eq:GBM_t}
\forall t \ge 0 \qquad 
S^i_t= S^i_0 \exp\big(\mu_i t + \sigma_i B_i(t)\big).
\end{equation*}
In the meantime, we can also observe that some exact simulation method exists for the CIR process sampled at a given fixed time. Nevertheless, the CIR process induces some numerical difficulties. 

\begin{itemize}
\item 
First, the SDE \eqref{eq:def_CIR} is not explicitly solvable at any time $t \in [0,1]$ so that the integral of the CIR between $0$ and $1$ needs to be approximated.
\item 
 Second the correlations between the different Brownian components of the portfolio are hardly compatible with the existing exact simulation methods of the CIR:  the presence of correlated noise, which is commonly accepted in financial modeling, is hardly tractable with the use of an exact simulation of the CIR (with the help of $\chi^2$ square distributions, see, \textit{e.g.} \cite{alfonsi:book}).
\item 
Third, one viable strategy to obtain an approximation could be the use of a discretization scheme as the implicit Euler or Milstein schemes with a fixed step-size $h>0$ for example, thanks to a recursion of the form:
$$
(r_{(k+1)h},Y_{(k+1)h}) = F(r_{k h},Y_{kh},h,\xi_{k+1}),
$$
where $F$ corresponds to an iterative update that uses the position $(r_{kh},Y_{kh})$ and $\xi_{k+1}$ is a two dimensional Gaussian innovation with a specific covariance. Nevertheless, such an attractive point of view induces some theoretical difficulties because of the form of the function $Z\un_{\langle Z,u\rangle \ge\theta}$ we need to approximate. This function is non-smooth and unbounded, which generates some technical difficulties when trying to use or adapt classical approaches on weak discretization errors.
\end{itemize}
We are  led to develop our own \textit{ad-hoc} simulation of the portfolio and prove its associated properties on the approximation.
 
 \paragraph{Correlated Brownian motion}
For this purpose, we introduce some independent standard Brownian paths $W(t)=(W_0(t),W_1(t),\ldots,W_{m'}(t))$
 and use the Cholesky decomposition to encode the correlations and recover correlated Brownian motions $B(t)=(B_0(t),B_1(t),\ldots,B_{m'}(t))$.

More precisely, the Cholesky decomposition used on $\Sigma$ yields a matrix $L$ such that: 
\begin{equation}\label{eq:chol} 
LL^T=\Sigma.
\end{equation}
We now define $B(t)=LW(t)$.

There are two important consequences of constructing $(B_i(t))_{t \ge 0}$ for $i=0, \dots, m'$ in this way.
The first one is that the Cholesky method yields a  lower triangular matrix $L$ whose first line is $1,0,\dots, 0$:
$$
L= \begin{pmatrix}
1 &0  & \dots  & 0 \\
\ell_{01} & \ell_{11} & \ddots & \vdots\\
\vdots &  & \ddots & 0 \\
\ell_{0m'} & \dots & \dots & \ell_{m'm'}
\end{pmatrix} =  \begin{pmatrix}
1 &0  & \dots  & 0 \\
\ell_{01} &  &   &  \\
\vdots &  & \tilde{L} &  \\
\ell_{0m'} &  &  & 
\end{pmatrix}.
$$
In particular, we see that we can take $B_0(t) = W_0(t)$
and $\tilde{W}(t)=(W_1(t), \dots,W_{m'}(t))$ is  independent of $W_0(t)$. The $m-1$ components may be simply written as:
$$
\begin{cases}
B_1(t)&= \ell_{01}W_0(t) + \ell_{11}W_1(t)\\
B_2(t)&= \ell_{02}W_0(t) + \ell_{12}W_1(t) +\ell_{22}W_2(t)\\
&\vdots\\
B_{m'}(t)&= \ell_{0m'}W_0(t) + \ell_{1m'}W_1(t) + \cdots+\ell_{m'm'}W_{m'}(t).
\end{cases}
$$

The second important and standard consequence is that these stochastic processes  $(B_i(t))_{t \ge 0}$ for $i=0, \dots, m'$ are indeed Brownian motions with the desired covariance matrix. 

\paragraph{Geometric Brownian Motion simulations}

From now, we consider $t=1$  and we alleviate the notations by noting $S^i_1=S_i$, $B_i(1)=B_i$ and $W_i(1)=W_i$. The asset $i$ at time $1$ for all $i \in \{1, \dots m'\}$ is a Geometric Brownian motion driven by $(B_i(t))_{t \ge 0}$, and we have:
\begin{align}
S_i &= S^i_0 \exp \Big( \big(\mu_i - \frac{(\sigma_i)^2}{2}\big) + \sigma_i B_i\Big)\nonumber\\
&=S^i_0 \exp \Big( \big(\mu_i - \frac{(\sigma_i)^2}{2}\big)+ \sigma_i \big(\ell_{0i}W_0 + \ell_{1i}W_1 + \cdots+\ell_{ii}W_i \big)\Big)\nonumber\\
&=\underbrace{S^i_0 \exp \Big(\big(\mu_i - \frac{(\sigma_i)^2}{2}\big)+ \sigma_i (\tilde{L} \tilde{W})_{i}\Big)}_{\mbox{independent of $W_0$}} \times e^{ \sigma_0 \ell_{0i} W_0},\label{eq:simulation_GBM}
\end{align}
since  $\ell_{1i}W_1 + \cdots+\ell_{ii}W_i  = (\tilde{L} \tilde{W})_{i}$ when $\tilde{W}=(W_1,\ldots,W_{m'})$.
 
\paragraph{Interest rate simulation}

Finally we focus on a discretization method for the CIR process.
We mention that discretizing the CIR process leads to some theoretical issues, as the coefficients in the SDE are not uniformly elliptic and bounded, as assumed in the seminal works of Bally and Talay \cite{bally:talay:1996}. Besides, a classical explicit Euler scheme generates positivity issues (because of the square root). However, many authors, notably Alfonsi \cite{alfonsi:2005,alfonsi:2010} proposed implicit Euler schemes and provided weak and strong error rates in the previously mentioned works. 
We choose to use an drift-implicit Euler scheme  which was introduced by \cite{alfonsi:2005} and studied in numerous articles \cite{Dereich_neuenkirch2012,alfonsi:2013} . 
The drift-implicit Euler scheme on a discrete time grid $(kh)_{0 \leq k \leq N}$ can be written by considering the SDE satisfied by $y_t=\sqrt{r_t}$ which leads to
$$
\hat{y}_{(k+1)h} =  \hat{y}_{k h} + \Big(  \frac{4ab-\sigma_0^2}{8\hat{y}_{(k+1)h}}-\frac{a}{2}\hat{y}_{(k+1)h}  \Big) h +\frac{ \sigma_0 }{2}\Delta B_0^{(k)},
$$
where $\Delta B_0^{(k)} = B_0((k+1)h) - B_0( kh)$.
This implicit scheme can be solved explicitely on $\rset_+$ from iteration $k$ to iteration $k+1$, which ensures the positivity of the scheme $(\hat{r}_{k h})_{0 \leq k \leq N}$. In particular, the update from $kh$ to $(k+1)h$ is given by: 
\begin{equation}\label{eq:discretisation_CIR}
\hat{r}_{(k+1)h}=
 \bigg(  \frac{\sqrt{\hat{r}_{kh}}+ \frac{\sigma_0}{2} \Delta B_0^{(k)}}{2(1+\frac{ah}{2})} 
 +\sqrt{  \frac{ \big(\sqrt{\hat{r}_{kh}}+ \frac{\sigma_0 }{2}\Delta B_0^{(k)}\big)^2}{4(1+\frac{ah}{2})^2}+ \frac{(4ab-\sigma_0^2)h}{8(1+\frac{ah}{2})} }  \bigg)^2.
\end{equation} 
Convergence results for this discretization scheme will be recalled in Appendix \ref{app:CIR}.

We  then use the integral representation  $Y_t = Y_0 \exp \big( \int_{0}^t r_s ds \big)$
and propose to approximate $Y_1$ with a Riemann integral approximation  between $0$ and $1$:
$$
\hat{I}_h:= \frac{1}{N} \sum_{k=1}^{N} \hat{r}_{k h}.
$$
This leads to the definition of our approximation:
\begin{equation}\label{eq:def_S1_Riemann}
\hat{Y}_1^{(h)} := Y_0 \exp(\hat{I}_h) =Y_0\exp\big(\frac{1}{N} \sum_{k=1}^{N} \hat{r}_{k h}\big).
\end{equation}

\paragraph{Overall simulation algorithm}
Combining all these steps, we are led to define the following algorithm to compute a biased simulation of $Z$.
\begin{algorithm}[h!]
  \KwData{Parameters of the CIR:  $a,b,\sigma_0$ and of the G.B.M. $(\mu_i,\sigma_i)$; Correlation matrix $\Sigma$ }
  \KwResult{A sample $\hat{Z}_1 = (\hat{Y}_1^{(h)}, S^1_1, \ldots, S^{m'}_1)$}
Compute the Cholesky factorization of $\Sigma$ (see Equation \eqref{eq:chol}): $L$ and $\tilde{L}$ .\\
Choose $\hat{r}_0$.\\
\For{$k=1, \ldots,N$}{ Compute the recursive approximation $\hat{r}_{k h}$ using  Equation \eqref{eq:discretisation_CIR}. } 
Set $\hat{Y}^{(h)}_1$ as the Riemann approximation given in   Equation \eqref{eq:def_S1_Riemann}\;
Simulate the G.B.M. $S^1_1, \ldots, S^{m'}_1$ with Equation \eqref{eq:simulation_GBM}.\\
    \caption{Approximation of the portfolio $\hat{Z}_1$} \label{algo:Simu_portefeuille}
\end{algorithm}

In order to propose a convergent algorithm derived from Theorem \ref{theo:bias_SMD_independent}, we need to derive an upper bound of the bias induced by Algorithm \ref{algo:Simu_portefeuille} involving the CIR approximation and the Riemann approximation of its integral computed in Equation \eqref{eq:def_S1_Riemann}.
The accuracy of our simulations depends on the number of discretization points $N=h^{-1}$ sampled between $0$ and $1$ in Algorithm \ref{algo:Simu_portefeuille} used to build the Riemann approximation \eqref{eq:def_S1_Riemann}. It is necessary to assess the accuracy of our method in terms of the value of $h$ to make   \eqref{def:bias_simulation1} and \eqref{def:bias_simulation2} explicit.
The simulation given in Algorithm \ref{algo:Simu_portefeuille} satisfies the next statements.

\begin{proposition}
\label{prop:wasserstein}
Assume that the portfolio parameters satisfy Assumption \ref{hyp:param} then using the above discretization scheme presented in Algorithm \ref{algo:Simu_portefeuille}, a constant $C$ exists (dependent on the CIR parameters) such that:
\begin{equation*}
\mathcal{W}_1(\mathcal{L}(\hat{Z}_{1}),\mathcal{L}(Z_1)) =\mathcal{W}_1(\mathcal{L}(\hat{Y}^{(h)}_{1}),\mathcal{L}(Y_1)) \leq C\sqrt{h}.
\end{equation*}
\end{proposition}
%\paragraph{Bias upper bound}
\begin{proposition}
\label{prop:biais}
Assume that the portfolio parameters satisfy Assumption \ref{hyp:param} then for any $\ee>0$, there exists a constant $K_\ee$ independent of $h$ and $m$ such that:
$$
\norm{\mathbb{E}\big[\langle Z_1,w \rangle\un_{\langle Z_1,w \rangle \ge \theta} - \langle \hat{Z_1},w \rangle \un_{\langle \hat{Z_1},w \rangle  \ge \theta}\big]}{2}  \leq K_{\ee} \sqrt{m} e^{\frac{\{\sigma^+\}^2 m^2 }{4\ee^2}} 
h^{\frac{1}{6}-\ee}.
$$
where $\sigma^+ = \sup_{ 1 \leq i \leq m} \sigma_i<+\infty$.
\end{proposition}
The proofs of these two results is postponed to Appendix \ref{app:CIR}.
{We emphasize that these two satellite results contain some sharp new estimations of numerical probability nature on the Euler scheme used for the simulation of our portfolio, and may be of independent interest for applications in mathematical finance. They are not so easy to obtain in particular regarding the influence of the number of assets $m$ and need some technical computations to obtain this last dependency that is crucial in our stochastic optimization procedure to design an efficient bias reduction strategy.}

\subsection{Stochastic Mirror Descent with sampling approximation}

We finally aggregate the optimization procedure described in Algorithm \ref{algo:SMD} with our sampling scheme in Algorithm \ref{algo:Simu_portefeuille} and we address the problem of adapting the step size of the discretization scheme to the step size of the SMD approximation in light of Theorems \ref{theo:bias_SMD_independent}, \ref{theo:bias_SMD_finite} and Propositions \ref{prop:wasserstein} and \ref{prop:biais}.

\paragraph{SMD at fixed time horizon $n$}
According to our previous results, we can now consider the choice of parameters induced by Corollary \ref{cor:finite_horizon}. Recall that we assume a constant step-size sequence $\eta$ and discretizaton step-size $h$. Corollary \ref{cor:finite_horizon}
combined with  Propositions \ref{prop:wasserstein} and \ref{prop:biais} induce that $h$ should be chosen as:
$$h^{1/4}+ h^{\frac{1}{6}-\ee}\sim n^{-1/2},$$
which entails that we could choose a discretization step-size close to $n^{-3}$.

\paragraph{SMD with decreasing step-size sequence}
Let us now choose adequate step-size sequences in order to obtain an a.s. convergent optimization algorithm (Theorem \ref{theo:bias_SMD_independent}). 
Recall that the sequence $(\eta_k)_{k \ge 1}$ scales the step-size of stochastic gradient descent and assume that 
$$\eta_k=k^{-\alpha}, \quad \alpha\in(\frac12,1].$$
We now chose the step sequence for approximating the CIR process as:
%We aim at adjusting the step used in the CIR approximation chosen as 
$$h_k= h^{(m)}_0 k^{-\beta}, \quad \beta>0.$$
According to Propositions \ref{prop:wasserstein} and \ref{prop:biais}, using the notations of Theorem  \ref{theo:bias_SMD_independent}, we deduce that:
$$\delta_k=  k^{-\frac{\beta}{2}} \qquad \text{and}\qquad v_k=K_{\ee} \sqrt{m} e^{\frac{\{\sigma^+\}^2 m^2 }{4\ee^2}} 
h^{-\beta(\frac{1}{6}-\ee)}.$$
Assumption in Theorem \ref{theo:bias_SMD_independent}
$\sum_{k\ge0} \eta_{k+1} (\sqrt{\delta_{k+1}}+\upsilon_{k+1}) < + \infty$
 now reads:
$$\sum_{k\ge1} k^{-\alpha}\big(k^{-\frac{\beta}{4}}+ k^{-\beta(\frac{1}{6}-\ee)}\big) <\infty,$$
which is equivalent to:
$$\alpha + \frac{\beta}{6}>1, \qquad \text{with}\quad \alpha\in(\frac12,1].$$
With these conditions on $\alpha $ and $\beta$, we can now derive the non-asymptotic bound Theorem \ref{theo:bias_SMD_finite}. We set $\epsilon= \alpha+\frac{\beta}{6}-1>0$,
using the notations of Theorem \ref{theo:bias_SMD_finite}, we deduce that:
$$a_{k+1} \sim k^{-1-\epsilon} \quad\text{and}\quad  b_{k+1} \sim k^{-(2\alpha\wedge (1+\epsilon))}. $$
Using comparisons between series and integral, we find that: 
$$\mathbf{D}_\phi^k\leq D_\Phi^0 \exp(k^{-\epsilon}))+ k^{-(2\alpha-\epsilon+1)}.$$
If we use the notation $\lesssim$ that denotes an inequality up to a constant $C_m$ which heavily depends on the number of assets in the portfolio, we then obtain that:
\begin{align*}
\PE [p_\lambda(\bar{X}^\eta_n)]-&p_\lambda(x^\star_\lambda)
\le \frac{1}{\sum_{j=0}^{n-1}\eta_{j+1} }\bigg(\mathbf{D}_\Phi^0+ \sum_{k=0}^{n-1} \Big[a_{k+1} \mathbf{D}_\Phi^k + b_{k+1}\Big] \bigg)\\
\lesssim &\frac{1}{n^{1-\alpha}}\bigg(\mathbf{D}_\Phi^0+ \sum_{k=0}^{n-1} \Big[k^{-1-\epsilon} [ \mathbf{D}_\Phi^0 \exp(k^{-\epsilon}))+ k^{-(2\alpha+\epsilon-1)} ]+ k^{-(2\alpha\wedge (1+\epsilon))} \Big] \bigg)\\
\lesssim&\, n^{\alpha-1} \mathbf{D}_\Phi^0\Big(1+\exp(n^{-\epsilon}) \Big) +n^{-\alpha-2} + n^{-\alpha\wedge (1+\epsilon-\alpha)}\\
\lesssim  &\,n^{\alpha-1} \mathbf{D}_\Phi^0 + n^{-\alpha\wedge \frac{\beta}{6}}.
\end{align*}
We point out that choosing $\alpha=1/2$ and $\beta>3$ allows for a convergence rate $O(n^{-1/2})$ which is similar to the previous case.

\section{Simulations}\label{sec:num}

In this brief numerical paragraph, we illustrate the behavior of our algorithm on both synthetic and real datasets. The simulations have been driven using Python 3.7.13 on a standard computer.
\subsection{Description of the simulation study}
\paragraph{Bregman divergence vs Euclidean projection}
In order to assess the specificity and efficiency of our algorithm, we also consider the projected stochastic gradient optimization that is obtained after a stochastic gradient step-size + a Euclidean projection on the simplex of probability distribution. The projection has been exactly computed with the help of the standard recursive method derived from the Lagrangian formulation.
 Below, we will refer to SMD for the stochastic mirror descent and PSGD for the projected stochastic gradient descent  (whose definition is straightforward (see \cite{singer,blondel,condat}). In what follows, we have used the recent algorithm introduced in \cite{singer} to project any points into the simplex of probability measure for which a fast  implementation is available in \cite{blondel} and for which a complexity bound of $O(m^2)$ is established in \cite{condat}. {That being said, it is also shown in \cite{condat} that $O(m^2)$ is the worst-case complexity  of the method while in practice the observed complexity is only $O(m)$. A such complexity is also assessed in expectation when using a random pivot strategy as proposed in \cite{Kiwiel}. As reported in Table 1 of \cite{condat}, the observed complexity of several projection methods, including the one of \cite{condat}, is $O(m)$. We have chosen to compare SMD with PSGD with the projection method of \cite{condat} as indicated in Section 4 of \cite{condat}, their state of the art algorithm generally outperforms the other methods for reasonable vectors to be projected into the simplex.}
 
\paragraph{SMD vs. MC-MD}
 We also compare our algorithm with a batch Monte-Carlo Mirror Descent, \textit{e.g.} we use a mini-batch simulation instead of a single one at each step of the algorithm. Therefore, in the Monte-Carlo Mirror Descent (MC-MD), an additional loop is included in order to perform a Monte-Carlo estimation of the gradients \eqref{eq:partial_u} and \eqref{eq:partial_theta}. Of course, a such strategy improves the quality of the estimation but has a supplementary cost in terms of simulation times. We present below a comparison of the two methods on synthetic data  where we will compare the results of SMD and MC-MD in terms of numerical costs in seconds.
 
 \paragraph{Diversification and influence of $\lambda$}
 To undertake the influence of the penalty parameter $\lambda$, and in particular its effect on the composition of the optimal portfolio, we vary its value in a regular grid $\big(\lambda_{\min} + s \frac{\lambda_{\max}-\lambda_{\min}}{\Lambda}\big)_{0 \leq s \leq \Lambda}$ defined with the help of a lower and upper value $(\lambda_{\min},\lambda_{\max})$ and a number of points $\Lambda+1$ in the grid. According to Proposition \ref{prop:equivalence}, it is expected that large values of $\lambda$ essentially favor low risk portfolio composition whereas small values of $\lambda$ allow for risky portfolio with higher expected returns.
 
\paragraph{Efficient frontier construction}
Since the value of $\lambda$ has to be determined from a practical point of view, we decided to use the Markowitz efficient frontier approach \cite{markowitz:1952}. The efficient frontier was initially introduced with the standard deviation as a natural measure of the risk of the portfolio. This curve was initially obtained by representing the risk in the X-axis and the Expected Return in the Y-axis. Below, we replace the volatility by the CV@R as a natural risk measure as proposed in \cite{Krokhmal:2001} to build our Markowitz efficient frontier. The frontier is constructed as follow, for an integer {$ s \in [0,\Lambda]$}, we compute with our SMD algorithm an approximation of the optimal portfolio associated to this penalty parameter $\lambda_s = \lambda_{\min} + s \frac{\lambda_{\max}-\lambda_{\min}}{\Lambda}$, \textit{i.e.} the set of weights $(\hat{u}_{\lambda_s},\hat{\theta}_{\lambda_s})$ that approximately minimizes $p_{\lambda_s}$:
$$
 (\hat{u}_{\lambda_s},\hat{\theta}_{\lambda_s}) =  \arg\displaystyle\min_{(u,\theta) \in \Db_m  \times \rset} \, p_{\lambda_s}.
$$

For convenience, we denote $$p_{\lambda} (u,\theta)=-J(u) + \lambda \psi_\alpha(\theta,u)$$ where $J(u)=- \sum_{i=1}^m u_i \mathbb{E}[Z_i]$.

We then draw the frontier obtained $(CV@R_{\alpha}(\hat{u}_{\lambda_s}),J(\hat{u}_{\lambda_s}))$ that may be computed on-line all along the iterations of the SMD algorithm with a Cesaro averaging strategy: if $(X_k)_{k \ge 0} = (u_k,\theta_k)_{k \ge 0}$ corresponds to a random sequence built with Algorithm \ref{algo:SMD} stopped at iteration $n$ and if $(\hat{Z}^k)_{k \ge 0}$ is the biased porfolio simulation sequence used all along Algorithm \ref{algo:SMD}, then we may approximate the expected return $J(\hat{u}_{\lambda_s})$ and the risk $CV@R_{\alpha}(\hat{u}_{\lambda_s})$ by:
  $$
\frac{1}{n}  \sum_{k=1}^n \langle u_k,\hat{Z}^{k} \rangle \quad \text{and} \quad \frac{1}{n} \sum_{k=1}^n \Big(\theta_k+ \frac{1}{1-\alpha}  \big(\langle \hat{Z}^k,u_k\rangle - \theta_k \big)^+\Big).
  $$
We emphasize that since this construction requires the optimization of the portfolio for each value of $\lambda$, it is mandatory to design an efficient and fast procedure to solve each optimization problem, which legitimates the use of our stochastic algorithm.

Finally, we propose to use the Sharpe ratio criterion (see \textit{e.g.} \cite{Sharpe}) to select the optimal portfolio among all the portfolio found on the Markowitz frontier. More precisely, if $\un_{CIR}$ refers to the risk-free portfolio hedging $\un_{CIR}=(1,\ldots,0)$, 
 the Sharpe ratio criterion selects the portfolio that maximizes the ratio between the difference $J(\hat{u}_{\lambda_s}) - J(\un_{CIR})$ and the risk measured by $CV@R_{\alpha}(\hat{u}_{\lambda_s})$.
\subsection{Synthetic dataset\label{sec:synthetic_app}} 
In this first set of simulations, we implement our SMD with purely artificial assets that are driven with a CIR and several GBM, whose parameters are chosen such that the mean reward / Expected Return (ER for short below) and the individual CV@R monotonically vary together. Figure  \ref{fig:trajectory} represents one simulation run of the assets we used. At this stage, we do not use any correlation between assets, which is of course a very simplifying assumption that is typically false for financial application purposes. In a second stage in Section \ref{sec:finance_app}, we will consider correlated assets. 

\begin{figure}
\begin{center}
\includegraphics[width=0.95\linewidth]{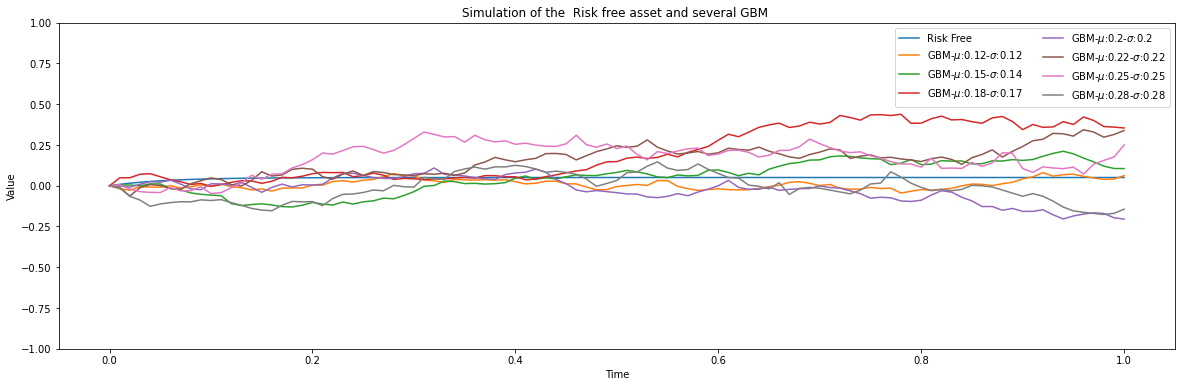}
\caption{Time evolution of the return of the discretized trajectories associated to the assets of the synthetic portfolio (CIR + GBM).\label{fig:trajectory}}
\end{center}
\end{figure}
\paragraph{Comparison between SMD and PSGD}
Our first experiment consists in comparing the behaviour of the SMD and the  PSGD. The value of $\lambda$ used here is $\lambda=0.9$.

Figure \ref{fig:comparaison_SMD_PSGD_1} shows that there is no real evidence that SMD and PSGD lead to very different results in terms of the rate of convergence of $p_{\lambda}(X_k) \longrightarrow \min (p_{\lambda})$, even though as indicated in Figure \ref{fig:comparaison_SMD_PSGD_1} the limiting weights reached by the two algorithms may not be the same. As indicated in Figure  \ref{fig:comparaison_SMD_PSGD_3}, the performances in terms of ER and CV@R are equivalent, which indicates that the convex function $p_\lambda$ may have an infinite set of minima, with a continuum of critical points. Even in this theoretically difficult situation, the trajectory of SMD (and the one of PSGD) almost surely converges as indicated in Theorem \ref{theo:bias_SMD_independent} and as shown in Figure \ref{fig:comparaison_SMD_PSGD_1}.

\begin{figure}
\begin{center}
\includegraphics[width=7cm,height=5cm]{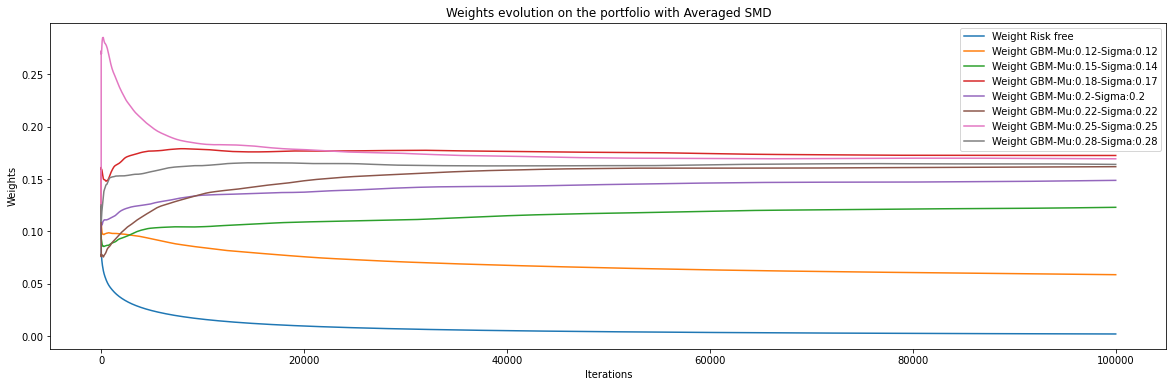}
\includegraphics[width=7cm,height=5cm]{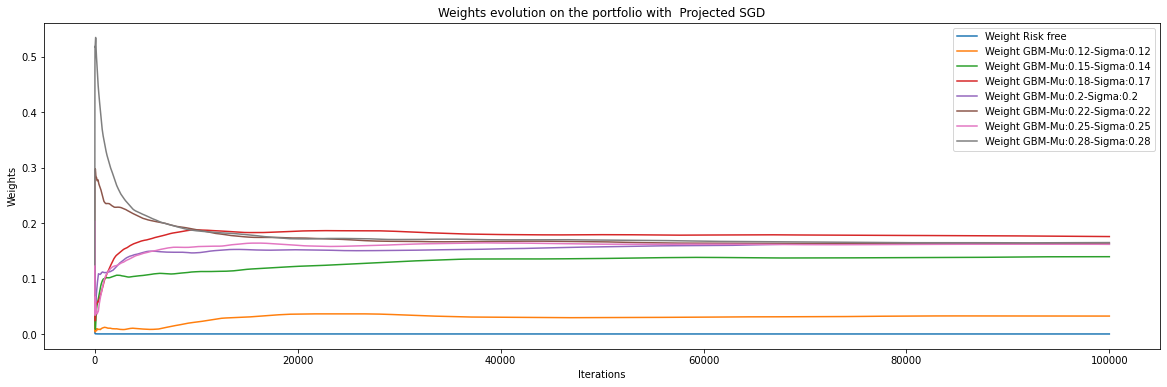}
\end{center}
\caption{Evolution of the weights of the assets with the number of iterations of the SMD (left) and of the PSGD (right) on $p_\lambda$ with $\lambda=0.9$ and $\alpha=0.05$.  The composition of optimal portfolio may be different as represented above with SMD and PSGD.\label{fig:comparaison_SMD_PSGD_1}}
\end{figure}

\begin{figure}
\begin{center}
\includegraphics[width=7cm,height=5cm]{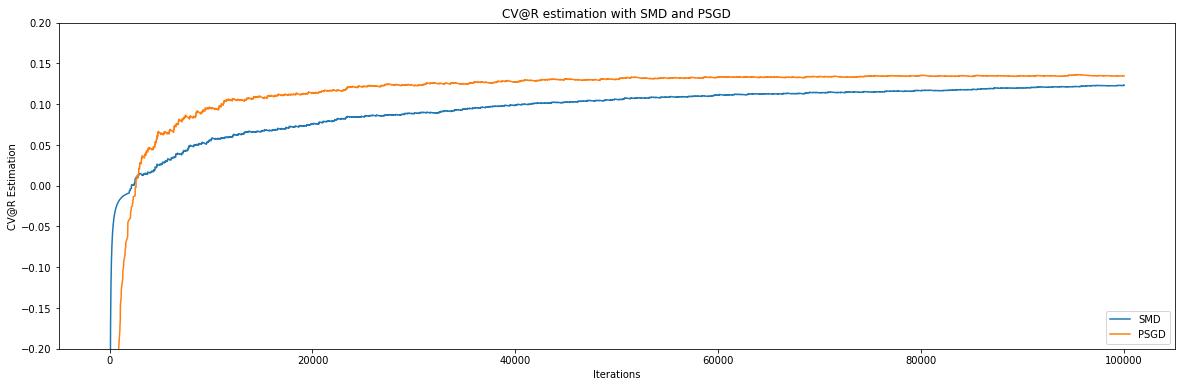}
\includegraphics[width=7cm,height=5cm]{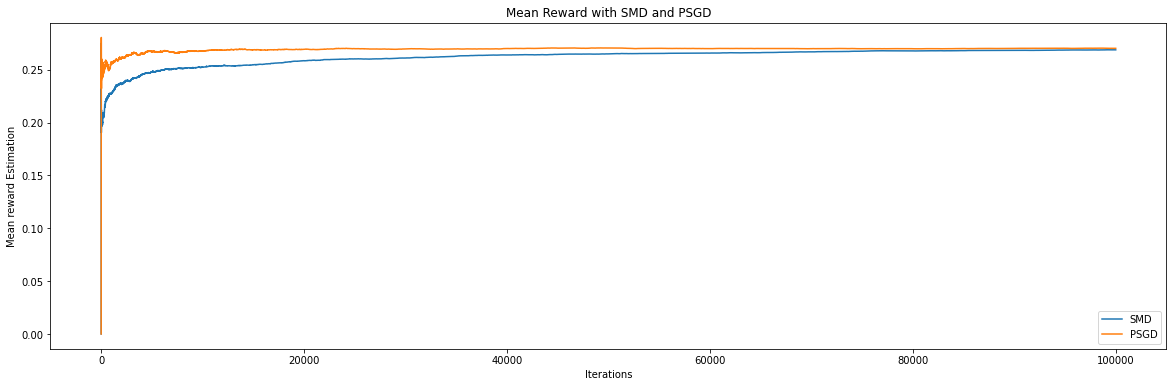}
\end{center}
\caption{Evolution of the CV@R and ER of the optimal portfolio with the number of iterations of the SMD (orange) and of the PSGD (blue) on $p_\lambda$ with $\lambda=0.9$ and $\alpha=0.05$. \label{fig:comparaison_SMD_PSGD_3}}
\end{figure}

Even though both algorithms achieve almost the same results in terms of weights, CV@R and ER, the numerical cost associated to the Euclidean projection is important and prevents from the use of the PSGD with a large number of assets. Table \ref{tab:cout} provides a brief comparison of the the computational time ratio between PSGD and SMD, which demonstrates that SMD with several dozen of assets requires generally at least  $5$ times less
computations when compared to the update with PSGD.
 This phenomenon is of course amplified when $m$ increases since the projection algorithm requires more and more operations in larger dimensions.

\begin{table}[ht]
\centering
\begin{tabular}{c  |  c c}
&  $n=10^4$ & $n= 10^5$ \\
\hline
$m=5$& 2 & 3 \\
$m=20$ &3 & 5 \\
$m=40$  & 4& 6\\
$m=80$   &5& 7\\
\end{tabular}
\caption{Ratio of computation with $n$ iterations on $m$ assets between PSGD and SMD (rounded at the closest integer).}
 \label{tab:cout}
\end{table}

{All the more, as indicated in Figure \ref{fig:comparaison_SMD_PSGD_3} (the same conclusion can be drawn from Figure \ref{fig:comparaison_SMD_PSGD_1}), the speed of convergence of SMD seems fastest than the one of PSGD from a numerical point of view. Of course, this remark is purely based on empirical observations from numerical simulations and not on  a mathematical  results as the theoretical rates of convergence of the methods are equivalent (both rates evolve as $n^{-1/2}$ where $n$ is the number of iterations of the algorithms and as $n^{-1}$ with the supplementary Ruppert-Polyak averaging strategy).
This is certainly due to the fact that the averaged SMD algorithm evolves smoothly over the simplex while the averaged PSGD approach is much more irregular due to the sequence of gradient step outside + projection inside the simplex. Finally, this last remark associated with the supplementary computational time due to the projection over the simplex (instead of the instantaneous mirror update) amplifies the benefits of SMD when compared to PSGD regarding the speed of computation. This last remark is inline with some recent observations in the machine learning community (see \textit{e.g.} \cite{zhou:etal:2017}).}

{\paragraph{Comparison between SMD and MC-MD}
We now present the result of the comparison of SMD and MC-MD. To take into account the computational cost of the MC step, we represent the estimation of the algorithm with respect to the computational time (in second).  In these simulation, we draw 10 independent realization of the portfolio in the MC step. We observe that the supplementary accuracy of the estimation of the gradient, obtained at each iteration, is not sufficient to overcome the additional supplementary cost (in terms of seconds) generated by the Monte-Carlo step.
As indicated by the right hand side of Figure \ref{fig:comparaison_SMD_MCMD_1}, it appears that both SMD and MC-MD converge to the same portfolio when we push the number of iterations (e.g. the computational time), which lead to the same value of the V@R of the portfolio.
We have compared in Figure \ref{fig:comparaison_SMD_MCMD_1} the learning speed of SMD and of MC-MD of the optimal weights and of the optimal V@R. Our simulations seem to recommend the use of a single stochastic simulation per iteration instead of the use of a mini-batch MC strategy for the problem we are studying.}

\begin{figure}[h]
\begin{center}
\includegraphics[width=7cm,height=5cm]{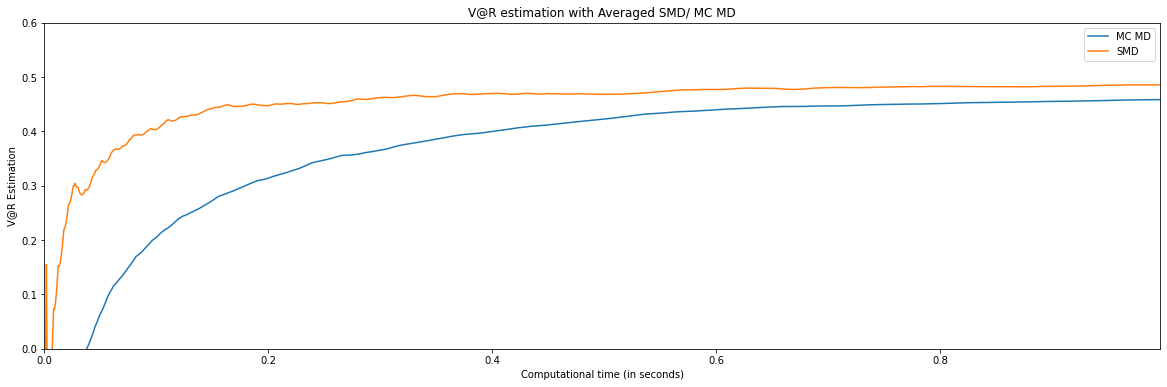}
\includegraphics[width=7cm,height=5cm]{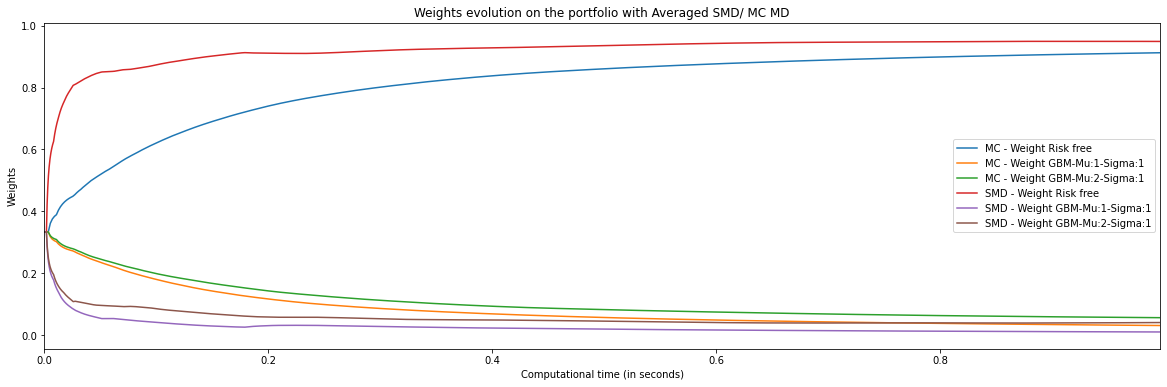}
\end{center}
\caption{Evolution of the V@R estimation (left) of the portfolio with the computational time of the SMD and of the MC-MD $p_\lambda$ with $\lambda=0.95$ and $\alpha=0.05$.  The composition of optimal portfolio (right) converges to similar values as represented on the right with SMD and MC-MD but with a different cost in computational time.\label{fig:comparaison_SMD_MCMD_1}}
\end{figure}

\paragraph{Effect of the penalty parameter $\lambda$}
In Figure \ref{fig:effect_lambda}, we assess the influence of the penalty coefficient $\lambda$ involved in $p_{\lambda}$ on the optimal solutions obtained with the SMD. The experiments have been set up with 3 and 8 assets in the portfolio. The fluctuations of the curves in Figure \ref{fig:effect_lambda} are due to the early stopping strategy we adopted to limit the computational cost, but the following conclusions are highly trustable since the evolution of each curve is rather clear.
 We have indicated in the legend of each subfigure in Figure \ref{fig:effect_lambda} the ER of each individual asset that is estimated on-line with our algorithm, as well as the CV@R. We observe that for small values of $\lambda$, the optimal portfolio with 3 and 8 assets favor the most risky ones and do not weight the risk-free component. It is of course the opposite behaviour for large values of $\lambda$ when essentially the CIR component gathers the essential weight of the portfolio. Finally, in the intermediary regime, the CV@R penalty generates a product diversification, which is recovered both on the left and on the right of Figure \ref{fig:effect_lambda}, which is also exactly the objective of the CV@R penalty.

\begin{figure}
\begin{center}
\includegraphics[width=7cm,height=5cm]{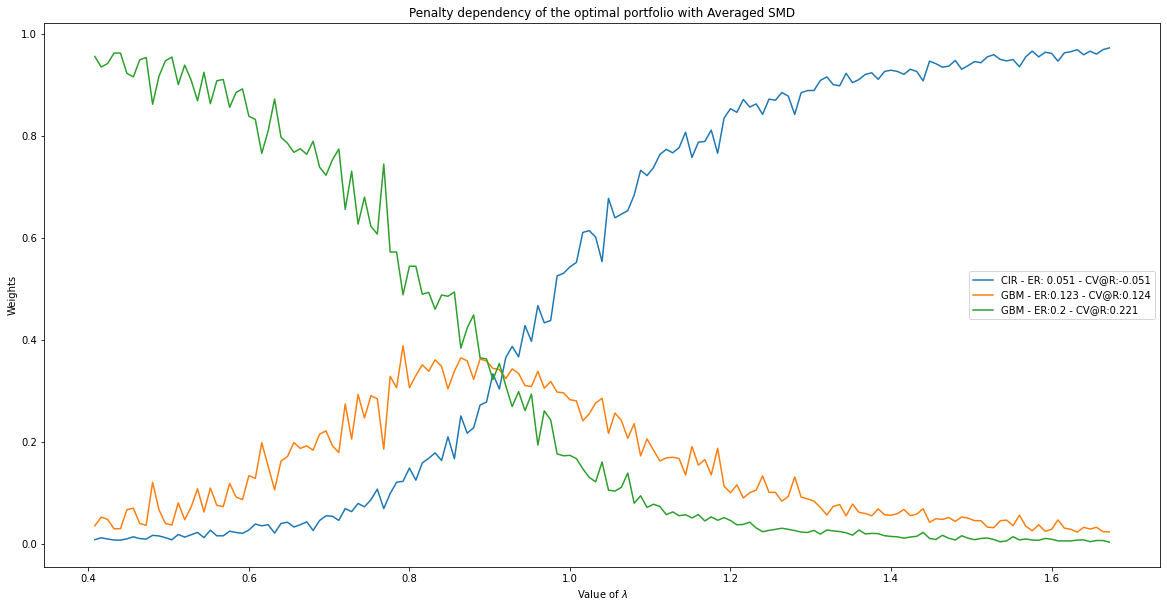}
\includegraphics[width=7cm,height=5cm]{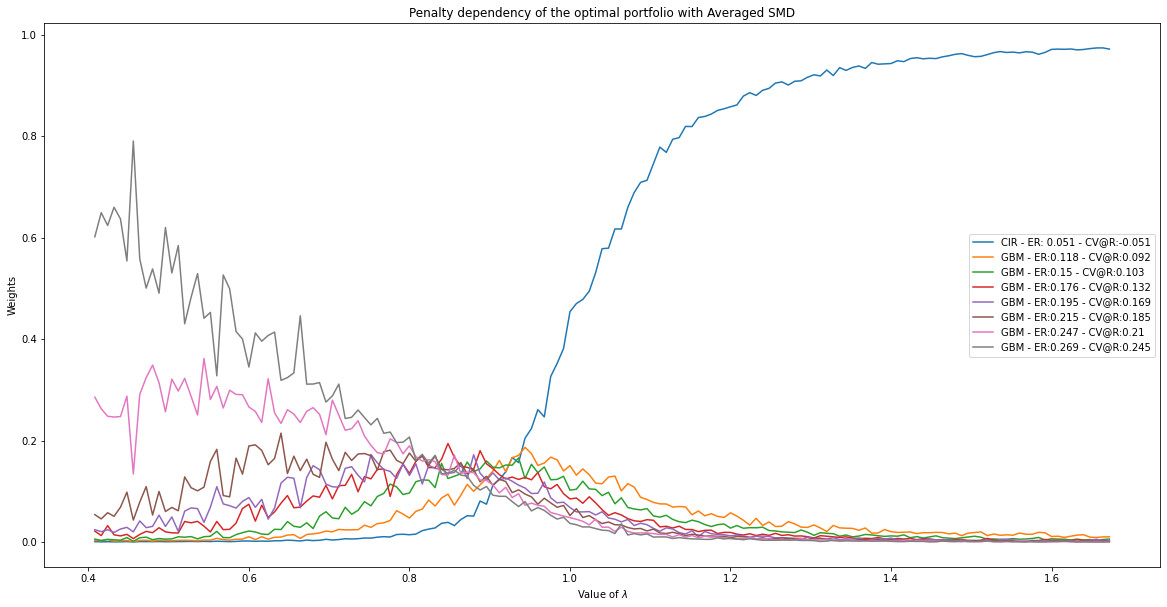}
\end{center}
\caption{ Composition of the optimal portfolio when $\lambda$ increases. Left: 3 assets, Right: 8 assets. ER and CV@R are indicated in the legend box of each subfigure. \label{fig:effect_lambda}}
\end{figure}
%\lorick{Tu veux dire que ER et CVAR sont dans l'encadré? }
\subsection{Real dataset \label{sec:finance_app}}
We briefly detail how we shall use our algorithm to estimate optimal portfolio. This short numerical paragraph should be understood as a proof of concept study since we believe that it illustrates the ability of our method.
\paragraph{Yahoo! dataset and parameters estimation}
In this final study, we consider some stock prices that are daily recorded on the Yahoo! website \url{https://finance.yahoo.com/lookup} and that are freely available with the yahoo-finance API of Python \cite{data}.

We use the daily dataset of the Treasury Yield 5 Years ($\hat{}$\,FVX) to calibrate the null risk asset as well as several assets that are detailed in Figure \ref{fig:assets_portfolio}, between 2014 and 2016 (to avoid the period of negative interest rates). We model the $\hat{}$\,FVX time series as a daily realization of a CIR stochastic model whereas risky assets are considered as correlated GBM. Parameters are estimated following the optimal martingale method as it is reported in \cite{these} as it outperforms the discretized log-likelihood maximization approach (for discretely observed CIR realizations). We refer to Equation 4.2 for the drift coefficients and Equation given at the end of Section 4.3 of \cite{these}.

Concerning the GBM, we estimate their coefficients ($\mu,\sigma$ and $\rho$) with the log-return series, which is a standard procedure described among other in Section 3.2.3 of \cite{Glasserman}  for example.

\begin{figure}[h!]
\begin{center}
\includegraphics[scale=0.23]{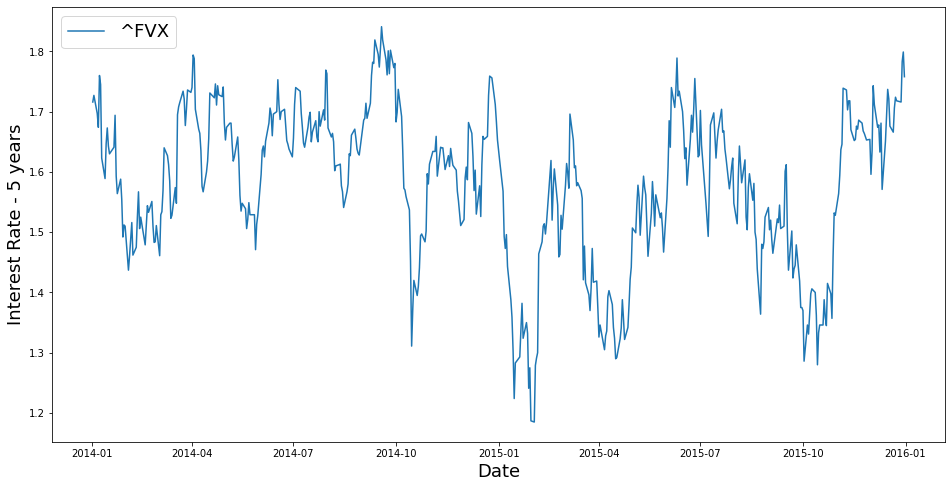}
\includegraphics[scale=0.23]{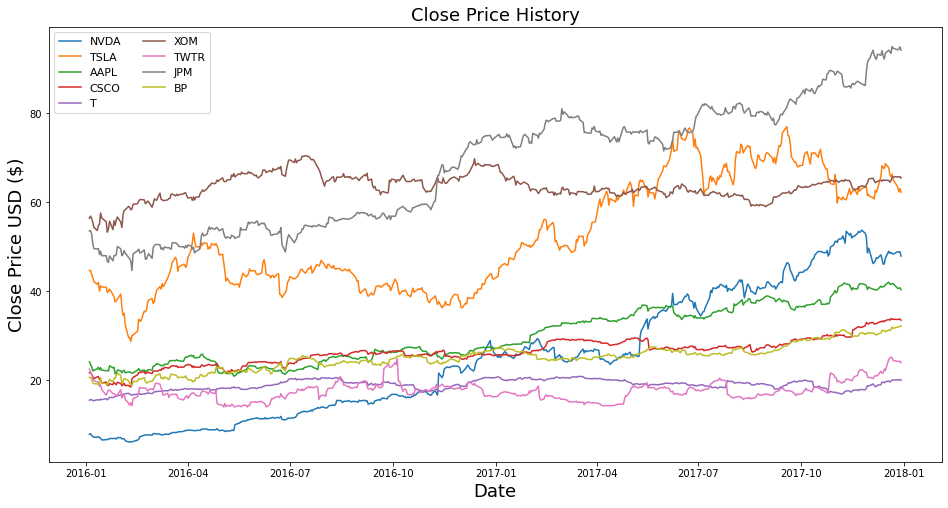}
\end{center}
\caption{Sample of the $\hat{}$\,FVX series and of the assets that composed our portfolio, from the Yahoo! database, bewtween the 01-01-2014 and 01-01-2018 (2016 for the $\hat{}$\,FVX time series). \label{fig:assets_portfolio}}
\end{figure}

\paragraph{Markowitz efficient frontier with CV@R}
Once the parameters are estimated, we assume that their values are kept fixed for the next daily observation so that we are able to use our SMD Monte-Carlo approach to estimate the optimal portfolio associated to these financial assets.  We then use our SMD to estimate the Markowitz efficient frontier, adapted with the CV@R as a measure of risk of the portfolio. We represent we obtain in Figure \ref{fig:Markowitz} for two different level of quantiles ($\alpha=5\%$ and $\alpha=1\%$).
\begin{figure}[h!]
\begin{center}
\includegraphics[scale=0.2]{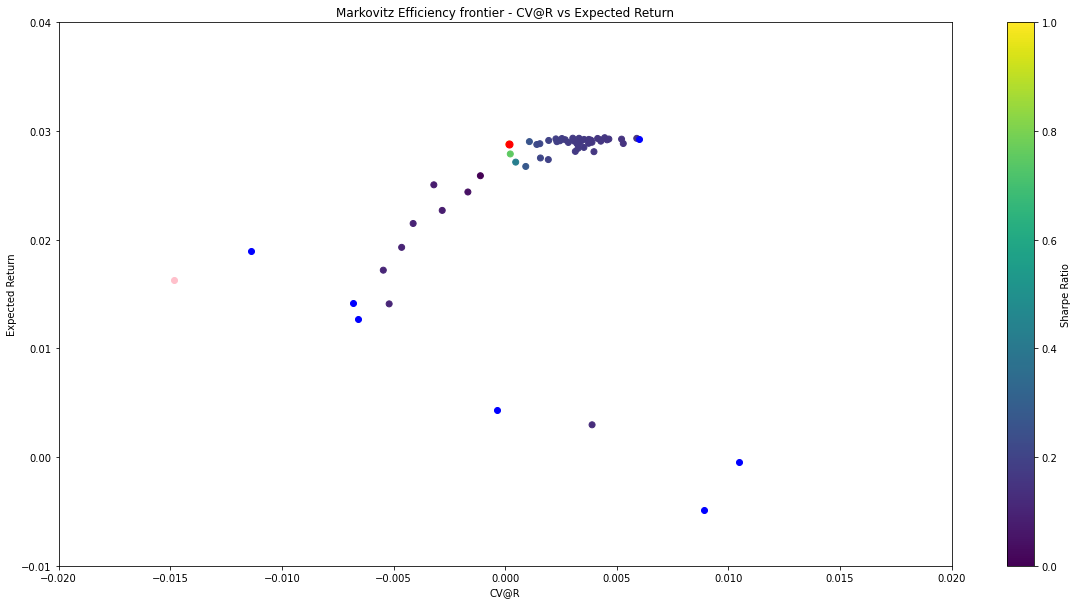}
\includegraphics[scale=0.2]{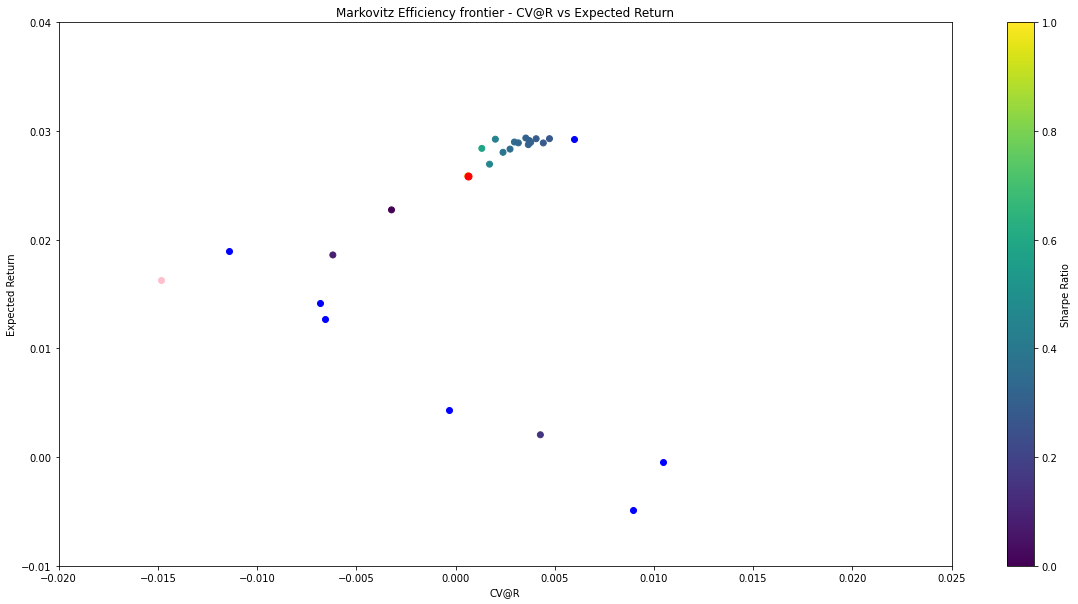}
\end{center}
\caption{Markowitz Efficient frontier. Left: $\alpha=5\%$ and Right: $\alpha=1\%$. Red point: optimal Sharpe ratio location. Blue points: location of the assets involved in the portfolio. Pink: risk free asset. \label{fig:Markowitz}}
\end{figure}

We use for the construction of our frontier a discretization of the SDE and compute the return of the portfolio for T=30 days. The red point represents the position of the "optimal" portfolio according to the Sharpe ratio. The computational cost for the construction of the Markowitz frontier with 10 assets and a time window $h \propto n^{-3}$ appears to be reasonnably fast since we compute these frontiers on standard PC with Python 3.0 in less than 10 minutes each. As indicated in Figure \ref{fig:Markowitz}, SMD produces each time an optimal portfolio that really balances the CV@R and the ER.
\section*{Competing Interests}
The authors declare no competing interests.

\appendix\normalsize

\section{Lagrangian formulation  \label{sec:appendix_lagrange}}
This appendix is dedicated to the proof of the Lagrangian unconstrained formulation used to deal with the \cva constraint.

\begin{proof}[Proof of Proposition \ref{prop:equivalence}]

We observe that $J(u) = - \sum_{i=1}^m u_i \mathbb{E}[Z_i] $
is a continuous convex function, defined on the convex simplex $\Db_m$ and $\mathcal{P}_M$ is a convex optimization problem defined with a constraint $g_M$ defined through a \cva inequality.
It is well known that \cva\ induces a coherent measure of risk, which implies the two important properties of positive homogeneity and sub-additivity. We refer to \cite{Pflug} for this important remark, and to \cite{Artzner97,Artzner99} for seminal contributions on risk measurements in mathematical finance. The important consequence of a such coherence is the convexity of the function that defines the constraint in $\mathcal{P}_M$. We verify that $\forall \lambda \in (0,1),$  $\forall (u,v) \in \Db_m$:
\begin{align*}
 \cva(\lambda u + (1-\lambda) v) &\leq  \cva(\lambda u) + \cva((1-\lambda)v)\\
&=\lambda \cva(u)+(1-\lambda) \cva(v).
\end{align*}

We are led to define the convex function $g_M$ that induces the constraint in $\mathcal{P}_M$ by:
$$
g_M: u \longmapsto \cva(u)-M.
$$
We shall observe that the collection of functions $(g_M)_{M \in \rset}$ defines a family of convex constraints and $\exists M_0 \in \rset$ such that for all $M > M_0$, { the constraint $g_M$ is strictly feasible: there exists $u \in \Db_m$ that is an interior point such that $g_M(u)<0$}. For a such $M > M_0$, we then define the Lagrangian function:
$$
\mathcal{L}_M(u,\lambda) := J(u) + \lambda g_M(u).
$$
{From the Slater condition, we deduce that the strong duality involved in $\mathcal{P}_M$ holds}: there is no duality gap and primal and dual optimization problems coincide. In particular, the Karush-Kuhn-Tucker conditions are satisfied: the solution $u^\star_M$ of the  primal problem $\mathcal{P}_M$  verifies:
\begin{equation}\label{eq:primal_lagrangian}
\exists \, \lambda^\star_M \ge 0 \quad u^{\star}_M = \arg\min_{u \in \Db_m} \mathcal{L}_{{M}}(u,\lambda^\star_M),
\end{equation}
and the complementary slackness is verified for the  pair $(u^\star_M,\lambda^\star_M)$:
\begin{equation}\label{eq:complementary}
\lambda^\star_M g_M(u^\star_M)=0.
\end{equation}

{We now establish the monotonous property. We consider a pair $(M_1,M_2)$ and consider $(u_1^\star,\lambda_1^\star)$ and $(u_2^\star,\lambda_2^\star)$ derived from the previous Lagrangian equation. We know that the complementary slackness condition holds for both $(u_1^\star,\lambda_1^\star)$ and $(u_2^\star,\lambda_2^\star)$.} Equations  \eqref{eq:primal_lagrangian} and \eqref{eq:complementary} entail:
\begin{align}
J(u_1^\star) &= {J}(u_1^\star)+\lambda_1^{\star} g_{M_1}(u_1^\star)\nonumber\\
 & = \mathcal{L}_{{M_1}}(u_1^\star,\lambda_1^{\star}) \nonumber\\
&
\leq  \mathcal{L}_{{M_1}}(u_2^\star,\lambda_1^{\star})  \nonumber\\
& = J(u_2^\star)+\lambda_1^{\star} g_{M_1}(u_2^\star) \nonumber\\
& = {J(u_2^\star) + \lambda_2^{\star} g_{M_1}(u_2^\star) + (\lambda_1^\star-\lambda_2^\star)  g_{M_1}(u_2^\star) }\nonumber\\
& = {J(u_2^\star) + \lambda_2^{\star} \big( g_{M_2}(u_2^\star)+M_2-M_1\big) + (\lambda_1^\star-\lambda_2^\star)  g_{M_1}(u_2^\star) }\nonumber\\
& = { \mathcal{L}_{{M_2}}(u_2^\star,\lambda_2^{\star}) + \lambda_2^{\star} \big( M_2-M_1\big) + (\lambda_1^\star-\lambda_2^\star)  g_{M_1}(u_2^\star) }\nonumber\\
& \le {\mathcal{L}_{{M_2}}(u_1^\star,\lambda_2^{\star}) + \lambda_2^{\star} big( M_2-M_1\big) + (\lambda_1^\star-\lambda_2^\star)  g_{M_1}(u_2^\star) }\nonumber\\
%& \le  J(u_2^\star)+\lambda_2^{\star} g_{M_2}(u_2^\star)  + \lambda_2^\star(M_2-M_1) +(\lambda_1^\star-\lambda_2^\star) [\cva(u_2^\star)-M_1] \\
&{=} J(u_1^\star)+\lambda_2^{\star} g_{M_2}(u_1^\star) + \lambda_2^\star (M_2-M_1)  +(\lambda_1^\star-\lambda_2^\star) [\cva(u_2^\star)-M_1]\nonumber\\
&{= J(u_1^\star)+\lambda_2^{\star} g_{M_1}(u_1^\star)   +(\lambda_1^\star-\lambda_2^\star) [g_{M_2}(u_2^\star) +M_2-M_1]}\nonumber\\
%&{= J(u_1^\star)+\lambda_2^{\star} g_{M_2}(u_1^\star) + \lambda_2^\star (M_2-M_1)  +(\lambda_1^\star-\lambda_2^\star) [\cva(u_2^\star)-M_1]\\
%& {= J(u_1^\star)+\lambda_2^{\star} g_{M_1}(u_1^\star) + \lambda_2^{\star} (M_1-M_2) + \lambda_2^\star (M_2-M_1)  +(\lambda_1^\star-\lambda_2^\star) [\cva(u_2^\star)-M_1]}\\
&\le   J(u_1^\star)+\lambda_2^{\star} g_{M_1}(u_1^\star) +(\lambda_1^\star-\lambda_2^\star) (M_2-M_1) +\lambda_1^{\star} g_{M_2}(u_2^\star),\label{eq:tec_inter}
\end{align}
{
where we used in the last line Equation \eqref{eq:complementary} at point $u_2^\star$.
Since Since $g_{M_1}(u_1^\star) \le  0$  and $g_{M_2}(u_2^\star) \le  0$ while $\lambda_1^\star \ge 0$ and  $\lambda_2^\star \ge 0$, we deduce from Equation \eqref{eq:tec_inter} that:
$$
J(u_1^\star) \leq J(u_1^\star) +(\lambda_1^\star-\lambda_2^\star) (M_2-M_1) .
$$
 It finally implies that 
$$
M_1<M_2 \Longrightarrow \lambda_1^\star \ge \lambda_2^\star.
$$
}
 
We now consider the penalized criterion with $\lambda_1>\lambda_2$, we verify that:
\begin{align*}
J(v_{\lambda_1}) + \lambda_1 \cva(v_{\lambda_1}) & \leq  J(v_{\lambda_2}) + \lambda_1 \cva(v_{\lambda_2}) \\
& =   J(v_{\lambda_2}) + \lambda_2 \cva(v_{\lambda_2}) +  (\lambda_1-\lambda_2)\cva(v_{\lambda_2}) \\
& \leq  J(v_{\lambda_1}) + \lambda_2 \cva(v_{\lambda_1}) +  (\lambda_1-\lambda_2)\cva(v_{\lambda_2}) \\
& =   J(v_{\lambda_1}) + \lambda_1 \cva(v_{\lambda_1})\\
&\qquad +  (\lambda_1-\lambda_2){big(\cva(v_{\lambda_2})-\cva(v_{\lambda_1})big)}.
\end{align*}
It implies that:
$$
\lambda_1 > \lambda_2 \Longrightarrow \cva(v_{\lambda_1}) \leq \cva(v_{\lambda_2}).
$$
Obviously, we also observe that $v_\lambda$ solves $\mathcal{P}_M$ for $M=\cva(v_\lambda)$, \textit{e.g.}: 
$$
v_{\lambda} =  \arg \min_{v \in Sm} \{ J(v)  : \cva(v) \leq \cva(v_{\lambda}) \}.$$
\end{proof}

\section{Proof of the theoretical results on the SMD}
\label{app:proofs_smd}
\subsection{Almost sure convergence of the algorithm}
Below, we will use some standard results that are valid for  any Bregman divergence $\DP$, whose statements are given below. We refer to \cite{Nemirovski:1983} for further details.
% The three points lemma, whose statement is given below, which is valid for

\begin{lemma}[Three points lemma]\label{lem:3points}
For any triple of points $(x,y,z)$, one has:
$$
\DP(x,z)=\DP(x,y)+\DP(y,z)-\langle \nabla \Phi(z)-\nabla \Phi(y),x-y\rangle.
$$
\end{lemma}

\begin{lemma}[Gradient of the Bregman divergence]\label{lem:grad_dp}
For any pair of points $(x,y)$, one has:
$$
\nabla_x \DP(x,y) = \nabla \Phi(x)-\nabla \Phi(y).
$$
\end{lemma}

With the help of these lemmas, we derive the proof of the almost sure convergence result stated in Theorem \ref{theo:bias_SMD_independent}. The main issue generated by our biased simulation setting is that $\hg$ involved in Equations \eqref{eq:partial_u}, \eqref{eq:partial_theta} and \eqref{def:bias_sub_diff} acts non-linearly on $\un_{\langle \hat{Z}^{k+1},U_k \rangle \ge \theta_k}$ in our stochastic approximation term, which lead to a significant amount of theoretical difficulties.

\begin{proof}[Proof of Theorem \ref{theo:bias_SMD_independent}]
Let us recall that throughout the proof, the position of the algorithm is $X_k=(U_k,\theta_k)$ and that we denote by $x=(u,\theta)$ a generic point of the set $\mathcal{X}=\Db_m\times\mathbb{R}$.\\

\paragraph{Proof of $i)$.}
The proof is divided into three steps. We first write a biased descent inequality, and then control the size of the bias induced by our model before using  the Robbins Siegmund Theorem.\\

\noindent \underline{\textit{Step 1: Biased descent inequality.}} The main point here is to obtain a recursive inequality on $\DP(x^\star_{\lambda}, X_k)$, where $x^\star_{\lambda}$ is a point that minimizes $p_\lambda$.
Our starting point is the definition of $X_{k+1}$:
$$
X_{k+1} = \arg \min_{x \in \mathcal{X}} 
\bigg\{\langle \hg_{k+1},x-X_k\rangle +\frac{ \DP(x,X_k)}{\eta_{k+1} }\bigg\},
$$
where we recall that $\hg_{k+1}$ is the stochastic approximation of the sub-gradients: 
\begin{equation*}%}\label{def:bias_sub_diff}
\begin{cases}
\hg_{k+1,1}& = - \hZ^{k+1} + \frac{\lambda}{1-\alpha} \hZ^{k+1} \un_{\langle\hZ^{k+1}, U_k\rangle \ge \theta_k},\vspace{1em}\\ 
\hg_{k+1,2} &= 1- \frac{1}{1-\alpha}\un_{\langle\hZ^{k+1}, U_k\rangle \ge \theta_k}.
\end{cases}
\end{equation*}
The first order condition entails that:
$$
\forall x \in \mathcal{X} \qquad \eta_{k+1} \langle \hg_{k+1},x-x_{k+1}\rangle + \langle \nabla_x \DP(X_{k+1},X_k),x-X_{k+1}\rangle  \ge 0.
$$
Using Lemma \ref{lem:grad_dp} on the second term, we deduce that:
$$
\forall x \in \mathcal{X} \qquad   \langle \nabla \Phi(X_{k+1}) -\nabla \Phi(X_{k})  ,x-X_{k+1} \rangle  \ge \eta_{k+1} \langle  \hg_{k+1} , X_{k+1}-x \rangle.
$$
We then apply Lemma \ref{lem:3points} with $y=X_{k+1}$ and $z=X_{k}$
and obtain that:
\begin{align*}
\eta_{k+1} \langle  \hg_{k+1} , X_{k+1}-x \rangle &\leq  \langle \nabla \Phi(X_{k+1})-\nabla \Phi(X_{k}),x-X_{k+1}\rangle \\
& \le \DP(x,X_{k})-\DP(x,X_{k+1})-\DP(X_{k+1},X_{k})
%& = \DP(x,X_k)-\DP(x,X_{k+1})-\DP(X_{k+1},X_k).
\end{align*}
Using the strong convexity of $\DP$ stated in Inequality \eqref{eq:rho_convex}, we deduce that:
$$
\eta_{k+1} \langle  \hg_{k+1} , X_{k+1}-x \rangle  \leq \DP(x,X_k)-\DP(x,X_{k+1}) -(\theta_{k+1}-\theta_k)^2 - \frac{1}{2} \|U_{k+1}-U_k\|^2,
$$
which can be written as:
\begin{align}
 \DP(x,X_{k+1})  \leq& \DP(x,X_k)-(\theta_{k+1}-\theta_k)^2 - \frac{1}{2} \|U_{k+1}-U_k\|^2 -\eta_{k+1} \langle  \hg_{k+1} , X_{k+1}-x \rangle\nonumber \\
 \leq& \DP(x,X_k)-(\theta_{k+1}-\theta_k)^2 - \frac{1}{2} \|U_{k+1}-U_k\|^2 \nonumber\\
 &  -\eta_{k+1} \langle  \hg_{k+1} , X_{k}-x \rangle-\eta_{k+1} \langle  \hg_{k+1} , X_{k+1}-X_k \rangle \label{eq:descente}\end{align}
Let us first handle the last term of the right hand side:
$$\eta_{k+1}\langle  \hg_{k+1} , X_{k+1}-X_k \rangle =  \eta_{k+1} \hg_{k+1,2} (\theta_{k+1}-\theta_k) + \eta_{k+1}
\langle  \hg_{k+1,1} , U_{k+1}-U_k \rangle.$$
The Young inequality $|ab| \leq \frac{a^2}{2c}+\frac{c b^2}{2}$ leads to:
\begin{align}
 \big| \eta_{k+1} \hg_{k+1,2} (\theta_{k+1}-\theta_k) \big| &\leq 
\frac{\eta_{k+1}^2 \hg_{k+1,2}^2}{4}  + (\theta_{k+1}-\theta_k)^2 \nonumber\\
& \leq  C_\alpha \eta_{k+1}^2  + (\theta_{k+1}-\theta_k)^2.\label{eq:young_1}
\end{align}
where in the last line we use that
associated with $|\hg_{k+1,2}|\leq \max(1,\frac{\alpha}{1-\alpha})$ (see Equation \eqref{def:bias_sub_diff}).\\
In the same way, we also have:
\begin{equation}
\big| \eta_{k+1} \langle \hg_{k+1,1} , U_{k+1}-U_k \rangle \big| \leq \frac{ \eta_{k+1}^2 \|\hg_{k+1,1}\|^2}{2 } + \frac{1}{2} \|U_{k+1}-U_k\|^2. \label{eq:young_2}
\end{equation}

%%%%%%%%%%%%%%%%%%%%%%%%

We use Equations \eqref{eq:young_1} and \eqref{eq:young_2} in Inequality \eqref{eq:descente} and obtain that:
\begin{align}\label{eq:descente_stochastique}
 \DP(x,X_{k+1}) & \leq \DP(x,X_k)+ \eta_{k+1}^2 \big( C_\alpha + \|\hg_{k+1,1}\|^2 \big)- \eta_{k+1} \langle  \hg_{k+1} , X_{k}-x \rangle.
\end{align}
We need to handle the biased drift that comes from the sampled random variables $Z_{k+1}$ involved in $\hg_{k+1}$. 
We write:
%$$\hg_{k+1}=\nabla p_\lambda(X_k) + \underbrace{( \mathbb{E}[\hg_{k+1} \, \vert \, \mathcal{F}_k] -\nabla p_\lambda(X_k))}_{:=\mathfrak{b}_{k+1}} + \underbrace{(\hg_{k+1}-\mathbb{E}[\hg_{k+1} \, \vert \, \mathcal{F}_k] )}_{:=\Delta M_{k+1}}.$$
\begin{align*}
\hg_{k+1}&=
\nabla p_\lambda(X_k) +( \mathbb{E}[\hg_{k+1} \, \vert \, \mathcal{F}_k] -\nabla p_\lambda(X_k)) + (\hg_{k+1}-\mathbb{E}[\hg_{k+1} \, \vert \, \mathcal{F}_k] )\\
&=\nabla p_\lambda(X_k) - \mathfrak{b}_{k+1} +\Delta M_{k+1},
\end{align*} 
with $\Delta M_{k+1}:=\hg_{k+1}-\mathbb{E}[\hg_{k+1} \, \vert \, \mathcal{F}_k] $ is a martingale increment and $\mathfrak{b}_{k+1}$ stands for the bias:
\begin{align}
&\mathfrak{b}_{k+1}  =
 \nabla p_\lambda(U_k,\theta_k)-\mathbb{E}[\hg_{k+1} \, \vert \, \mathcal{F}_k] \nonumber \\
 &=
\begin{pmatrix} 
\mathbb{E}[\hZ^{k+1}\,\vert \, \mathcal{F}_k] -\mathbb{E}[Z]+ \frac{\lambda}{1-\alpha} \big(\mathbb{E}[
Z \un_{\langle Z, U_k\rangle \ge \theta_k} ]-\PE[\hZ^{k+1} \un_{\langle\hZ^{k+1}, U_k\rangle \ge \theta_k} \,\vert \, \mathcal{F}_k] \big)\\
\mathbb{P}(\langle Z, U_k\rangle \ge \theta_k)  -
{\mathbb{P}(\langle\hZ^{k+1}, U_k\rangle \ge \theta_k \vert \mathcal{F}_k)}
% \mathbb{E}[ \un_{\langle\hZ^{k+1}, U_k\rangle \ge \theta_k} \vert \mathcal{F}_k] 
\end{pmatrix}.\label{def:h_k}
\end{align}
Now, choosing 
$x=x^\star_{\lambda}$ and using this decomposition into Equation \eqref{eq:descente_stochastique}, we obtain the main descent inequality, which will be the cornerstone of our analysis:
\begin{align}
\DP(x^\star_{\lambda},X_{k+1})  \leq & \DP(x^\star_{\lambda},X_k) - \eta_{k+1} \langle \nabla p_{\lambda}(X_k), X_k-x^\star_{\lambda} \rangle  - \eta_{k+1} \langle \mathfrak{b}_{k+1} ,X_k-x^\star_{\lambda}\rangle\nonumber\\
&- \eta_{k+1} \langle \Delta M_{k+1}, X_k-x^\star_{\lambda} \rangle + \eta_{k+1}^2 \Big( \frac{1}{4} + \|\hg_{k+1,1}\|^2 \Big). \label{eq:step1_general}
\end{align}
The Cauchy-Schwarz inequality entails:
\begin{equation}\label{eq:step1_cs}
\eta_{k+1} |\langle \mathfrak{b}_{k+1} , X_k-x^\star_{\lambda}\rangle| \leq \eta_{k+1} \|\mathfrak{b}_{k+1}\| \|X_k-x^\star_{\lambda}\|.
\end{equation}
Now,   the Young inequality and Equation \eqref{eq:rho_convex} implies that:
\begin{equation}\label{eq:step1_bete_mais_utile}
\|X_k-x^\star_{\lambda}\| \leq \frac{1+\|X_k-x^\star_{\lambda}\|^2}{2} \leq \frac{1}{2}+2 \DP(x^\star_{\lambda},X_{k}).
\end{equation}
We then plug Equations \eqref{eq:step1_bete_mais_utile} and \eqref{eq:step1_cs} into \eqref{eq:step1_general} and obtain that:
\begin{align}
\DP&(x^\star_{\lambda},X_{k+1})  \leq \DP(x^\star_{\lambda},X_{k}) \big(1+2 \eta_{k+1} \|\mathfrak{b}_{k+1}\|\big) - \eta_{k+1} \langle \nabla p_\lambda(X_k), X_k-x^\star_{\lambda}\rangle \nonumber\\
&- \eta_{k+1} \langle \Delta M_{k+1}, X_k-x^\star_{\lambda} \rangle + \eta_{k+1}^2 \Big(\frac{1}{4} +  \|\hg_{k+1,1}\|^2 \Big) + \frac{1}{2} \eta_{k+1}\|\mathfrak{b}_{k+1}\| . \label{eq:inegalite_step1}
\end{align}

\noindent \underline{\textit{Step 2: Bias upper bound.}} 
We are led to estimate the size of $\|\mathfrak{b}_{k+1}\|$.  For this purpose, we introduce the Kolmogorov distance $d_{Kol}$, defined on real distributions and defined by:
$$
d_{Kol}(\mu,\nu) = \sup_{a \in \mathbb{R}} \big|\mu(]-\infty,a])-\nu(]-\infty,a])\big|.
$$
The Kolmogorov and Wasserstein distance for real distributions (see \textit{e.g.} Theorem 3.1 of \cite{Shao}) are related as follows:
$$
d_{Kol}(\mu,\nu) \leq 2 \sqrt{\mathcal{W}_1(\mu,\nu)}.
$$

Using the expression of $\mathfrak{b}_{k+1}=(\mathfrak{b}_{k+1, 1},\mathfrak{b}_{k+1,2})$ in \eqref{def:h_k}, its second coordinates verifies :
\begin{align*}
{\mathbb{E}}[\|\mathfrak{b}_{k+1,2}\| ]&\le d_{Kol}\big(\mathcal{L}(\langle\hZ^{k+1}, {U_k}\rangle,\mathcal{L}(\langle Z,{U_k}\rangle )\big)\\
& \leq 2 \sqrt{\mathcal{W}_1\big(\mathcal{L}(\langle\hZ^{k+1}, {U_k}\rangle),\mathcal{L}(\langle Z,{U_k}\rangle ) \big)}.
\end{align*}
Using now the variational characterization of the $\mathcal{W}_1$ distance on $1$-Lipschitz function, and using that for any $k$, $U_k \in \Db_m$, we have:
$\|U_k\|^2 \leq \|U_k\|_1^2 = 1$, we then deduce that:
$$
\sup_{u \in \Db_m} \mathcal{W}_1\big(\mathcal{L}(\langle\hZ^{k+1}, {U_k}\rangle),\mathcal{L}(\langle Z,{U_k}\rangle ) \big) \leq 
\mathcal{W}_1(\mathcal{L}(\hZ^{k+1}),\mathcal{L}(Z)).
$$
We then obtain from Assumption \ref{hyp:biais} Equation \eqref{def:bias_simulation1} that:
$$
{\mathbb{E}}[\|\mathfrak{b}_{k+1,2}\|] \leq 2 \sqrt{\delta_{k+1}}.
$$
We are led to study an upper bound of $\|\mathfrak{b}_{k+1,1}\|$. We observe that:
\begin{align*}
{\mathbb{E}}[\|\mathfrak{b}_{k+1,1}\|] & \leq \| \mathbb{E}[\hZ^{k+1}-Z]\| +\frac{\lambda}{1-\alpha} \big\| 
\mathbb{E}[\hZ^{k+1}\un_{\langle\hZ^{k+1},{U_k} \rangle\ge \theta_k}-Z
\un_{\langle Z,{U_k}\rangle \ge \theta_k}]
\big\| \\
& \leq \delta_{k+1} +\frac{\lambda}{1-\alpha} \big\| 
\mathbb{E}[\hZ^{k+1}\un_{\langle\hZ^{k+1},{U_k}\rangle \ge \theta_k}-Z
\un_{\langle Z,{U_k}\rangle \ge \theta_k}]
\big\|.
\end{align*}
We simply observe that in this case, the second term of the right hand side of the previous inequality is upper bounded by $\upsilon_{k+1}$ thanks to Assumption \ref{hyp:biais} Equation \eqref{def:bias_simulation2}. We deduce that: 
\begin{equation}\label{eq:borne_h_k}
\PE[\|\mathfrak{b}_{k+1}\| ]\le{\mathbb{E}} [\|\mathfrak{b}_{k+1,1}\|]+{\mathbb{E}}[\|\mathfrak{b}_{k+1,2}\|] \leq 2 \sqrt{\delta_{k+1}} + \delta_{k+1} + \frac{\lambda}{1-\alpha} \upsilon_{k+1}.
\end{equation}

\noindent
\underline{\textit{Step 3: Conclusion of the proof.}}
We apply the Robbins-Siegmund Lemma to Equation \eqref{eq:inegalite_step1}: we observe that the compactness of $\Db_m$ and the existence of a second order moment for $Z$ leads to the existence of $C>0$ such that:
\begin{equation}\label{second:moment:gk}
\mathbb{E}[\|\hg_{k+1}\|^2] \leq C < + \infty.
\end{equation}

 Finally, from Equation \eqref{eq:borne_h_k} and our assumptions on $(\delta_{k+1})_{k \ge 0}$ and $(\upsilon_{k+1})_{k \ge 0}$, one has:
\begin{equation}
\label{eq:conv_serie1}
 \sum_{k \ge 0} \eta_{k+1} \mathbb{E}[\|\mathfrak{b}_{k+1}\|] < + \infty.
\end{equation}
From Equation \eqref{eq:conv_serie1},we deduce the existence of a second order moment and since $
 \eta_{k+1} \langle \Delta M_{k+1}, X_k-x^\star_{\lambda} \rangle$ is a $\mathcal{F}_{k}$ Martingale-increment, the Robbins-Siegmund Lemma (see \textit{e.g.} \cite{RobbinsSiegmund1971}) implies that:
$$\DP(x^\star_{\lambda},X_k) \underset{k\to\infty}{\overset{a.s}{\longrightarrow}}D_{\infty} \in L^1 \quad \text{and} \quad \sum_{k \ge 0} \eta_{k+1} \langle \nabla  p_\lambda(X_k),X_k-x^\star_{\lambda}\rangle< + \infty \quad  a.s.$$
This ends the proof of $i)$. \hfill $\diamond$

\paragraph{Proof of $ii)$} This result is almost straightforward starting from $i)$. We verify from $\DP(x^\star_{\lambda},X_k) \underset{k\to\infty}{\overset{a.s}{\longrightarrow}}D_{\infty} \in L^1$ that $(X_k)_{k \ge 1}$ is an almost surely bounded sequence. Hence, almost surely the Cesaro average sequence $(\bar{X}^\eta_k)_{k \ge 0}$ is a Cauchy sequence and therefore converges towards $\bar{X}^\eta_{\infty}$.
We then use the convexity of $p_\lambda$ and observe that
\begin{align*}
p_{\lambda}(\bar{X}^\eta_k) - \min(p_{\lambda})& \leq \sum_{i=0}^k \bigg(\frac{\eta_i}{\displaystyle\sum_{j=0}^k \eta_j}\bigg) p_{\lambda}(X_k) - \min(p_\lambda)\\
& = \frac{\displaystyle\sum_{i=0}^k \eta_i (p_\lambda(X_k)-\min(p_\lambda))}{\displaystyle\sum_{j=0}^k \eta_j} \longrightarrow 0 \quad a.s.
\end{align*}
The continuity of $p_\lambda$ then implies that almost surely:
$$
\bar{X}^\eta_k \longrightarrow \bar{X}^\eta_{\infty} \quad \text{and} \quad 
p_\lambda(\bar{X}^\eta_{\infty}) = \min(p_{\lambda}). 
$$
$\hfill \diamond$

\paragraph{Proof of $iii)$} { The proof of this last point is more intricate and needs the introduction of more sophisticated tools of dynamical systems to obtain the convergence of the \textit{overall} sequence $(X_k)_{k \ge 0}$ towards a point of $\arg\min (p_{\lambda})$. }

\noindent \underline{Step 1: continuous time trajectory.}
We follow the roadmap of \cite{benaim_LN} that introduces the asymptotic pseudo-trajectories for stochastic algorithms, and \cite{Mertikopoulos} that adapts these tools to the mirror descent.\\
We introduce the mirror map $Q$ and the Fenchel conjugate $\Phi^*$ defined on $\in \Db_m \times \rset$ by:
\begin{equation*}
\label{def:mirror_map}
Q(y) = \displaystyle\arg\max_{x \in \Db_m \times \rset} \langle y,x\rangle - \Phi(x) \qquad \text{and} \qquad \Phi^*(y) = \max_{x \in \Db_m \times \rset} \langle y,x\rangle - \Phi(x) .
\end{equation*}
Note that the Fenchel coupling verifies that for any $(x,y)\in \{\Db_m \times \rset \}^2$:
\begin{equation}
\label{eq:Fenchel_Dphi}
\Phi(x)+\Phi^*(y)-\langle y,x \rangle = D_{\Phi}(x,Q(y)).
\end{equation}
The Mirror Descent ordinary differential equation is then  defined as:
\begin{equation}\label{eq:md_ode}
\begin{cases}
\dot{y} &= - \nabla p_{\lambda}(x)\\
x&=Q(y)
\end{cases}.
\end{equation}
From Equation \eqref{eq:rho_convex}, we know that $\Phi$ is $1/2$ strongly convex, which implies first that $Q$ is a $2$-Lipschitz continuous function (see Proposition 2.2 of \cite{Mertikopoulos}). Second, $-\nabla p_{\lambda}$ is a continuous bounded function, which implies with the Cauchy-Arzela-Peano theorem that any solution of the Cauchy problem always exists.\footnote{It is not clear that $\nabla p_{\lambda}$ is a Lipschitz continuous function, which prevents from the direct use of Theorem 2.4 of \cite{Mertikopoulos}. Indeed a such Lipschitz property   depends on the distribution of the random variable $Z$. Nevertheless, the nature of the gradient dynamics prevents the explosion of the solutions, which compactifies the trajectories and permits to by-pass the uncertainty on this Lipschitz property.}
We consider any $x^\star$ any minimizer of $p_{\lambda}$ and $(x_t,y_t)$ a solution of \eqref{eq:md_ode}. Following the arguments of Nemirovski and Yudin \cite{Nemirovski:1983}, we define the Fenchel coupling as:
$$
V_{x^\star}(t) = \Phi(x^\star)+\Phi^*(y_t)-\langle y_t,x^\star\rangle.
$$
Note that using \eqref{eq:Fenchel_Dphi}, $V_{x^\star}(t)=D_\phi(x^\star, y_t)$.
The Fenchel inequality yields $V_{x^\star}(t) \ge 0$ for any time $t$.
We observe that :
$$
V_{x^\star}'(t) = - \langle \nabla p_\lambda(x_t),x_t-x^\star\rangle \leq p_{\lambda}(x^\star)-p_{\lambda}(x_t)\le0
$$
such that $V_{x^\star}$ is non-increasing with time.
In turn, it implies the compactness of any trajectory (and the convergence with a $O(t^{-1})$ of the value of $p_{\lambda}(\bar{x}_t)$ towards $\min(f)$ where $\bar{x}_t$ is the Cesaro average of the trajectory).
%$$
%\exists K >0 \quad \forall t  \ge 0 \qquad \|x_t \| \leq K.
%$$
From any trajectory $(x_t)_{t \ge 0}$, we consider $x_\infty$ any adherence point. If $x_\infty$ does not belong to $\arg \min(p_{\lambda})$, then we can find a compact neighbourhood $\mathcal{V}_1$ of $x_{\infty}$ and an increasing sequence $(t_k)_{k \ge 0}$ such that $x_{t_k} \in \mathcal{V}_1$. Using the continuity of $p_{\lambda}$ we can find a second set $\mathcal{V}_2$ and a radius $r>0$ such that
$$
\forall x \in \mathcal{V}_1: \quad B(x,r) \subset \mathcal{V}_2,
$$
and for which we have:
$$
\forall x \in \mathcal{V}_2: \qquad - \langle \nabla p_{\lambda}(x),x-x^\star \rangle \leq  (\min(p_{\lambda})-p_\lambda(x)) \leq  -a <0.
$$
We use that $Q$ is $2$-Lipschitz to obtain that:
\begin{align*}
 \forall s >0: \qquad \|x_{t_k+s}-x_{t_k}\| & = \|Q(y_{t_k+s})-Q(y_{t_k})\|
\\
& \leq 2 \| y_{t_k+s}-y_{t_k}\|\\
& \leq 2 \int_{t_k}^{t_k+s} \| \dot{y}_u \| \text{d}u \\
& \leq 2 s \sup_{  x \in \Db_m \times \rset } \|\nabla p_{\lambda}(x)\|\\
& \leq 2 s  (\alpha+\mathbb{E}[\|Z\|])\big(1+\frac{\lambda}{1-\alpha}\big),
\end{align*} 
where the last line comes from Equations \eqref{eq:partial_u} and \eqref{eq:partial_theta}. We can deduce that:
$$
\forall s \leq s_r := \frac{r}{2 (\alpha+\mathbb{E}[\|Z\|])(1+\frac{\lambda}{1-\alpha})}: \qquad x_{t_k+s} \in B(x_{t_k},r) \subset \mathcal{V}_2.
$$
Finally, we write that:
$$
V_{x^\star}(t_k+s_r) - V_{x^\star}(0) \leq \sum_{j=0}^k \int_{t_j}^{t_j+s_r} V_{x^\star}'(u) \text{d}u \leq - k \times (a  s_r)   \longrightarrow - \infty \quad \text{as} \quad k \longrightarrow + \infty. 
$$
This contradicts the compactness of the trajectories and leads to a contradiction. We deduce that $x_\infty \in \arg \min(p_{\lambda})$ and that any adherence point of $(x_t)_{t \ge 0}$ belongs to $\arg \min(p_{\lambda})$.

Finally, we prove that the whole trajectory converges towards $x_{\infty}$. The construction of $V_{x^\star}$ holds for any $x^\star$ in $\arg \min \{ p_{\lambda}\}$, in particular for $x^\star=x_{\infty}$. For this particular choice, we verify that:
$$
\forall t \ge t_k \qquad 0 \leq V_{x_{\infty}}(t) \leq V_{x_{\infty}}(t_k),
$$
because $V$ is non-increasing and larger than $0$. \\
Since  by construction $x_{t_k}$ converges to $x^\star$, then $D_{\Phi}(x_\infty,x_{t_k})=V_{x^\star}(t_k) \longrightarrow 0$ as $k \longrightarrow + \infty$, which in turn implies that $V_{x_{\infty}}(t) \longrightarrow 0$ as $t \longrightarrow + \infty$. We then conclude that the whole trajectory converges towards $x_{\infty}$.

\noindent \underline{Step 2: stochastic algorithm convergence.}
To derive the asymptotic behaviour of the SMD, we follow the argument of Proposition 4.1 in \cite{benaim_LN} . We consider $(X_k,Y_k)$ that satisfies the stochastic recursion
$$
\begin{cases}
Y_{k+1} &= Y_k - \eta_{k+1} \hat{g}_{k+1}\\
X_{k+1} &=Q(Y_{k+1})
\end{cases}
$$
We then define the continuous time affine interpolation of $(X_k,Y_k)_{k \ge 1}$ using the natural time-scale $\tau_{0}=0$ and $\tau_{k+1}=\tau_k+\eta_{k+1}$ and introduce $(\hat{X}_t,\hat{Y}_t)_{t \ge 0}$ the stochastic process defined as
$$
\forall t \in [\tau_{k},\tau_{k+1}] \qquad (\hat{X}_t,\hat{Y}_t) =
(X_k,Y_k) +  \frac{t-\tau_k}{\tau_{k+1}-\tau_k}  \Big((X_{k+1},Y_{k+1})-(X_k,Y_k)\Big).
$$
We still use the decomposition derived from \eqref{def:h_k}:
$$
\hat{g}_{k+1} = \nabla p_{\lambda}(X_k) + \Delta M_{k+1} + \mathfrak{b}_{k+1},
$$
where the martingale increment $\Delta M_{k+1}$ satisfies (see \eqref{second:moment:gk}):
$$
\sup_{k \ge 0} \mathbb{E}[\|\Delta M_{k+1}\|^2] < + \infty,
$$
and the bias term satisfies under Equation \eqref{eq:borne_h_k} and
assumption of $iii)$ that:
$$
\sum_{k \ge 0} \mathbb{E}[ \|\mathfrak{b}_{k+1}\|] \leq 
\sum_{k \ge 0} \Big(2 \sqrt{\delta_{k+1}}+\delta_{k+1}+\frac{\lambda}{1-\alpha}\upsilon_{k+1}\Big) \leq   + \infty.
$$
Then, the Markov inequality and the Borel-Cantelli lemma yields
$$
\mathfrak{b}_k \longrightarrow 0 \quad \text{a.s.} \quad \text{as} \quad k \longrightarrow + \infty.
$$
We then use Remark 4.5 and Proposition 4.2 of \cite{benaim_LN} and deduce that Assumption A1 of Proposition 4.1 \cite{benaim_LN} holds. In the meantime, Assumption A2 of Proposition 4.1 \cite{benaim_LN} is valid from the proof of $i)$ and the convergence of $D_\Phi(x_\lambda^\star,X_k)$. Since $Q$ is a $2$-Lipschitz function, we then deduce that $(\hat{X}_t,\hat{Y}_t)_{t \ge 0}$ is an aymptotic pseudo-trajectory of the deterministic mirror descent o.d.e. stated in \eqref{eq:md_ode}. In particular, $(X_t)_{t \ge 0}$ converges using Step 1 of $iii)$, which concludes the proof of $iii)$. \hfill  
\end{proof}

\subsection{Proof of the non-asymptotic behaviour (Theorem \ref{theo:bias_SMD_finite}).\label{app:smd_finite}}

\paragraph{Step 1: control of $\PE[\DP(x^\star_{\lambda},X_{k})]$.}
%We obtained in the proof of Theorem \ref{theo:bias_SMD_independent} that $\DP(x^\star_{\lambda},X_{k})$ to $0$ almost surely, however, the study of non asymptotic properties of the algorithm requires control of the mean behavior of this divergence.
For the sake of convenience, we denote by:
 $$\PE[\DP(x^\star_{\lambda},X_{k})] = \mathbf{D}_\Phi^k.$$
Taking the expectation in \eqref{eq:inegalite_step1} and using $\langle \nabla p_\lambda(X_k),X_k-x^\star_{\lambda}\rangle >0$, we get:
\begin{align*}
\mathbf{D}_\Phi^{k+1}  &\leq \mathbf{D}_\Phi^{k}  +2 \eta_{k+1} \PE\big[\DP(x^\star_{\lambda},X_{k})\|\mathfrak{b}_{k+1}\|\big] \\
&\qquad+ \eta_{k+1}^2 \Big( \frac{1}{4} +  \PE[ \|\hg_{k+1,1}\|^2 ]\Big) + \frac{1}{2} \eta_{k+1}\PE[\|\mathfrak{b}_{k+1}\|]\\
&\leq \mathbf{D}_\Phi^{k} \Big(1+2 \eta_{k+1} \big(2 \sqrt{\delta_{k+1}} + \delta_{k+1} + \frac{\lambda \upsilon_{k+1}}{1-\alpha} \big)\Big) \\
&\qquad + \eta_{k+1}^2 \Big( \frac{1}{4} +  \PE [\|\hg_{k+1,1}\|^2] \Big) + \frac{1}{2} \eta_{k+1}\big(2 \sqrt{\delta_{k+1}} + \delta_{k+1} + \frac{\lambda \upsilon_{k+1}}{1-\alpha}\big),
\end{align*}
where we used \eqref{eq:borne_h_k} and a conditional expectation argument. The inequality can be written as:
$$
\mathbf{D}_\Phi^{k+1} \le \mathbf{D}_\Phi^k (1+ a_{k+1}) +b_{k+1},
$$
using that $\PE \|\hg_{k+1,1}\|^2\le M$, a $C>0$ exists such that:
\begin{equation}
\label{def:ab}
\begin{cases}
&a_{k+1}=2 \eta_{k+1} \big(2 \sqrt{\delta_{k+1}} + \delta_{k+1} + \frac{\lambda \upsilon_{k+1}}{1-\alpha}\big) \\
& b_{k+1}=  C\Big(\eta_{k+1}^2 + \eta_{k+1}\big(2 \sqrt{\delta_{k+1}} + \delta_{k+1} + \frac{\lambda \upsilon_{k+1}}{1-\alpha}\big) \Big)\end{cases}
\end{equation}

Now, denote:
$$
\mathcal{V}_k=\frac{\mathbf{D}_\Phi^k}{\prod_{i=1}^{k} (1+a_i )} - \sum_{j=0}^{k} \frac{b_j}{\prod_{i=1}^j (1+a_i)},
$$
 we verify that:
\begin{align*}
\mathcal{V}_{k+1}&=\frac{\mathbf{D}_\Phi^{k+1}}{\prod_{i=1}^{k+1} (1+a_i )} - \sum_{j=1}^{k+1} \frac{b_j}{\prod_{i=1}^j (1+a_i)} \\
&\le \frac{ \mathbf{D}_\Phi^k (1+ a_{k+1}) +b_{k+1}}{\prod_{i=1}^{k+1} (1+a_i )} - \sum_{j=1}^{k+1} \frac{b_j}{\prod_{i=1}^j (1+a_i)}\\
&\le  \frac{ \mathbf{D}_\Phi^k}{\prod_{i=1}^{k} (1+a_i )} + \frac{b_{k+1}}{\prod_{i=1}^{k+1} (1+a_i )}- \sum_{j=1}^{k+1} \frac{b_j}{\prod_{i=1}^j (1+a_i)}\\
&\le \frac{ \mathbf{D}_\Phi^k}{\prod_{i=1}^{k} (1+a_i )} - \sum_{j=1}^{k} \frac{b_j}{\prod_{i=1}^j (1+a_i)} = \mathcal{V}_{k}.
\end{align*}
From this inequality, we deduce that:
$$
\frac{\mathbf{D}_\Phi^k}{\prod_{i=1}^{k} (1+a_i )} - \sum_{j=1}^{k} \frac{b_j}{\prod_{i=1}^j (1+a_i)} \le  \mathcal{V}_{0}=\mathbf{D}_\Phi^0.
$$
Let us remark that:
\begin{align}
\mathbf{D}_{\Phi}^0 = D_{\Phi}(X_0, x^\star_\lambda) & =  \frac{(\theta_0-\theta^\star_\lambda)^2}{2} + D_{\varphi}(U_0,u^\star_\lambda)\nonumber \\
& =  \frac{(\theta_0-\va(u^\star_\lambda))^2}{2} + D_{\varphi}(U_0,u^\star_\lambda)\nonumber\\
& \leq \frac{(\theta_0-\va(u^\star_\lambda))^2}{2} + \log m  := \{\Delta_\Phi^0\}^2, \label{borne:DPhi_0}
\end{align}
where we used the complementary slackness condition (see Equation \eqref{eq:complementary} in the Appendix \ref{sec:appendix_lagrange})  and the upper bound on the entropic Bregman divergence used in \cite{Lan:2012}.
We finally deduce that:
\begin{equation}\label{eq:borne_Dphi_k}
\mathbf{D}_\Phi^k \le \mathbf{D}_\Phi^0\prod_{i=1}^{k} (1+a_i ) +  \Big(\sum_{j=1}^{k} \frac{b_j}{\prod_{i=1}^j (1+a_i)} \Big)\prod_{i=1}^{k} (1+a_i ).
\end{equation}
\paragraph{Step 2: non-asymptotic control of the position of the SMD.}
Going back again to \eqref{eq:inegalite_step1}, we can now use the convergence of $\mathcal{D}_\Phi(x^\star_\lambda, X_k)\to 0$ to get:
\begin{align*}
\eta_{k+1} \langle \nabla p(X_k), X_k-x_\star \rangle  &\leq \DP(x_\star,X_{k})\Big(1+2 \eta_{k+1} \|\mathfrak{b}_{k+1}\|\Big)-\DP(x_\star,X_{k+1})   \nonumber\\
&\qquad+ \eta_{k+1} \langle \mathbb{E}[\hg_{k+1} \, \vert \mathcal{F}_k]-\hg_{k+1}, X_k-x_\star \rangle\\
 &\qquad+ \eta_{k+1}^2 \big( \frac{1}{4} +  \|\hg_{k+1,1}\|^2 \big)
 + \frac{1}{2} \eta_{k+1}\|\mathfrak{b}_{k+1}\|.
\end{align*}
From the convexity of $p_\lambda$, we deduce that  $ \langle \nabla p_\lambda(X_k), X_k-x^\star_{\lambda}\rangle \ge p_\lambda(X_k) - p_\lambda(x^\star_{\lambda})$, and taking the expectation in both sides, we deduce that:
$$
\eta_{k+1} \big(\PE [p_\lambda(X_k)] - p_\lambda(x^\star_\lambda)  \big)\le \Big(\mathbf{D}_\Phi^k - \mathbf{D}_\Phi^{k+1} \Big) +a_{k+1}\mathbf{D}_\Phi^k + b_{k+1}.
$$
Summing these inequalities from $k=0$ to $n$, we get:
\begin{align*}
\sum_{k=0}^{n}\eta_{k+1} \big(\PE [p_\lambda(X_k) ]-p_\lambda(x^\star_\lambda) \big) 
&\le \sum_{k=0}^{n}\Big(\mathbf{D}_\Phi^k - \mathbf{D}_\Phi^{k+1} \Big) + \sum_{k=0}^{n} a_{k+1} \mathbf{D}_\Phi^{k}+ b_{k+1}\\
&\le \Big(\mathbf{D}_\Phi^0 - \mathbf{D}_\Phi^{n+1} \Big) + \sum_{k=0}^{n} a_{k+1} \mathbf{D}_\Phi^{k}+ b_{k+1}\\
&\le \mathbf{D}_\Phi^0 + \sum_{k=0}^{n} a_{k+1} \mathbf{D}_\Phi^{k}+ b_{k+1}.
\end{align*}
Following the convexity argument of $ii)$ Theorem \ref{theo:bias_SMD_independent} and the definition of $\bar{X}_n^{\eta}$ in\eqref{eq:def_cesaro_average}, we deduce that:
\begin{align}
\PE [p_\lambda (\bar{X}^\eta_n)] - p_\lambda(x^\star_\lambda) &\le \sum_{k=0}^{n}\frac{\eta_{k+1}}{\sum_{j=0}^{n}\eta_{j+1} } \Big(\PE [p_\lambda(X_k) ]-p_\lambda(x^\star_\lambda) \Big)\nonumber\\
 &\le \frac{1}{\sum_{j=0}^{n}\eta_{j+1} }\bigg(\mathbf{D}_\Phi^0+ \sum_{k=0}^{n} \Big(a_{k+1} \mathbf{D}_\Phi^k + b_{k+1}\Big)\bigg), \label{eq:maj_plambda_k}
 \end{align}
which ends the proof. \hfill$\square$

\begin{proof}
[Proof of Corollary \ref{cor:finite_horizon}]
We now consider a fixed step sequence for the algorithm and the discretization stopped at a horizon $n$. 
More precisely, we assume  $\eta_{k+1}=\eta>0$ for $0\le k\le n$ and that \eqref{eq:borne_h_k} translates:
$$\mathbb{E}[\|\mathfrak{b}_{k+1}\|] \le \omega, \quad \forall\, 0\le k\le n.$$
We aim at choosing constant values for step sequences adapted to the finite horizon $n$. 

Within this framework the step sequences $(a_{k+1})$ and $(b_{k+1})$ defined in \eqref{def:ab} become
$$a_{k+1}=2\eta\omega\quad \text{and}\quad b_{k+1}= C(\eta^2+\eta\omega).$$
Let us now see how both $\eta$ and $\omega$ should depend on $n$.
With these choices, \eqref{eq:borne_Dphi_k} becomes:
$$\mathbf{D}_\Phi^k\le \mathbf{D}_\Phi^0 (1+2\eta\omega)^k +C(\eta^2+\eta\omega)  \frac{(1+2\eta\omega)^{k}-1}{2\eta\omega}.$$

Equation \eqref{eq:maj_plambda_k} then yields :
\begin{align}
\PE[ p_\lambda (\bar{X}^\eta_n)] - p_\lambda(x^\star_\lambda)&\le \frac{1}{(n+1)\eta }\Big( \mathbf{D}_\Phi^0+ \sum_{k=0}^{n}\big( 2\eta\omega \mathbf{D}_\Phi^k + C\eta(\eta+\omega)\big)\Big) \nonumber\\
&\le  \frac{1}{(n+1)\eta }\bigg( \mathbf{D}_\Phi^0\Big(1+ 2\eta\omega \sum_{k=0}^{n}(1+2\eta\omega)^k\Big)\nonumber\\
&\qquad+ C\eta(\eta+\omega)\Big( n+1+\sum_{k=0}^{n} \big((1+2\eta\omega)^{k+1}-1\big) \Big)\bigg) \nonumber\\
&\le  \frac{\mathbf{D}_\Phi^0 }{(n+1)\eta}\Big( 1+2\eta\omega\frac{(1+2\eta\omega)^{n+1}-1}{2\eta\omega}\Big)\nonumber\\
&\qquad+ \frac{C \eta (\eta+\omega)}{ \eta  (n+1)} {\Big(\frac{(1+2\eta\omega)^{n+1}-1}{2\eta\omega}+n+1\Big)}\nonumber\\
 &\leq  \frac{\mathbf{D}_\Phi^0  e^{2 \eta \omega (n+1)}}{(n+1)\eta}+C (\eta+\omega)\Big( \frac{e^{2 \eta \omega (n+1)}}{ 2\eta \omega (n+1)} +1\Big) \label{eq:up_general}
\end{align}
where we used in the last inequality that $1+x\le e^x$. We  then  choose $(\eta,\omega)$ such that:
 $$\eta = \frac{\sqrt{\mathbf{D}_\Phi^0}}{2\sqrt{n+1}} \quad \text{and} \quad \omega =\frac{1}{\sqrt{n+1} \sqrt{\mathbf{D}_\Phi^0}}$$
 and obtain, using the fact that $\eta \omega (n+1) \leq 1/2$, that a large enough $C$ exists such that
 $$
\PE [p_\lambda(\hat{X}^\eta_n)] - p_\lambda(x^\star_\lambda) 
\le C \frac{\sqrt{\mathbf{D}_\Phi^0}}{\sqrt{n+1}} .
$$
To conclude, we plug-in the upper bound  \eqref{borne:DPhi_0} for $\sqrt{\mathbf{D}_\Phi^0}$ in the previous equations.

\noindent
Starting from \eqref{eq:up_general}, we also observe that the choice $\eta=\omega=1/\sqrt{n+1}$ yields:
$$
\PE [p_\lambda(\hat{X}^\eta_n)] - p_\lambda(x^\star_\lambda) 
\le C \frac{\mathbf{D}_\Phi^0}{\sqrt{n+1}} .
$$
\end{proof}

%%%%%%%%%%%%%%%%%%%%%%%%%%%%%%%%%%%%%%%%%%

\section{Proofs associated with the simulation of the portfolio}
\label{app:CIR}

The main goal of this section is to prove Proposition \ref{prop:wasserstein} and Proposition \ref{prop:biais}. To this aim, we first need technical properties on the drift-implicit Euler Scheme coupled with the Riemann integral approximation and on the density of the portfolio $Z$.

\subsection{Properties of the CIR process and its discretization}
\paragraph{Conditional moments of the CIR process}
We recall some useful properties satisfied by the CIR process and on some important objects that are related to the CIR introduced in Equation \eqref{eq:def_CIR}.
\begin{proposition}\label{prop:CIR_technique}
\begin{itemize} $(r_t)_{t \ge 0}$ is a Markov process and satisfies: 
\item[i)] If $2 a b > \sigma_0^2$, then $(r_t)_{t \ge 0}$ remains strictly positive almost surely.
\item[ii)] The (conditional) expectation and variance are given by:
\begin{equation}
\label{eq:ito1_CIR}
 \PE[r_{t+s} \vert  r_s] = r_s e^{-a t} + b(1-e^{- at}), 
 \end{equation}
 and 
 \begin{equation}
\label{eq:ito2_CIR}
 \mathbb{V}(r_{t+s} \vert r_s) = r_s \frac{\sigma_0^2}{a} (e^{- a t} - e^{- 2 a t}) + \frac{b \sigma_0^2}{2 a} (1-e^{- a t })^2. 
\end{equation}
\item[iii)] The CIR possesses an exponential integrability:
$$
\forall \lambda < \frac{a^2}{2 \sigma_0^2} \quad \forall t \ge 0: \qquad
\PE\Big[\exp\big( \lambda \int_0^t r_s ds \big)\Big]  \leq C(\lambda,t) < + \infty.
$$
\end{itemize}
\end{proposition}

The two first items of the previous proposition are classical, and may be found in \cite{Glasserman}. The exponential integrability of the CIR process is more recent and may be traced back to \cite{cozma} (Proposition 3.2). We emphasize that within realistic situations, the values of $a$ and $\sigma$ for the CIR process are in general linked within a reasonable ratio of $5 \leq \frac{a}{\sigma}\leq 10$, so that the limiting support of $\lambda$ is generally significantly larger than $1$.
Note also that this last result holds at any horizon time but it may be explicited for $t=1$ with a larger size of $\lambda$.

Proposition \ref{prop:CIR_technique} will be useful for deriving some important properties related to the bias of our approximation strategy.

\paragraph{Weak error associated to the implicit Euler scheme}

We recall the following statement of \cite{alfonsi:2013} proved in \cite{Dereich_neuenkirch2012}, that
 provides an upper bound of the weak error associated to the implicit Euler scheme\eqref{eq:discretisation_CIR}. 
In the previous reference, a strong error rate of order $1$ is also provided with stronger assumptions on the parameters.
\begin{proposition}[Theorem 1 in \cite{alfonsi:2013}]
\label{alfonsi:discr:CIR}
Assume that $2ab > \sigma_0^2$, then for any $p\in[1, \frac{2ab}{\sigma^2})$, a constant $K_p$ exists such that:
$$
\bigg( \PE\Big[ \underset{k \in \{0,\ldots,N\}}{\max}\big| 
\hat{r}_{ k h}  - r_{k h}\big|^p  \Big]\bigg)^{1/p} \leq K_p \sqrt{h}.
$$
\end{proposition}

\paragraph{Control of the error in the approximation of the integral of the CIR}
Recall the definition of $\Delta_h $ in \eqref{eq:delta_N}
\begin{equation*}
%\label{eq:delta_N}
\Delta_h = \int_0^1 r_s ds - \frac{1}{N} \sum_{k=0}^{N-1} \hat{r}_{kh }.
\end{equation*}
Let us now prove two technical lemmas first on the estimation of $\PE(\Delta_h^2)$ then on exponential moments.

\begin{lemma}\label{lem:moment_delta_fini}Assume that the CIR parameters satisfy $ab>\sigma^2$, then there exists a constant $C$ such that, for any $h=N^{-1}>0$, 
$$
\PE[ |\Delta_h|^2 ] \le C h.
$$

\end{lemma}

\begin{proof}[Proof of Lemma \ref{lem:moment_delta_fini}]
%\lorick{\\Je fais la preuve dans le cas $p=2$ pour minimiser les nouvelles maths à taper.}

We   bound $\Delta_h^2$ using first the Jensen inequality and then the  Cauchy-Schwarz inequality, we obtain that:

\begin{align*}
{\Delta}_h^2 = \Big( \frac{1}{N} N\sum_{k=0}^{N-1}  \int_{kh}^{(k+1)h} (r_s - \hat{r}_{kh}) \text{d} s \Big)^2 
& \leq \frac{1}{N} \sum_{k=1}^{N} \Big(N \int_{kh}^{(k+1)h} (r_s - \hat{r}_{kh}) \text{d}s \Big)^2 \\
& \leq \sum_{k=1}^{N} \int_{kh}^{(k+1)h} (r_s - \hat{r}_{kh})^2 \text{d}s.  
\end{align*}
Now taking the expectation and decomposing between the approximation error from the numerical scheme at time $kh$ and the regularity of the CIR process we obtain that:
\begin{align}
\PE[{\Delta}_h^2]
&\leq 2 \sum_{k=1}^{N} \Big( \int_{(k-1)h}^{kh} \PE[(r_s - r_{kh})^2] +\PE [(r_{kh}- \hat{r}_{kh})^2]\Big) \text{d}s \nonumber\\
&\leq 2 \sum_{k=1}^{N}  \int_{(k-1)h}^{kh} \PE [(r_s - r_{kh})^2 ] \text{d}s +2h \sum_{k=0}^{N-1} \PE [(r_{kh}- \hat{r}_{kh})^2] \label{decomposition:EDelta}.
\end{align}

Let us start with the first term in the right hand side of the equation above.  We use the explicit expression of conditional expectation and variances of the CIR $r_t $ stated in $ii)$ of Proposition \ref{prop:CIR_technique}. Thus, for any $s=kh+u$ with $0\le u \le h$, the bias-variance decomposition associated with Equations \eqref{eq:ito1_CIR} and \eqref{eq:ito2_CIR} leads to:
\begin{align*}
\PE\big[ (r_s - r_{kh})^2 \, \vert r_{kh}\big] &= \big(\PE[r_s \, \vert r_{kh}] - r_{kh}\big)^2 + \mathbb{V}[r_s \, \vert r_{kh}]\\
& = (b-r_{kh})^2(1-e^{-a u})^2 + \frac{\sigma^2}{a} r_{kh} e^{- a s} (1-e^{- au}) \\
&\qquad  + b \frac{\sigma^2}{2 a}(1-e^{- a u})^2.
\end{align*}
We then use again Equation \eqref{eq:ito1_CIR} to obtain a recursive expression of $\PE(r_{kh})$: 
\begin{align*}
\PE[r_{(k+1)h}]
&=\PE\big[\PE[r_{(k+1)h}\vert r_{kh}]\big]=\PE[r_{kh}) e^{-ah}+b(1-e^{-ah}]\\
&=e^{-a(k+1)h}(\PE[r_0]-b) +b.
\end{align*}
Therefore 
$$\sum_{k=1}^{N}\PE[r_{kh}]=bN+(\PE[r_0]-b) \sum_{k=1}^{N}e^{-akh}=bN+(\PE[r_0]-b)\frac{1-e^{-a}}{1-e^{-ah}}.$$
From Equation \eqref{eq:ito2_CIR} and a conditional expectation argument  we deduce that:
$$\PE [(b-r_{(k+1)h})^2]=\PE[r_{kh}] \frac{\sigma^2}{a} (e^{-ah} - e^{- 2 ah }) + \frac{b \sigma^2}{2 a} (1-e^{- ah })^2.$$
We then obtain that:
\begin{align*}
\sum_{k=1}^{N}  \int_{(k-1)h}^{kh} \PE[(r_s &- r_{kh})^2] \text{d} s
=\sum_{k=1}^{N}\PE [(b-r_{kh})^2 ]\int_{0}^h (1-e^{-a u})^2  \text{d} u
\\
&+ \frac{\sigma^2}{a} \sum_{k=1}^{N}\PE[r_{kh}] \int_{0}^h  e^{- au} (1-e^{-a u})  \text{d} u+ b \frac{\sigma^2}{2a} N\int_0^h (1-e^{-a u})^2  \text{d} u.
\end{align*}
We then compute that:
$$
\int_{0}^h (1-e^{-a u})^2  \text{d} u = \frac{a^2 h^3}{3}  + O_{a}(h^4)$$ and$$
\int_{0}^h  e^{- au} (1-e^{-a u})  \text{d} u =  \frac{a h^2}{2} -  \frac{a^2 h^3}{2} +O_{a}(h^4).
$$
Finally we deduce that for some $C>0$
\begin{equation}\label{eq:Delta_h}
\sum_{k=1}^N  \int_{(k-1)h}^{kh} \PE[(r_s - r_{kh})^2] =C h+ O_{a,b}(h^2).
\end{equation}
We now turn to the second term above, that we estimate using the strong error rate of Proposition \ref{alfonsi:discr:CIR} for $p=2$.
Assuming $ab>\sigma_0^2$, we get:
$$
\PE[(r_{kh}- \hat{r}_{kh})^2] \le \PE[ \max_{1 \le k \le N}  (r_{kh}- \hat{r}_{kh})^2 ] \le C h.
$$
Since this bound is uniform in $k$, we deduce that :
\begin{equation}\label{ctrl:discr:err}
\sum_{k=1}^N \PE[(r_{kh}- \hat{r}_{kh})^2]  \le C.
\end{equation}
Combining estimates \eqref{eq:Delta_h} and \eqref{ctrl:discr:err} in equation \eqref{decomposition:EDelta} gives the conclusion.
\end{proof}

\begin{lemma}\label{lem:exp_delta}
For any choice of $q \leq 2$, if $a >2 \sqrt{2} \sigma_0$ then: 
\begin{equation}\label{TL:finie}
\sup_{N} \PE [e^{q |\Delta_h|}] < + \infty.
\end{equation}
\end{lemma}
\begin{proof}[Proof of Lemma \ref{lem:exp_delta}]
We use the Cauchy-Schwarz inequality to upper-bound the expectation:
\begin{align*}
\PE[e^{q |\Delta_h|}] &=\PE\Big[e^{q\big|\int_0^1 r_s ds - \frac{1}{N} \sum_{k=1}^{N} \hat{r}_{k h}\big|}\Big]\\
&\le \PE\Big[e^{q\big|\int_0^1 r_s ds \big|}e^{q\big|\frac{1}{N} \sum_{k=1}^{N} \hat{r}_{kh}\big|}\Big]\\
&\le \PE \Big[ e^{2q \int_{0}^{1}  r_s \text{d} s}\Big]^\frac12 \times \PE \Big[ e^{\frac{2q}{N} \sum_{k=1}^{N} \hat{r}_{kh}}\Big]^\frac12.
\end{align*}
We apply $iii)$ of Proposition \ref{prop:CIR_technique} and observe that since $a > 2 \sqrt{ 2} \sigma_0$, then $2 q \leq 4 \leq \frac{a^2}{2 \sigma_0^2}$ and:
$$
 \PE\Big[ e^{2 q \int_{0}^{1}  r_s \text{d} s}\Big]^\frac12 < + \infty.
$$

For the second part, we rely on the recursive expression for $ \hat{r}_{kh}$ stated in \eqref{alfonsi:discr:CIR}, that we bound using:
$$\forall \epsilon >0 \qquad (u+v)^2=u^2+v^2+2uv\le u^2(1+1/\epsilon)+v^2(1+\epsilon).$$
Thus 
\begin{align*}
\hat{r}_{(k+1)h}&= \bigg(
  \frac{\sqrt{\hat{r}_{kh}}+ \frac{\sigma_0}{2} \Delta B_0^{(k)}}{2(1+\frac{ah}{2})} +\sqrt{  \frac{ \big(\sqrt{\hat{r}_{kh}}+\frac{\sigma_0 }{2}\Delta B_0^{(k)}\big)^2}{4(1+\frac{ah}{2})^2}+  \frac{(4ab-\sigma_0^2)h}{8(1+\frac{ah}{2})} } 
   \bigg)^2
\\
&\le
  (1+\frac{\epsilon}{2}) \bigg(  \frac{\sqrt{\hat{r}_{kh}}+ \frac{\sigma_0}{2} \Delta B_0^{(k)}}{2(1+\frac{ah}{2})}\bigg)^2 \\
  &\qquad+ (1+\frac{1}{2\epsilon} )\bigg( \frac{ \big(\sqrt{\hat{r}_{kh}}+ \frac{\sigma_0 }{2}\Delta B_0^{(k)}\big)^2}{4(1+\frac{ah}{2})^2}+  \frac{(4ab-\sigma_0^2)h}{8(1+\frac{ah}{2})}  \bigg)\\
  &\le
  (2+\frac{\epsilon}{2}+\frac{1}{2\epsilon}) \frac{ \big(  \sqrt{\hat{r}_{kh}}+ \frac{\sigma_0}{2} \Delta B_0^{(k)}\big)^2}{4(1+\frac{ah}{2})^2} + (1+\frac{1}{2\epsilon} ) \frac{(4ab-\sigma_0^2)h}{8(1+\frac{ah}{2})}  \\
    &\le
\frac{  (2+\frac{\epsilon}{2}+\frac{1}{2\epsilon})( 1+2\tau)}{4(1+\frac{ah}{2})^2} \hat{r}_{kh}+ \frac{  (2+\frac{\epsilon}{2}
+\frac{1}{2\epsilon})( 1+\frac{1}{2\tau})}{4(1+\frac{ah}{2})^2}\frac{\sigma_0^2}{4} \{\Delta B_0^{(k)}\}^2\\
&\qquad+ (1+\frac{1}{2\epsilon} ) \frac{(4ab-\sigma_0^2)h}{8(1+\frac{ah}{2})}.
\end{align*}

Choosing $\epsilon=1$ and $\tau=\frac17<\frac{1}{6}$, we deduce that: 
\begin{align*}
\hat{r}_{(k+1)h}
    &\le
\frac{ 1}{(1+\frac{ah}{2})^2} \hat{r}_{kh}+ \frac{27}{32(1+\frac{ah}{2})^2}\frac{\sigma_0^2}{4} \{\Delta B_0^{(k)}\}^2+  \frac{3(4ab-\sigma_0^2)h}{16(1+\frac{ah}{2})} 
\end{align*}
Using the independence and stationarity of the Brownian motion, we set $\xi_{k+1}$ a standard Gaussian random variable such that $\xi_{k+1} \sqrt{h} =\Delta B_0^{(k)}$, therefore:
$$
\hat{r}_{(k+1)h} \le \hat{r}_{kh}   +  \alpha \xi_{k+1}^2 + \beta$$
where 
$$\alpha =  \frac{27 }{32(1+\frac{ah}{2})^2}\frac{\sigma_0^2}{4} h \quad \text{and} \quad 
\beta =  \frac{3(4ab-\sigma_0^2)h}{16(1+\frac{ah}{2})} .
$$
A straightforward recursion yields:
\begin{align*}
\hat{r}_{(k+1)h} \le \hat{r}_{0}   +  \alpha\sum_{j=0}^{k} \xi_{j+1}^2 + \beta k.
\end{align*}
 In particular, we can use this to upper bound the sum with:
$$
\frac{1}{N}\sum_{k=1}^N \hat{r}_{kh} \le \hat{r}_{0}   +  \alpha\sum_{j=1}^{N} \xi_{j+1}^2 + \beta N.
$$
We can now bound the Laplace transform, using the independence of the $\xi_j$ and the Laplace transform of a $\chi^2$ distribution $\PE(e^{\lambda \xi^2})=(1-2\lambda)^{-1/2}, \forall \lambda < 1/2$.
\begin{align*}
\PE \big[ e^{\frac{2q}{N}\sum_{k=1}^N \hat{r}_{kh } }\big]
& \le e^{2q\hat{r}_0+ 2qN \beta} \prod_{j=1}^{N} \PE \big[  e^{ 2q\alpha \xi_j^2} \big] \\ 
&\le e^{2q\hat{r}_0+ 2qN \beta} 
  \big( \frac{1}{1-4 q \alpha N}\big)^{1/2} \\
&\le e^{2q\hat{r}_0+ 2qN \beta}  e^{2 q \alpha N},
\end{align*}
where the last line derives from $1/(1-t) \le e^{t}$.
From our choice of $\alpha$ and $\beta$, we observe that $\alpha N< +\infty$ and $\beta N< + \infty$.
We then conclude that  
$\PE \big[ e^{2q \frac{1}{N} \sum_{k=1}^N \hat{r}_{kh}}\big]^\frac12<+\infty$.
\end{proof}

\subsection{Results on the portfolio $Z(t)$}
\begin{proposition}[Density of the portfolio]\label{prop:densite_bornee}
Assume that $\Sigma$ is positive definite, then the distribution of $S=(Y_1,S_1,\ldots,S_{m'})$ at time $1$ is absolutely continuous with respect to the Lebesgue measure. Moreover, the density $p$ may be written as:
$$
p:(z_1,\ldots,z_m) \longmapsto \frac{M(z_1,\ldots,z_m)}{(\sigma_1 \ldots \sigma_{m'})  \det(L) z_1 \ldots z_m},
$$
where $M$ is a bounded continuous function, whose bound is independent from  $m$ and $\Sigma$.
\end{proposition}
\begin{proof}
Let us argue first in the case of uncorrelated Brownian motions in equations \eqref{eq:portfolio}, \textit{i.e.} when $\Sigma$ is the identity matrix.
The correlated case will again be dealt with using Cholesky decomposition.
Note that 
$$
Y_1(t) =  Y_1(0) \exp ( \int_0^t r_s ).
$$
The density of the integral of the CIR process can be derived from Pitman and Yor \cite{pitman:yor:1982}, wherein the authors give an expression of the Laplace transform of the integral of a Bessel process. 
This representation can in turn be used to derive the Laplace transform of the $Y_1(t)$, which upon inversion, gives its density, see e.g. Gulisashvili and Stein \cite{gulisashvili:stein:2010}. Next, since the assets $S_1(t),\dots, S_{m'}(t)$ are geometric brownian motions, their density can be computed explicitly, and we obtain the density of the vector $(Y_1,S_1,\ldots,S_{m'})$ by multiplying these densities together, thanks to the independence of the driving Brownian motions.

Furthermore,  to justify the existence of a density for the pair $(\int_0^1 r_s ds,W_0(1))$, we can once again cite \cite{gulisashvili:stein:2010}. This density is denoted by $q$ and   verifies for any bounded function $\Psi$:
$$
\mathbb{E}\big[\Psi(\int_{0}^1 r_s \text{d}s,W_0(1))\big] = \int \Psi(I,w_0) q(I,w_0) \text{d}I \text{d}w_0.
$$
Below, we denote by $q_1$ the marginal distribution of $\int_{0}^1 r_s \text{d}s$, whose density is bounded on $\mathbb{R}$ (see among others \textit{e.g.} \cite{pitman:yor:1982} and \cite{gulisashvili:stein:2010}).
To describe the other components, we introduce: $$c_k=\mu_k- \frac{\{\sigma_k\}^2}{2},$$ and using a conditional expectation argument, we have:
\begin{align*}
\mathbb{E}&[f(Y_1,S_1\ldots,S_{m'})] \\
&= 
\mathbb{E}\Big[f\big(Y_1(0) e^{\mathcal{I}},S_1(0) e^{c_1 + \sigma_1 ((\tilde{L} \tilde{W})_1+\ell_{0,1} W_0)},\ldots,S_{m'}(0) e^{c_{m'} + \sigma_{m'} ((\tilde{L} \tilde{W})_{m'}+\ell_{0,m'} W_0)} \big)\Big]\\
& = \int_{\mathbb{R}_+ \times \mathbb{R}} \int_{\mathbb{R}^{m-1}} q(I,w_0)   \gamma_{m'}(\tilde{w}) \\
&  f\Big(Y_1(0) e^{I},S_1(0) e^{c_1 + \sigma_1 \big((\tilde{L} \tilde{w})_1+\ell_{0,1} w_0\big) },S_{m'}(0) e^{c_{m'} + \sigma_{m'} \big((\tilde{L} \tilde{w})_{m'}+\ell_{0,m'} w_0\big)}\Big) \text{d}\tilde{w} \text{d}I \text{d}w_0,
\end{align*}
where $\gamma_{m-1}$ refers to the $m-1$ dimensional standard centered Gaussian distribution. For a fixed value of $w_0$, we consider the change of variables:
$$
z_1=Y_1(0) e^I \quad \text{and} \quad z_{i+1} = S_i(0) e^{c_i + \sigma_i(\tilde{L} \tilde{w})_i + \sigma_i \ell_{0,i} w_0},
$$
which is equivalent to $I(z_1)=\log(\frac{z_1}{Y_1(0)})$
and
$$ \tilde{w}(z_2,\ldots,z_m,w_0) = \tilde{L}^{-1} \begin{pmatrix}
 \{\sigma_1\}^{-1} \Big(\log\big( \frac{z_2}{S_1(0)}\big)-c_1-\ell_{0,1} w_0 \Big) \\ \vdots \\ 
 \{\sigma_{m'}\}^{-1} \Big( \log\big( \frac{z_m}{S_{m'}(0)}\big)-c_{m'}-\ell_{0,m'} w_0\Big)
\end{pmatrix}.
$$
Since $\tilde{L}$ is invertible, the Jacobian of the change of variable is then given by:
$$
\text{d}I \text{d} \tilde{w} = \frac{1}{(\sigma_1 \ldots \sigma_{m'}) \det \tilde{L}} \frac{\text{d} z_1 \text{d} z_2 \ldots d z_m}{ z_1 \ldots z_m}.
$$
We observe that $\det L = \det \tilde{L}$, which leads to:
\begin{align*}
\mathbb{E}&[f(Y_1,S_1\ldots,S_{m'})]=\\
&\int_{\{\mathbb{R}_+\}^m} f(z_1,\ldots,z_m) \frac{\int_{\mathbb{R}} q\Big(\log\big(\frac{z_1}{Y_1(0)}\big),w_0\Big) \gamma_{m-1}(\tilde{w}(z_2,\ldots,z_m,w_0))\text{d}w_0}{(\sigma_1 \ldots \sigma_{m'}) \det L \times (z_1 \ldots z_m)} \text{d}z_1 \ldots \text{d}z_m.
\end{align*}
The previous expression then guarantees that $(Y_1,S_1,\ldots,S_{m'})$ has a distribution uniformly continuous with respect to the Lebesgue measure on $\{\mathbb{R}_+\}^m$, whose density is of the form:
$$
(z_1,\ldots,z_m) \longmapsto \frac{M(z_1,\ldots,z_m)}{(\sigma_1 \ldots \sigma_{m'})  \det(L) z_1 \ldots z_m},
$$
where 
$$
M(z_1,\ldots,z_m) = 
\int_{\mathbb{R}} q\Big(\log\big(\frac{z_1}{Y_1(0)}\big),w_0\Big) \gamma_{m-1}(\tilde{w}(z_2,\ldots,z_m,w_0))\text{d}w_0.
$$
The Gaussian density $\gamma_{m-1}$ being continuous and bounded by $(2 \pi)^{-m'/2}$, we then have:
$$
M(z_1,\ldots,z_m) \leq 
\int_{\mathbb{R}} q\Big(\log\big(\frac{z_1}{Y_1(0)}\big),w_0\Big) \text{d}w_0 = q_1(I(z_1)) \leq \|q_1\|_{\infty} < + \infty.
$$
Therefore, $M$ is a bounded continuous function, whose bound is independent from  $m$ and $\Sigma$.
\end{proof}

Below, we will need to upper bound the probability of sliced events that are defined as follows. Consider $u=(u_1,v)$ a vector of the simplex $\Sm$ and $\epsilon>0$, we introduce:
$$
\Omega_{\epsilon}(\theta,u) = \big\{ \theta \leq \langle  Z,u\rangle \leq \theta+\epsilon\big\}.
$$
From Proposition \ref{prop:densite_bornee}, we deduce the following property.
\begin{proposition}[Sliced events]\label{prop:sliced_events}
Consider a portfolio $Z$ defined by Equation \eqref{eq:portfolio}. Assume 
$\sigma^+ = \sup_{ 1 \leq i \leq m} \sigma_i<+\infty$, and that the
 correlation matrix $\Sigma$ is invertible. Assume that the initialization is given by $Z_i(0)=1$.
A constant $K$ exists such that for any $\epsilon>0$, for any $\rho \in (0,1)$:
$$
\mathbb{P}[\Omega_{\varepsilon}(\theta,u)] \leq 
 K m e^{\frac{\{\sigma^+\}^2 m^2 (1-\rho)}{4\rho^{2}}} \epsilon^{\frac{1-\rho}{2}}, \qquad\forall \theta \in \mathbb{R}, \quad \forall u \in \Sm. 
$$
\end{proposition}

\begin{proof}
We consider $\eta>0$ small enough (whose dependency with $\epsilon$ and $m$ will be precised later on) and we observe that:
\begin{align}
\mathbb{P}[\Omega_{\epsilon}(\theta,u)] & = \mathbb{E} \big[ \un_{\Omega_{\epsilon}(\theta,u)}\big] \nonumber\\
& = \mathbb{E} \big[ \un_{\Omega_{\epsilon}(\theta,u)} \un_{\min Z_i \ge \eta} + \un_{\Omega_{\epsilon}(\theta,u)} \un_{\exists i \, :  Z_i \le \eta}  \big] \nonumber\\
& \leq \mathbb{E} \big[ \un_{\Omega_{\epsilon}(\theta,u)} \un_{\min Z_i \ge \eta}  \big]  + \sum_{i=1}^m \mathbb{P}[Z_i \le \eta].\label{eq:decomposition_event}
\end{align}
\underline{Step 1:}
Let us first consider the second term in Equation  \eqref{eq:decomposition_event}. We then consider two separate cases.
\begin{itemize}
\item If $i=1$ and considering $\eta \leq Z_1(0)$, we then observe that $Z_1 \ge Z_1(0)$ since $r$ is a non-negative process, so that $\mathbb{P}[Z_1 \le \eta] = 0$.
\item 
Assume now that $i'=i-1 \in \{1,\ldots,m'\}$,  using the notations of Proposition \ref{prop:densite_bornee}, we observe that:
$$
Z_{i'} \leq \eta \Longleftrightarrow Z_{i'}(0) e^{c_{i'} + \sigma_{i'} B_{i'}} \leq \eta \Longleftrightarrow B_{i'} \leq - \{\sigma_{i'}\}^{-1} \big(\log\frac{1}{\eta} +c_{i'}\big),
$$
so that
\begin{align*}
\mathbb{P}[Z_{i'} \leq \eta]& = \mathbb{P}\Big[B_{i'} \leq - \{\sigma_{i'}\}^{-1} \big(\log\frac{1}{\eta} +c_{i'}\big)\Big]\\
& = \mathbb{P}\Big[ \mathcal{N}(0,1) \ge \{\sigma_{i'}\}^{-1} \big(\log\frac{1}{\eta} +c_{i'}\big)\Big] \\
& \leq \frac{\sigma_{i'} e^{-\frac{\big(\log\frac{1}{\eta} +c_{i'}\big)^2}{2 \{\sigma_{i'}\}^2}}}{\sqrt{2\pi} \big(\log\frac{1}{\eta} +c_{i'}\big)},
\end{align*}
where the last line comes from standard estimation of the Gaussian tail. Choosing $\eta$ small enough, we then deduce that a constant $K$ exists such that:
\begin{equation}
\label{eq:part1}
\forall i \in \{2,\ldots,m\} \qquad 
\mathbb{P}[Z_i \leq \eta] \leq K e^{- \frac{(\log\frac{1}{\eta})^2}{4 \{\sigma^i\}^2}}.
\end{equation}
Defining $\sigma^+ = \sup_{ 1 \leq i \leq m} \sigma_i$, a union bound then leads to:
$$
 \sum_{i=1}^m \mathbb{P}[Z_{i} \le \eta] \leq K m  e^{- \frac{(\log\frac{1}{\eta})^2}{4 \{\sigma^+\}^2}} = K m (
 \eta^{\frac{\log \frac{1}{\eta}}{4 \{\sigma^+\}^2}}).
$$
\end{itemize} 
\noindent\underline{Step 2: } We study the first term of \eqref{eq:decomposition_event} and for any vector $q =(q_1,\ldots,q_m)$ and for any integer $\ell$, we denote by $q^{-\ell}$ the vector $q=(q_1,\ldots,q_{\ell-1},q_{\ell+1},\ldots,q_m)$.
 For any $u \in \Sm$, we can find $\ell$ such that $u_{\ell} \ge \frac{1}{m}$, so that:
\begin{align*}
\mathbb{E} [ \un_{\Omega_{\epsilon}(\theta,u)} \un_{\min Z_i \ge \eta}  ] & = 
\mathbb{E} [ \un_{\theta - \langle Z^{-\ell},u^{-\ell}\rangle \leq Z_\ell u_{\ell} \leq \theta +\epsilon- \langle Z^{-\ell},u^{-\ell}\rangle} \un_{\min Z_i \ge \eta}  ] \\
& \leq \mathbb{E} [ \un_{(u_\ell)^{-1}[\theta - \langle Z^{-\ell},u^{-\ell}\rangle ] \leq Z_\ell  \leq (u_\ell)^{-1}[\theta- \langle Z^{-\ell},u^{-\ell}\rangle] + (u_{\ell})^{-1}\epsilon} \un_{\min Z_i \ge \eta}  ].
\end{align*}
Now, we recall the joint density $p(z_1,\dots z_m)$ defined in  Proposition \ref{prop:densite_bornee} and rewrite the previous inequality as:
\begin{align*}
&\mathbb{E} \Big[ \un_{\Omega_{\epsilon}(\theta,u)} \un_{\min Z_i \ge \eta}  \Big]\\
& \le 
\int_{[\eta,+\infty[^m} \un_{(u_\ell)^{-1}[\theta - \langle z^{-\ell},u^{-\ell}\rangle ] \leq z_\ell  \leq (u_\ell)^{-1}[\theta- \langle z^{-\ell},u^{-\ell}\rangle] + (u_{\ell})^{-1}\epsilon} p(z_1, \dots, z_m) \text{d}z_1 \dots \text{d}z_m\\
&\le
\int_{[\eta,+\infty[^{m-1}} \bigg(  \int_\eta^\infty \un_{(u_\ell)^{-1}[\theta - \langle z^{-\ell},u^{-\ell}\rangle ] \leq z_\ell  \leq (u_\ell)^{-1}[\theta- \langle z^{-\ell},u^{-\ell}\rangle] + (u_{\ell})^{-1}\epsilon} p(z_1, \dots, z_m) \text{d}z_\ell\bigg). \\
\end{align*}
We   use the Cauchy-Schwarz inequality in order to bound the inner integral:
\begin{align*}
&\int_\eta^\infty \un_{(u_\ell)^{-1}[\theta - \langle z^{-\ell},u^{-\ell}\rangle ] \leq z_\ell  \leq (u_\ell)^{-1}[\theta- \langle z^{-\ell},u^{-\ell}\rangle] + (u_{\ell})^{-1}\epsilon} p(z_1, \dots, z_m) \text{d}z_\ell\\
&\le\sqrt{\int_\eta^\infty \un_{(u_\ell)^{-1}[\theta - \langle z^{-\ell},u^{-\ell}\rangle ] \leq z_\ell  \leq (u_\ell)^{-1}[\theta- \langle z^{-\ell},u^{-\ell}\rangle] + (u_{\ell})^{-1}\epsilon} \text{d}z_\ell}\sqrt{\int_\eta^\infty p(z_1, \dots, z_m)^2 \text{d}z_\ell}\\
&\le \sqrt{m\epsilon} \sqrt{\int_\eta^\infty p(z_1, \dots, z_m)^2 \text{d}z_\ell},
\end{align*}
where the last line derives from $u_\ell^{-1}\epsilon\le m\epsilon$.
The previous inequality combined with the Jensen inequality
$  \int |g| \text{d}\mu \le \sqrt{\int g^2 \text{d}\mu} $
and the Tonelli theorem yield:
\begin{align*}
\mathbb{E} [ \un_{\Omega_{\epsilon}(\theta,u)} \un_{\min Z_i \ge \eta}  ] \le &
\sqrt{m\epsilon} 
\int_{[\eta,+\infty[^{m-1}}\big( \sqrt{\int_\eta^\infty p(z_1, \dots, z_m)^2 \text{d}z_\ell}\big)\prod_{i\neq \ell}\text{d}z_i \\
&\le\sqrt{m\epsilon} \sqrt{
\int_{[\eta,+\infty[^{m}}  p(z_1, \dots, z_m)^2  \text{d}z_1 \dots \text{d}z_m}.\end{align*}
The expression of $p$ obtained in Proposition~\ref{prop:densite_bornee} yields:
\begin{align*}
\mathbb{E} [ \un_{\Omega_{\epsilon}(\theta,u)} \un_{\min Z_i \ge \eta}  ] \le &
 \sqrt{m\epsilon}\sqrt{
\int_{[\eta,+\infty[^{m}}   \frac{M(z_1,\ldots,z_m)^2}{(\sigma_1 \ldots \sigma_{m'})^2  \det(L)^2 (z_1 \ldots z_m)^2}, \text{d}z_1 \dots \text{d}z_m}\\
\le& 
\sqrt{m\epsilon}\sqrt{\frac{K}{\eta^m}
\int_{[\eta,+\infty[^{m}}   \frac{M(z_1,\ldots,z_m)}{(\sigma_1 \ldots \sigma_{m'})  \det(L) (z_1 \ldots z_m)}, \text{d}z_1 \dots \text{d}z_m},
\end{align*}
where we use the fact that we integrate over $[\eta,+\infty[$ and that the function $M$ is bounded. We finally obtain that for a positive constant $K$:
\begin{equation}\label{eq:part2}
\mathbb{E} [ \un_{\Omega_{\epsilon}(\theta,u)} \un_{\min Z_i \ge \eta}  ]\le \sqrt{m\epsilon}\frac{K}{\eta^{m/2}} .
\end{equation}

\noindent\underline{Step 3: Final bound}
We now choose $\eta$ to balance the size of Equation \eqref{eq:part1} and
Equation \eqref{eq:part2} in the decomposition \eqref{eq:decomposition_event}. Fix   $\rho \in (0,1)$, we choose $\eta$ such that $\eta=\epsilon^{\rho/m}$, so that:
$$
\eqref{eq:part2} \leq K \sqrt{m} \sqrt{\eta^{-m} \epsilon}  = K \sqrt{m} \epsilon^{\frac{1-\rho}{2}}.
$$

We also verify   with this choice of $\eta$ that:
$$
\eqref{eq:part1} \leq K m e^{-\frac{(\log \frac{1}{\eta})^2}{4 \{\sigma^+\}^2}} \leq K m e^{-\frac{m^2}{4 \{\sigma^+\}^2 \rho^2} ( \log (\epsilon))^2}.
$$ To balance both contributions, we aim at finding the smallest constant $C$ such that:
$$-\frac{m^2}{4 \{\sigma^+\}^2 \rho^2} ( \log (\epsilon))^2 \le \frac{1-\rho}{2}\log(\epsilon) +C,
$$
which is equivalent to 
$$\frac{m^2}{4 \{\sigma^+\}^2 \rho^2} ( \log (\epsilon))^2 + \frac{1-\rho}{2}\log(\epsilon)\ge -C.
$$
Now we consider the left hand side as a polynomial of degree $2$ in $\log(\epsilon)$ and obtain that the inequality is satisfied for any $\epsilon$ when  the constant $C$ is chosen as:
$$C=\frac{m^2\{\sigma^+\}^2 (1-\rho)}{4\rho^2}.$$
Finally we obtain that: 
$$
 \leq K m e^{\frac{\{\sigma^+\}^2 m^2 (1-\rho)}{4\rho^{2}}} \epsilon^{\frac{1-\rho}{2}}, 
$$
which concludes the proof.
\end{proof}

\subsection{Proof of Proposition \ref{prop:wasserstein}}

\begin{proof}
Let us first rewrite the Wasserstein distance as: 
$$\mathcal{W}_1(\mathcal{L}(\hat{Y}_1^{h}),\mathcal{L}(Y_1))= \sup_{f\in Lip_1}\{ \PE( f(\hat{Y}_1^{h}) -\PE(f(Y_1))\}.$$
Since the G.B.M. are exactly simulated and the test function $f$ is chosen among $1-$Lipschitz functions it is sufficient to bound 
$\PE[\lvert \hat{Y}_1^{h}- Y_1\lvert].$
Recall the definition of the approximation $\hat{Y}^{h}$ in Equation \eqref{eq:def_S1_Riemann}, we obtain 
\begin{align*}
\PE[ \lvert Y_1-\hat{Y}_1^{h} \lvert]&=\mathbb{E}\bigg[\Big|Y_0 \big( e^{\int_0^1 r_s ds}- e^{\frac{1}{N} \sum_{k=1}^{N} \hat{r}_{kh}}\big)\Big| \bigg]\\
=&\mathbb{E} \bigg[\Big|Y_1 \big( 1- e^{\Delta_h}\big)\Big|\bigg],
\end{align*}
where
\begin{equation}
\label{eq:delta_N}
\Delta_h = \int_0^1 r_s ds - \frac{1}{N} \sum_{k=1}^{N} \hat{r}_{kh}.
\end{equation}
Using $|1-e^{x}| \leq |x|e^{|x|}$
and a three way Holder's inequality,  we obtain that for $p,q$ and $r$ such that $\frac{1}{p}+\frac{1}{q}+\frac{1}{r}=1$ then:
$$
\PE( \lvert \hat{Y}_1^{h}- Y_1\lvert)
\le \PE\Big(|Y_1|^r\Big)^{1/r} \PE \Big( |\Delta_h|^p \Big)^{1/p} \PE\Big(e^{q\Delta_h} \Big)^{1/q}.
$$
Using Lemma \ref{lem:moment_delta_fini} and \ref{lem:exp_delta}, we deduce   that for any choice of $r, q$ such that $\frac{1}{q}+\frac{1}{r}=1/2$, we have:
\begin{equation}
\label{eq:maj_err_discretization}
\PE[ \lvert \hat{Y}_1^{h}- Y_1\lvert]
\le C_q\PE[|Y_1|^r]^{1/r} \sqrt{h},
\end{equation}
for $C_q$ large enough (independent from $h$), which concludes the proof.
\end{proof}

\subsection{Upper bound of the bias term}

\begin{proof}[Proof of Proposition \ref{prop:biais}]
 Below, we alleviate the notations and denote by $S_i =S_i(1)$, which is the return of asset $i$ at time $1$. From the exact simulation of the G.B.M. $\Sb=(S_1,\ldots,S_{m'})$, we have that $Z=(Y,\Sb)$ whereas
$\hat{Z}=\hat{Z}^{h} = (\hat{Y}^{h},\Sb)$ where $\hat{Y}^{h}$ is a biased approximation of $Y_1$ defined in \eqref{eq:def_S1_Riemann}.
We denote by $w=(u,v)\in\R_+\times(\R_+)^m$ the repartition of the portfolio and recall that $ \norm{w}{1}=u+\norm{v}{1}=1$.
The goal of this section is to control the bias in terms of a function of the discretization parameter $h$:
$$
\mathcal{B}=\norm{\mathbb{E}[\langle Z,w \rangle\un_{\langle Z,w \rangle \ge \theta} - \langle \hat{Z},w \rangle \un_{\langle \hat{Z},w \rangle  \ge \theta}]}{2}=
\norm{\mathbb{E}[\langle Z\un_{\langle Z,w \rangle \ge \theta}- \hat{Z}\un_{\langle \hat{Z},w \rangle  \ge \theta},w \rangle ] }{2}.
$$
From the re-investment condition, we know that $\norm{w}{1}=1$ and the Cauchy-Schwarz inequality yields:
$$
\mathcal{B} \le  \sqrt{m} \mathbb{E}[  \norm{ Z \un_{\langle Z,w \rangle \ge \theta}- \hat{Z}\un_{\langle \hat{Z},w \rangle  \ge \theta} }{2}].
$$
Since we discretize only the first component, we have that:
$$
Z = (Y_0 e^{\int_0^1 r_s ds}, S_1, \dots, S_{m'}), \mbox{ and } \hat{Z} = (Y_0 e^{\frac{1}{N} \sum_{k=1}^N \hat{r}_{kh }}, S_1, \dots, S_{m'}),
$$
We can rewrite the difference as:
\begin{align*}
&\norm{ Z \un_{\langle Z,w \rangle \ge \theta}- \hat{Z}\un_{\langle \hat{Z},w \rangle  \ge \theta} }{2}\\
& \qquad\qquad\le \norm{Y\un_{\langle Z,w \rangle \ge \theta}- \hat{Y}\un_{\langle \hat{Z},w \rangle \ge \theta }}{2}+\norm{\Sb(\un_{\langle Z,w \rangle \ge \theta}- \un_{\langle \hat{Z},w \rangle \ge \theta}  ) }{2}\\
&\qquad\qquad\le  \norm{ Y (  \un_{ \langle Z,w \rangle \ge \theta}- \un_{\langle \hat{Z},w \rangle \ge \theta} )  }{2}
+\norm{(Y-\hat{Y})\un_{\langle \hat{Z},w \rangle \ge \theta} }{2}\\
&\qquad\qquad\qquad +\norm{\Sb(\un_{\langle Z,w \rangle \ge \theta}- \un_{\langle \hat{Z},w \rangle \ge \theta} ) }{2}\\
&\qquad\qquad\le \vert Y-\hat{Y} \vert + (\vert Y\vert + \norm{\Sb}{2}  )  \vert \un_{\langle Z,w \rangle \ge \theta}- \un_{\langle \hat{Z},w \rangle \ge \theta}\vert .
\end{align*}
Taking the expectation and applying the Cauchy-Schwarz inequality in the second term we obtain that $\mathcal{B}$ can be bounded by:
\begin{equation}
\label{eq:maj_B}
\mathcal{B}  \le \sqrt{m}\big(\PE[ (\vert Y\vert + \norm{\Sb}{2} )^2 ]\big) ^{1/2}
\Big(\PE\big[|\un_{\langle Z,w \rangle \ge \theta}- \un_{\langle \hat{Z},w \rangle  \ge \theta}|\big] \Big)^{1/2}  +\sqrt{m} \mathbb{E} \big[ \|Y-\hat{Y}\| \big] .
\end{equation}
We apply Proposition \ref{prop:wasserstein}
and deduce that the second term is  of the order $\sqrt{h}$. We then focus on the difference of indicator functions.
Let us remark that:
$$ \langle Z,w \rangle=Yu+\langle \Sb,v\rangle \quad \text{and} \quad  \langle \hat{Z},w \rangle=\hat{Y} u+\langle \Sb,v\rangle,
$$ thus 
$$
\PE\Big( \Big|\un_{\langle Z,w \rangle \ge \theta}- \un_{\langle \hat{Z},w \rangle  \ge \theta}\Big| \Big)=
  \Big|\PE\Big(\un_{ \langle \Sb,v\rangle \ge \theta - Yu}- \un_{\langle \Sb,v\rangle   \ge \theta - \hat{Y}u } \Big) \Big| 
 = \PP(A_1) + \PP(A_2),
$$
where 
$$
A_1= \{  \theta \le \langle Z,w \rangle \le \theta + u(Y-\hat{Y}) \}
\mbox{ and  }
 A_2= \{  \theta+ u(Y-\hat{Y}) \le \langle Z,w\rangle \le \theta  \}.
$$
These two events are handled similarly, we write:
$$
Y-\hat{Y}=Y_1(1-e^{\Delta_h}),
$$
where $\Delta_h $ is given by \eqref{eq:delta_N}.
Let us fix $\epsilon>0$ and consider the event $\Omega_\varepsilon=\{ u|Y-\hat{Y} |\le \varepsilon\}$, then
$$
\PP(A_1) =\PP(A_1\cap \Omega_\varepsilon)+\PP(A_1\cap \Omega^c_\varepsilon),
$$
and now, observe that:
$$
A_1\cap \Omega_\varepsilon \subset\{ \theta \le \langle Z,w\rangle \le \theta + \varepsilon\},
$$
which lead to control the probability the the portfolio $Z$ belongs to a slice. Consider $w=(u,v)$ a vector of the simplex $\Db_{m}$ and $\epsilon>0$, we introduce:
$$
\Omega_{\epsilon}(\theta,w) = \big\{ \theta \leq \langle Z, w\rangle \leq \theta+\epsilon\big\}.
$$
An upper bound for this probability was obtained in Proposition \ref{prop:sliced_events}.
A constant $K$ exists such that for any $\epsilon>0$, for any $\rho \in (0,1)$:
$$
\mathbb{P}[\Omega_{\varepsilon}(\theta,u)] \leq 
 K m e^{\frac{\{\sigma^+\}^2 m^2 (1-\rho)}{4\rho^{2}}} \epsilon^{\frac{1-\rho}{2}}, \qquad\forall \theta \in \mathbb{R}, \quad \forall \rho \in (0,1) \quad \forall u \in \Db_{m'}. 
$$
%\end{proposition}
We finally study $A_1\cap \Omega_\varepsilon^c$ and  write that for any $\zeta>0$, and using $|1-e^{x}| \leq |x|e^{|x|}$:
\begin{align*}
\PP(A_1\cap \Omega_\varepsilon^c)\le \PP[ uY (1-e^{\hat{h}}) \ge \varepsilon] &\leq 
\PP[Y  (1-e^{\Delta_h}) \ge \varepsilon] \nonumber\\
& \leq \PP[Y  \ge \zeta] 
+ \PP[ |1-e^{\Delta_h}| \ge  \zeta^{-1} \varepsilon] \nonumber\\
& \leq \PP[ Y  \ge \zeta] 
+ \PP[ |\Delta_h| e^{|\Delta_h|} \ge \zeta^{-1} \varepsilon] \nonumber\\
& \leq \frac{ \PE[Y ]}{\zeta}
+ \zeta \frac{\PE[ |\Delta_h| e^{|\Delta_h|}]}{  \varepsilon} \nonumber\\
& \leq \frac{ \PE[Y^{r} ]}{\zeta^{r}}
+ \zeta \frac{\Big(\PE[ \Delta_h^2]\Big)^{1/2} \Big(\PE[ e^{2 |\Delta_h|}]\Big)^{1/2}}{  \varepsilon},
\end{align*}
where we applied first a union bound and then the Markov and the Cauchy-Schwarz inequalities.
We finally obtain from Lemma \ref{lem:exp_delta} and \ref{lem:moment_delta_fini} that:
\begin{equation}
\PP(A_1\cap \Omega_\varepsilon^c)\le\frac{ \PE[Y^{r} ]}{\zeta^{r}} + C \zeta \frac{\sqrt{h}}{\epsilon} .
\label{eq:inegalite_zeta_Delta}
\end{equation}
Combining Proposition \ref{prop:sliced_events} and \eqref{eq:inegalite_zeta_Delta}, we deduce that: 
\begin{align*}
\PE\Big[ \big|\un_{\langle Z,w \rangle \ge \theta}- \un_{\langle \hat{Z},w \rangle  \ge \theta}\big| \Big] &\le C \big( K m e^{\frac{\{\sigma^+\}^2 m^2 (1-\rho)}{4\rho^{2}}} \epsilon^{\frac{1-\rho}{2}} +\frac{ \PE[Y^{r} ]}{\zeta^{r}} + C \zeta \frac{\sqrt{h}}{\epsilon} \big).
\end{align*}
Now using this bound as well as \eqref{eq:maj_err_discretization} in \eqref{eq:maj_B} we obtain that:
\begin{equation}\label{eq:maj_B_final}
\mathcal{B} \le  K \sqrt{m} e^{\frac{\{\sigma^+\}^2 m^2 (1-\rho)}{4\rho^{2}}}\big( \epsilon^{\frac{1-\rho}{2}} +\frac{ 1}{\zeta^{r}} + \zeta \frac{\sqrt{h}}{\epsilon}  +\sqrt{h}\big),
\end{equation}
where the parameters $r>1$, $\varepsilon>0$, $\zeta>0$ and $\rho \in(0,1)$ have to be chosen. Note that the constant depends heavily on the number of assets $m$ and on the choice of $\rho$. We optimize only on the two parameters $\varepsilon >0$ and $\zeta>0$ the quantity:
$$
\epsilon^{\frac{1-\rho}{2}} +\frac{ 1}{\zeta^{r}} + \zeta \frac{\sqrt{h}}{\epsilon}  + \sqrt{h}.
$$
It is immediate to verify that $\zeta$ should be chosen such that
$$
\zeta^{-r} = \zeta \frac{\sqrt{h}}{\varepsilon} \Longrightarrow \zeta=\frac{\varepsilon^{\frac{1}{1+r}}}{\sqrt{h}^\frac{1}{1+r}}.
$$
In the meantime, the optimal value of $\varepsilon$ satisfies:
$$
\varepsilon^{\frac{1-\rho}{2}} = \zeta \frac{\sqrt{h}}{\varepsilon} \Longrightarrow\varepsilon^{\frac{3-\rho}{2}} = \varepsilon^{\frac{1}{1+r}} \sqrt{h}^{1-\frac{1}{1+r}}.
$$
Some straightforward computations show that
$\varepsilon = h^{\frac{2r}{1-\rho+r(3-\rho)}}$, which in turn implies that  a constant $K$ large enough exists such that:
$$
\mathcal{B} \leq K_{r,\rho} \sqrt{m} e^{\frac{\{\sigma^+\}^2 m^2 (1-\rho)}{4\rho^{2}}} 
h^{\frac{1}{2 ( \frac{3-\rho}{1-\rho}+\frac{1}{r})}}.
$$
It is then easy to see that the rate may be arbitrarily close to $h^{1/6}$ by picking $r$ large enough and $\rho$ close to $0$.
More precisely, for any $\ee>0$, we can find $\rho$ and $r$ such that:
$$
\mathcal{B} \leq K_{\ee} \sqrt{m} e^{\frac{\{\sigma^+\}^2 m^2 }{4\ee^2}} 
h^{\frac{1}{6}-\ee}.
$$
\end{proof}

%%%%%%%%%%%%%If your paper has an appendix, please use there the same font as for the main body of the paper. This is achieved by using \textbackslash appendix \textbackslash normalsize.

%\begin{acknowledgements}
%\end{acknowledgements}



\begin{thebibliography}{10}

\bibitem{alfonsi:2005}
A.~Alfonsi.
\newblock On the discretization schemes for the cir (and bessel squared)
  processes.
\newblock {\em Monte Carlo Methods and Applications}, 11:355--384, 2005.

\bibitem{alfonsi:2010}
A.~Alfonsi.
\newblock High order discretization schemes for the cir process: application to
  affine term structure and heston models.
\newblock {\em Mathematics of Computation}, 79(269):209--237, 2010.

\bibitem{alfonsi:2013}
A.~Alfonsi.
\newblock Strong order one convergence of a drift implicit euler scheme:
  Application to the cir process.
\newblock {\em Statistics \& Probability Letters}, 83(2):602--607, 2013.

\bibitem{alfonsi:book}
A.~Alfonsi.
\newblock {\em Affine diffusions and related processes: simulation, theory and
  applications}, volume~6.
\newblock Springer, 2015.

\bibitem{Artzner97}
P.~Artzner, F.~Delbaen, J.M. Eber, and D.~Heath.
\newblock Thinking coherently.
\newblock {\em Risk}, 10, 1997.

\bibitem{Artzner99}
P.~Artzner, F.~Delbaen, J.M. Eber, and D.~Heath.
\newblock Coherent measures of risk.
\newblock {\em Mathematical Finance}, 9:203--228, 1999.

\bibitem{atchade:fort:moulines:2017}
Y.~Atchad{\'e}, G.~Fort, and E.~Moulines.
\newblock On perturbed proximal gradient algorithms.
\newblock {\em The Journal of Machine Learning Research}, 18(1):310--342, 2017.

\bibitem{bally:talay:1996}
V.~Bally and D.~Talay.
\newblock The law of the {E}uler scheme for stochastic differential equations.
\newblock {\em Probability theory and related fields}, 104(1):43--60, 1996.

\bibitem{bardou:etal:book}
O.~Bardou, N.~Frikha, and G.~Pag{\`e}s.
\newblock Recursive computation of value-at-risk and conditional value-at-risk
  using mc and qmc.
\newblock In Pierre L'~Ecuyer and Art~B. Owen, editors, {\em Monte Carlo and
  Quasi-Monte Carlo Methods 2008}, pages 193--208, Berlin, Heidelberg, 2009.
  Springer Berlin Heidelberg.

\bibitem{bardou2016cvar}
O.~Bardou, N.~Frikha, and G.~Pag{\`e}s.
\newblock Cvar hedging using quantization-based stochastic approximation
  algorithm.
\newblock {\em Mathematical Finance}, 26(1):184--229, 2016.

\bibitem{bardou2009computing}
O.~Bardou, N.~Frikha, and G.~Pagès.
\newblock Computing var and cvar using stochastic approximation and adaptive
  unconstrained importance sampling.
\newblock {\em Monte Carlo Methods and Applications}, 15(3):173--210, 2009.

\bibitem{Beck:2003}
Amir Beck and Marc Teboulle.
\newblock Mirror descent and nonlinear projected subgradient methods for convex
  optimization.
\newblock {\em Operations Research Letters}, 31(3):167--175, 2003.

\bibitem{Teboulle:1986}
A.~Ben~Tal and M.~Teboulle.
\newblock Expected utility, penalty functions and duality in stochastic
  nonlinear programming.
\newblock {\em Manage. Sci.}, 32:1445--1466, 1986.

\bibitem{benaim_LN}
M.~Bena{\"\i}m.
\newblock Dynamics of stochastic approximation algorithms.
\newblock {\em S\'eminaire de probabilit\'es de Strasbourg}, 33:1--68, 1999.

\bibitem{benaim1996asymptotic}
M.~Bena{\"\i}m and M.~Hirsch.
\newblock Asymptotic pseudotrajectories and chain recurrent flows, with
  applications.
\newblock {\em Journal of Dynamics and Differential Equations}, 8(1):141--176,
  1996.

\bibitem{Bubeck:2015}
S.~Bubeck.
\newblock Convex optimization: Algorithms and complexity.
\newblock {\em Foundations and Trends in Machine Learning}, (8), 2015.

\bibitem{Shao}
L.~Chen and Q.-M. Shao.
\newblock Stein's method for normal approximation.
\newblock In {\em An introduction to {S}tein's method}, volume~4 of {\em Lect.
  Notes Ser. Inst. Math. Sci. Natl. Univ. Singap.}, pages 1--59. Singapore
  Univ. Press, Singapore, 2005.

\bibitem{condat}
L.~Condat.
\newblock Fast projection onto the simplex and the l1 ball.
\newblock {\em Mathematical Programming Series A}, 158:575--585, 2016.

\bibitem{CIR}
J.C. Cox, J.E. Ingersoll, and S.A. Ross.
\newblock A theory of the term structure of interest rates.
\newblock {\em Econometrica}, 53:385--407, 1985.

\bibitem{cozma}
A.~Cozma and C.~Reisinger.
\newblock Exponential integrability properties of {E}uler dicretization schemes
  for the cox-ingersoll-ross process.
\newblock {\em Discrete and continuous dynamical systems, series B},
  21(10):3359--3377, 2016.

\bibitem{Dereich_neuenkirch2012}
S.~Dereich, A.~Neuenkirch, and L.~Szpruch.
\newblock An {E}uler-type method for the strong approximation of the cox
  ingersoll ross process.
\newblock {\em Proceedings of the Royal Society A: Mathematical, Physical and
  Engineering Sciences}, 468(2140):1105--1115, 2012.

\bibitem{singer}
J.~Duchi, S.~Shalev-Shwartz, Y.~Singer, and T.~Chandra.
\newblock Efficient projections onto the l 1 -ball for learning in high
  dimensions.
\newblock {\em Int. Conf. on Machine learning (ICML)}, 2008.

\bibitem{frikha:huang:2015}
N.~Frikha and L.~Huang.
\newblock A multi-step richardson--romberg extrapolation method for stochastic
  approximation.
\newblock {\em Stochastic Processes and their Applications},
  125(11):4066--4101, 2015.

\bibitem{Glasserman}
P.~Glasserman.
\newblock {\em Monte Carlo Methods in Financial Engineering}.
\newblock Stochastic Modelling and Applied Probability. Springer, 2003.

\bibitem{these}
A.~Going-Jaeschke.
\newblock Parameter estimation and bessel processes in financial models and
  numerical analysis in hamiltonian dynamics.
\newblock {\em PhD dissertation of Swiss Federal Institutde of Technology of
  Zurich}, 1998.

\bibitem{gulisashvili:stein:2010}
A.~Gulisashvili and E.~Stein.
\newblock Asymptotic behavior of the stock price distribution density and
  implied volatility in stochastic volatility models.
\newblock {\em Applied Mathematics and Optimization}, 61(3):287--315, 2010.

\bibitem{Kiwiel}
K.C. Kiwiel.
\newblock Breakpoint searching algorithms for the continuous quadratic knapsack
  problem.
\newblock {\em Math. Program., Ser. A}, 112:473--491, 2008.

\bibitem{kloeden:platen:1999}
P.~Kloeden and E.~Platen.
\newblock {\em Numerical solution of SDE}.
\newblock Springer Science, 1999.

\bibitem{Krokhmal:2001}
P.~Krokhmal, J.~Palmquist, and S.~Uryasev.
\newblock {Portfolio optimization with conditional value-at-risk objective and
  constraints.}
\newblock {\em Journal of Risk}, 4(2):43--68, 2001.

\bibitem{Lan:2012}
G.~Lan, A.~Nemirovskij, and A.~Shapiro.
\newblock Validation analysis of mirror descent stochastic approximation
  method.
\newblock {\em Math. Program., Ser. A}, (134):425--458, 2012.

\bibitem{laruelle:pages:2012}
Sophie Laruelle and Gilles Pagès.
\newblock Stochastic approximation with averaging innovation applied to
  finance.
\newblock {\em Monte Carlo Methods and Applications}, 18(1):1--51, 2012.

\bibitem{blondel}
A.~Fujino M.~Blondel and N.~Ueda.
\newblock Large-scale multiclass support vector machine training via euclidean
  projection onto the simplex.
\newblock {\em International Conference on Pattern Recognition}, 2014.

\bibitem{markowitz:1952}
H.M. Markowitz.
\newblock Portfolio selection.
\newblock {\em Journal of Finance}, 7:77–91, 1952.

\bibitem{Mertikopoulos}
P.~Mertikopoulos and M.~Staudigl.
\newblock On the convergence of gradient-like flows with noisy gradient input.
\newblock {\em SIAM Journal on Optimization}, 28(1):163--197, 2018.

\bibitem{Nemirovski:1983}
A.~Nemirovskij and D.~Yudin.
\newblock {\em Problem complexity and method efficiency in optimization}.
\newblock Wiley-Interscience, 1983.

\bibitem{nesterov:2009}
Y.~Nesterov.
\newblock {Primal-dual subgradient methods for convex problems}.
\newblock {\em Mathematical Programming, Serie B}, 120:221--259, 2009.

\bibitem{data}
[DATA]~Python package panda$\_$datareader.
\newblock get$\_$data$\_$yahoo("nvda", "tsla", "aapl", "csco", "t", "xom",
  "twtr", "jpm", "bp").
\newblock start="2016-01-01", end="2018-01-01".

\bibitem{Pflug}
G.Ch. Pflug.
\newblock Some remarks on the value-at-risk and the conditional value-at-risk.
\newblock In {\em Probabilistic Constrained Optimization: Methodology and
  Applications, Ed. S. Uryasev}. Kluwer Academic Publishers, 2000.

\bibitem{pitman:yor:1982}
M.~Pitman, J .~Yor.
\newblock A decomposition of bessel bridges.
\newblock {\em Zeitschrift f{\"u}r Wahrscheinlichkeitstheorie und verwandte
  Gebiete}, 59(4):425--457, 1982.

\bibitem{RobbinsMonro1951}
H.~Robbins and S.~Monro.
\newblock A stochastic approximation method.
\newblock {\em The Annals of Mathematical Statistics}, 22:400--407, 1951.

\bibitem{RobbinsSiegmund1971}
H.~Robbins and D.~Siegmund.
\newblock A convergence theorem for non negative almost supermartingales and
  some applications.
\newblock {\em Optimizing methods in stat.}, pages 233--257, 1971.

\bibitem{rockafellar:royset:2014}
R.~Rockafellar and J.~Royset.
\newblock Random variables, monotone relations, and convex analysis.
\newblock {\em Mathematical Programming}, 148(1):297--331, 2014.

\bibitem{Rockafellar2000}
R.~T. Rockafellar and S.~Uryasev.
\newblock Optimization of conditional value-at-risk.
\newblock {\em The Journal of Risk}, 2(3):21--41, 2000.

\bibitem{Rockafellar2002}
R.~T. Rockafellar and S.~Uryasev.
\newblock Conditional value-at-risk for general loss distributions.
\newblock {\em Journal of Banking and Finance}, 26(7):1443--1471, 2002.

\bibitem{Sharpe}
W.~F. Sharpe.
\newblock Mutual fund performance.
\newblock {\em Journal of Business}, 39:119--138, 1966.

\bibitem{talay:tubaro:1990}
D.~Talay and L.~Tubaro.
\newblock Expansion of the global error for numerical schemes solving
  stochastic differential equations.
\newblock {\em Stochastic analysis and applications}, 8(4):483--509, 1990.

\bibitem{zhou:etal:2017}
Z.~Zhou, P.~Mertikopoulos, N.~Bambos, S.~Boyd, and P.~Glynn.
\newblock Stochastic mirror descent in variationally coherent optimization
  problems.
\newblock {\em Advances in Neural Information Processing Systems}, 30, 2017.

\end{thebibliography}


\end{document}